\documentclass[11pt]{amsart}
\usepackage{amsfonts}
\usepackage{latexsym}
\usepackage{amssymb}
\usepackage{amsmath}
\usepackage{xcolor}
\usepackage{bbm}
\usepackage{tikz}
\usepackage{enumerate}
\usepackage{mathrsfs}
\usepackage{todonotes}
\usepackage{relsize}
\usepackage{xfrac}

\usepackage{verbatim} %(für comment-Umgebung)
\usepackage{enumitem} % (für Aufzählung mit Buchstaben)

\usepackage{csquotes} % TK: Für richtige Anführungszeichen

\usepackage[left=2.5cm, top=2.5cm,bottom=2.5cm,right=2.5cm]{geometry}

\newcommand{\F}{\mathbb F}

\newcommand{\R}{\mathbb R}
\newcommand{\N}{\mathbb N}

\newcommand{\E}{\mathbb E}
\newcommand{\Pro}{\mathbb P}

\newcommand{\vol}{\mathrm{vol}}
\newcommand{\Mat}{\mathscr{M}}

\def\dint{\textup{d}}
\newcommand{\SSS}{\ensuremath{{\mathbb S}}}

\newcommand{\B}{\ensuremath{{\mathbb B}}}

\DeclareMathOperator{\Tr}{Tr}

\renewcommand{\Re}{\operatorname{Re}}  %Realteil
\newcommand{\ssbp}{\scriptscriptstyle\boxplus} %Kleines Boxplus

\newtheorem{thm}{Theorem}[section]
\newtheorem{cor}[thm]{Corollary}
\newtheorem{lemma}[thm]{Lemma}
\newtheorem{df}[thm]{Definition}
\newtheorem{proposition}[thm]{Proposition}

\newtheorem{example}[thm]{Example}

{
\theoremstyle{definition}

\newtheorem{rmk}[thm]{Remark}

}
\def\bB{\mathbf{B}}
\def\bC{\mathbf{C}}

\def\bE{\mathbf{E}}

\def\bG{\mathbf{G}}
\def\bM{\mathbf{M}}
\def\bN{\mathbf{N}}
\def\bP{\mathbf{P}}

\def\bU{\mathbf{U}}

\def\bW{\mathbf{W}}

\definecolor{kgreen}{rgb}{0, 0.42, 0}

\usepackage{nomencl}
\makenomenclature
\begin{document}

%%%%%%%%%%%%%%%%%%%%%%%%%%%%%%%%%%%%%%%%%%%%%5

\title[Weighted $p$-radial Distributions on Euclidean and Matrix $p$-balls]{Weighted $p$-radial Distributions on Euclidean\\ and Matrix $p$-balls with Applications to Large Deviations}

\author[Tom Kaufmann]{Tom Kaufmann}
\address{Tom Kaufmann: Faculty of Mathematics, Ruhr University Bochum, Germany} \email{tom.kaufmann@rub.de}

\author[Christoph Th\"ale]{Christoph Th\"ale}
\address{Christoph Th\"ale: Faculty of Mathematics, Ruhr University Bochum, Germany} \email{christoph.thaele@rub.de}


%\thanks{}

%\date{\today}

\begin{abstract}
A probabilistic representation for a class of weighted $p$-radial distributions, based on mixtures of a weighted cone probability measure and a weighted uniform distribution on the Euclidean $\ell_p^n$-ball, is derived. Large deviation principles for the empirical measure of the coordinates of random vectors on the $\ell_p^n$-ball with distribution from this weighted measure class are discussed. The class of $p$-radial distributions is extended to $p$-balls in classical matrix spaces, both for self-adjoint and non-self-adjoint matrices. The eigenvalue distribution of a self-adjoint random matrix, chosen in the matrix $p$-ball according to such a distribution, is determined. Similarly, the singular value distribution is identified in the non-self-adjoint case. Again, large deviation principles for the empirical spectral measures for the eigenvalues and the singular values are presented as an application.
\end{abstract}

\maketitle

%\text{}\\
%\\
%\\
%\\
%\tableofcontents

%\newpage

\section{Introduction}\label{sec:Introduction}

In $n$-dimensional Euclidean space there is a one-to-one correspondence between norms and symmetric convex bodies. Any given norm $\| \cdot \|$ on $\R^n$ defines a symmetric convex body in the form of its unit ball 
	$$\B_{\| \cdot \|} := \{ x \in \R^n: \| x \| \le 1 \},$$
	and, vice versa, a symmetric convex body $K \subset \R^n$ induces a norm $\| \cdot \|_K$ on $\R^n$ via the Minkowski functional
	$$\|x\|_K := \inf \left\{r \in [0,\infty): x \in r K \right\}, \qquad x \in \R^n,$$
	with respect to which $K$ itself is the unit ball $\B_{\| \cdot\|_K}$. This shows how the study of norms (or normed spaces) and symmetric convex bodies are closely related.	The study of convex bodies in high dimensions, known today as asymptotic geometric analysis, has arisen from the local theory of Banach spaces, which aimed at analyzing infinite-dimensional normed spaces via their local substructures, such as their unit balls. Given an infinite-dimensional Banach space, structures like its unit ball are naturally of infinte dimension as well, and since working in infinite dimensions is inherently more difficult than working in the finite-dimensional setting, it is a fuitful approach to instead study the finite-dimensional counterparts of such structures asymptotically in the limit of the dimension. This was the motivating impulse giving rise to the field of asymptotic geometric analysis and has yielded a number of highly relevant results, such as solutions to Banachs' hyperplane problem \cite{Gowers} or the unconditional basic sequence problem \cite{GowersMaurey} (also see \cite{MilmanTJ} for a broader context on these results). \\
Despite having its origin in the realm of functional analysis, the field has since established itself in its own right, also considering problems beyond the study of symmetric convex bodies that occur naturally as the unit balls of Banach spaces. High-dimensional convexity furthermore has a large number of applications, e.g. in signal processing, such as compressed sensing (see \cite{chafai2012interactions, Foucart}) and sparse signal recovery (see \cite[Chapter 10]{Vershynin2018}), or random information and approximation theory (see e.g. \cite{hinrichs2021random, hinrichs2021random2, hinrichs2019curse, krieg2020random}).\\
\\
In high dimensions convex bodies exhibit certain regularities, such as volume concentration phenomena (see, e.g., \cite{GuedonConcentrationPhenomena}), which make it highly useful to approach them from a probabilistic perspective. As pointed out in \cite{BookAGA}, it might seem counter-intuitive to analyze something exhibiting regularities from a probabilistic perspective, as probability concerns itself with studying the nature of irregularity, i.e., randomness, of given quantities. But as with well-known limit theorems from probability, such as the law of large numbers and the central limit theorem, with large sample sizes (and analogously -- with high dimensionality) random objects exhibit interesting patterns well characterized in the language of probability and vice versa. Thus, one can view asymptotic geometric analysis as being located somewhat at the intersection between geometry, functional analysis, and probability theory. Several analogues of probabilistic results have been found in high dimensional convex geometry, the central limit theorem being the most notable example (see, e.g., Anttila, Ball and Perissinaki \cite{ABPclt}, Klartag \cite{KlartagCLT, KlartagCLT2}). In fact, many of these results have been extended beyond the realm of distributions on convex bodies to isotropic log-concave measures. This extension allows to give functional versions of classic geometric identities and prove results from probability via a geometric approach, hence giving it the name \enquote{\textit{Geometrization of Probability}} (see \cite{ProceedingsAGA, MilmanGeoProb}). Thus, asymptotic geometric analysis is not merely relevant to understanding high-dimensional geometric objects and normed spaces, but is also a promising area of study for broader probability theory. For background information on asymptotic geometric analysis, we refer the reader to the surveys and monographs \cite{BookAGA, BookGICB, GuedonConcentrationPhenomena, GuedonConcentrationIneq}.

The study of $\ell_p^n$-balls has long been a prominent area of research in high dimensional convex geometry and the local theory of Banach spaces, as they are the unit balls in the finite-dimensional sequence spaces $\ell_p^n$. For random vectors in $\ell_p^n$-balls in high-dimensional Euclidean space many properties, such as concentration phenomena and projection behaviors, are known. We refer to the survey by Prochno, Th\"ale and Turchi \cite{PTTSurvey} for a comprehensive overview of old and more recent results. Let us denote by $\bU_{n,p}$ the uniform distribution on the Euclidean $\ell_p^n$-ball $\B_p^n$ and by $\bC_{n,p}$ the cone probability measure on the $\ell_p^n$-sphere $\SSS_p^{n-1}$. For the Euclidean sphere $\SSS_2^{n-1}$ the Poincar\'e-Maxwell-Borel lemma states that the joint distribution of any fixed number $k$ of coordinates of a random vector with distribution $\bC_{n,2}$ is approximately standard Gaussian (see \cite{DiaconisFreedman}). This was furthered to $\SSS_p^{n-1}$ for any $p\in[1,\infty]$ by Rachev and R\"uschendorf \cite{RachevRueschendorf} and Naor and Romik \cite{NR2003}. Moreover, an extension of this was given by Johnston and Prochno \cite{JohnstonProchnoMaxwell} for generalized Orlicz-balls, whose $k$-marginals are, however, not given by generalized Gaussians. %however with non-Gaussian marginals.
%with the marginals given either by uniform distributions or via certain Gibbs densities, depending on the radius of the Orlicz-ball.} 
Rachev and R\"uschendorf \cite{RachevRueschendorf} and Schechtman and Zinn \cite{SchechtmanZinn} also provided a probabilistic representation for random vectors with distributions $\bU_{n,p}$ and $\bC_{n,p}$. For $p\in(0,\infty]$ this was generalized by Barthe, Gu\'edon, Mendelson and Naor \cite{BartheGuedonEtAl}, who gave a probabilistic representation for a class of mixtures of $\bC_{n,p}$ and $\bU_{n,p}$. For a Borel probability measure $\bW$ on $[0,\infty)$ they defined the class of distributions $\bP_{n,p,\bW} := \bW(\{0\})\bC_{n,p}+\Psi\bU_{n,p}$ on $\B_p^n$, where $\Psi$ is an appropriate $p$-radial density that depends on $\bW$, and provided a convenient representation of $\bP_{n,p,\bW}$ via a random vector of generalized Gaussians. The choice of $\bW$ determines how exactly the cone probability measure and the uniform distribution get mixed. This class of measures and its corresponding representations have gained considerable interest in asymptotic and convex geometric analysis and were used in a variety of applications (see \cite{APTldp, APTclt, BartheGamboaEtAl, GKR, NaorTAMS, PaourisWernerRelativeEntropy, SodinIsoperimetric}, to name just a few). In this paper, we will extend these results further by considering a similar class of distributions on $\B^n_p$ weighted by an additional homogeneous function. For some suitable function $f:\R^n\to [0,\infty)$ we construct a weighted uniform distribution $\bU_{n,p, f}$ and cone probability measure $\bC_{n,p,f}$, and show a weighted analogue to \cite{BartheGuedonEtAl} for $\bP_{n,p,\bW,f} := \bW(\{0\})\bC_{n,p,f}+\Psi_f\bU_{n,p,f}$ on $\B_p^n$. This will turn out to be very useful when considering analogues of $\ell_p^n$-balls in other spaces than $\R^n$.

In the present paper, we will study concentration phenomena on $p$-balls in both Euclidean space and within finite dimensional Schatten trace classes $\mathcal{S}^n_p$ in matrix space. Generally, for a given $p\in(0,\infty]$, the Schatten trace class $\mathcal{S}_p$ is the Banach space of compact linear operators between two Hilbert spaces whose singular values form a sequence within the sequence space $\ell_p$. We will however focus on the finite dimensional Schatten trace classes $\mathcal{S}^n_p$, i.e., the spaces of $(n\times n)$-matrices (with real, complex or quaternionic entries) whose singular values form a vector in $\ell_p^n$. Additionally, we will also consider their self-adjoint subclasses, that is, the  spaces of self-adjoint $(n\times n)$-matrices whose eigenvalues also form a vector in $\ell_p^n$. The unit balls in these Schatten trace classes $\mathcal{S}^n_p$ are what we will refer to as matrix $p$-balls.

There has been a rising interest in the study of these Schatten trace classes and their unit balls in recent years. For example, Gu\'edon and Paouris \cite{GuedonPaouris} provided concentration inequalities for points uniformly distributed within the matrix $p$-ball and König, Meyer and Pajor \cite{KMP1998} showed that the isotropy constants of matrix $p$-balls (for $p\in[1,\infty]$) are bounded. Barthe and Cordero-Erausquin \cite{BartheInvariance} derived variance estimates, Radke and Vritsiou \cite{RadkeVritsiou} proved the thin shell conjecture and Vritsiou \cite{Vritsiou} proved the variance conjecture for the operator norm in $\mathcal{S}_p^n$. Hinrichs, Prochno and Vybiral \cite{HPVEntropy, HPVGelfand} derived optimal bounds for the entropy numbers and sharp estimates for the Gelfand numbers of natural embeddings of $\mathcal{S}_p^n$, and Prochno and Strzelecki \cite{PSApprox} also considered the approximation numbers of such embeddings and studied their relationship to the Gelfand and Kolmogorov numbers. Kabluchko, Prochno and Th\"ale \cite{KPTEnsembles, KPTVolumeRatio} gave the exact asymptotic volumes and volume ratios of matrix $p$-balls and studied their intersection volumes. Also, Kabluchko, Prochno and Th\"ale \cite{KPTEnsembles, KPTSanov} studied the eigenvalue distribution as well as singular value distribution of random matrices distributed according to the cone probability measure and the uniform distribution in matrix $p$-balls. Following a line of classical arguments in the spirit of \cite{R1984} in combination with an approach from log-potential theory, they showed that for such random matrices the vector of the eigenvalues (or singular values) has respective distribution $\bC_{n,p,f}$ and $\bU_{n,p,f}$ on the Euclidean $\ell_p^n$-ball $\B^n_p$, with $f$ being the suitable repulsion factor between the eigenvalues (or singular values) of the random matrices (see e.g. \cite{AGZ2010}). Our aim here is to put this last result into a wider context by investigating the eigenvalue and singular value distribution of random matrices that have the analogue distribution to $\bP_{n,p,\bW}$ on matrix $p$-balls. Using similar arguments, we will show that the vector of eigenvalues of such a random matrix also is $p$-radially distributed according to $\bP_{n,p,\bW,f}$ on $\B^n_p$, with $f$ being the appropriate repulsion factor again, and the same holds for the vector of singular values on the non-negative segment of $\B^n_p$, denoted as $\B^n_{p,+}$. This connection paves the way to approach concentration phenomena on matrix $p$-balls via those in Euclidean space with appropriately weighted distributions.

As an application of the connection just described, we study the large deviation behaviors of random elements in Euclidean and matrix $p$-balls. The usage of large deviations theory was only recently introduced to asymptotic geometric analysis by Gantert, Kim and Ramanan \cite{GKR}, who derived a large deviation principle for the norm of projections of $\ell_p^n$-balls onto one-dimensional subspaces. Since then, large deviations have been applied as a useful tool in several other works (see e.g. \cite{APTldp, APTclt, KimPhD, KimRamanan}). In case of the Euclidean $\ell_p^n$-balls, the results of Kim and Ramanan \cite{KimRamanan} are of particular interest to us. For a random vector with distribution $\bC_{n,p}$ they gave a large deviation principle for the empirical measure of its coordinates. Their findings are in the spirit of the theorem of Sanov \cite[Theorem 2.1.10]{DZ}, as the corresponding rate function is the relative entropy perturbed by a $p$-th moment penalty. We want to expand on their results and give a large deviation principle for the empirical measure of a random vector with distribution $\bP_{n,p,\bW}$. We will show that, even though the distribution $\bP_{n,p,\bW}$ is highly dependent on the choice of $\bW$, for certain classes of $\bW$ the corresponding rate function will be universal to all $\bP_{n,p,\bW}$. The results of Kim and Ramanan have been further generalized by Frühwirth and Prochno \cite{FProchnoSanovOrlicz}, who derived a Sanov-type large deviation principle for the empirical measure of random vectors uniformly distributed in Orlicz-balls. In case of the matrix $p$-ball, an analogue result to that of Kim and Ramanan \cite{KimRamanan} has been given by Kabluchko, Prochno and Th\"ale \cite{KPTSanov}. They derived a large deviation principle for the empirical spectral measure (for both eigenvalues and singular values) of random matrices that are distributed according to the uniform distribution or the cone probability measure on the matrix $p$-ball. We will derive similar results for the analogue of $\bP_{n,p,\bW}$ on matrix $p$-balls and show a similar universality of the rate function. To do so, we will utilize the probabilistic representation results for the eigenvalue and singular value distributions we derived beforehand.

Summarizing, our overall goals are threefold. First, we want to expand the results from \cite{BartheGuedonEtAl} to weighted $p$-radial distributions $\bP_{n,p,\bW,f}$. This will be done in Section \ref{sec:Weighted}. Second, we want to show that for self-adjoint and non-self-adjoint random matrices, which are distributed according to the analogue of $\bP_{n,p,\bW}$ on matrix $p$-balls, the corresponding eigen- and singular value distributions are given by $\bP_{n,p,\bW,f}$ on $\B^n_p$ (and its non-negative analogue on $\B^n_{p,+}$), with $f$ being the appropriate repulsion factor. This will be done in Section \ref{sec:EigenSingularDistr}. And third, Sections \ref{sec:ApplicationEucl} and \ref{sec:ApplicationMat} will then use the previous results to derive several large deviation principles for Euclidean and matrix $p$-balls, respectively. We will prove a large deviation principle for the empirical measure of the coordinates of a random vector with distribution  $\bP_{n,p,\bW}$ on $\B_p^n$. Then we will show large deviation principles for the empirical spectral measures (for eigenvalues and singular values) of random matrices distributed according to the analogue of  $\bP_{n,p,\bW}$ on matrix $p$-balls by using the representations of the eigenvalue and singular value distributions as  $\bP_{n,p,\bW,f}$ from Section \ref{sec:EigenSingularDistr} for suitable choices of $f$. In the following Section \ref{sec:Preliminaries} preliminaries and notation will be collected.
\section{Preliminaries and Notation}\label{sec:Preliminaries}
\subsection{Notation and important distributions}\label{subsec:Distributions}
In this paper, we denote by $\text{vol}_n$ the $n$-dimensional Lebesgue measure on $\R^n$.  If $\mathbb{X}$ is a topological space, we write $\mathcal{B}(\mathbb{X})$ for the $\sigma$-field of Borel sets in $\mathbb{X}$. For a random variable $X$ with distribution $\bP$ we write  $X \sim \bP$ and denote by $\E X$ its expectation. For two random variables $X,Y$ with the same distribution we write $X \overset{d}{=}Y$. 
	For a random variable $X$ we denote by $\Lambda_X$ its cumulant  generating function with $\Lambda_X(t) := \log \E \left[ e^{ t X} \right], t\in \R$, where we often omit the index when it is clear from context. We call $\mathcal{D}_\Lambda := \{t \in \R: \Lambda(t) < +\infty\}$ the effective domain of $\Lambda_X$. 
	Furthermore, we define its Legendre-Fenchel transform $\Lambda^*_X$ as 
	\begin{equation} \label{eq:DefLegendre}
	\Lambda^*_X(x) := \sup_{t \in \R} \big[ xt - \Lambda_X(t)\big], \qquad x\in \R.
	\end{equation}
	Note that where $\Lambda_X$ is differentiable, the Legendre-Fenchel transform $\Lambda^*_X$ is an involution, i.e., for all $t \in \R$ where $\Lambda_X$ is differentiable, $\Lambda_X(t)= (\Lambda^*_X)^*(t)$ (see e.g.\,\cite[Chapter 4, p.\,72]{Gelfand2000}).
We recall that a real valued random variable $X$ is gamma distributed with shape $a>0$ and rate $b>0$ if its distribution has density  
$$\rho_{\bG}(x):= \displaystyle \frac{b^a}{\Gamma(a)} \, x^{a-1} \, e^{-bx} \, \mathbf{1}_{(0,\infty)}(x), \qquad x\in\R,$$
with respect to the Lebesgue measure on $\R$. We denote this by $X \sim \bG(a,b)$. For $a=1$ we call this an exponential distribution and write $X \sim \bE(b)$. Similarly, a real valued random variable $X$ is beta distributed with parameters $a,b >0$ if its distribution has Lebesgue density 
$$\rho_{\textbf{B}}(x):= \displaystyle \frac{1}{B(a,b)} \, x^{a-1} \, (1-x)^{b-1} \, \mathbf{1}_{(0,1)}(x),\qquad x\in\R,$$
where $a,b>0$ and $B(\, \cdot \, , \, \cdot \,)$ is the beta function. We denote this by $X \sim \textbf{B}(a,b)$. Finally, a real valued random variable $X$ has a so-called generalized Gaussian distribution if its distribution has density 
$$\displaystyle \rho_{\textup{gen}}(x):=  \displaystyle \frac{b}{2 a\Gamma\big(\frac{1}{b}\big)} \, e^{-\big(\frac{|x- \mu|}{a}\big)^b},\qquad x\in\R,$$
where $\mu \in \R$ and $a,b>0$, and denote this by $X \sim {\bN}_{\textup{gen}}(\mu, a, b)$. The generalized Gaussians are intimately connected to the geometry of $\ell_p^n$-balls. As we will see in Proposition \ref{prop:Barthe}, the generalized Gaussian distributions are the essential building block when constructing useful probabilistically equivalent representations for random vectors in $\ell_p^n$-balls with a wide variety of distributions. For these constructions we will be using the specific generalized Gaussian distribution ${\bN}_{\textup{gen}}(0, 1, p)=: \bN_p$ (for any $p\in (0,\infty)$) with density
$$\displaystyle \rho_{\bN_p}(x) := \frac{1}{2\Gamma\big(1+\frac{1}{p}\big)}\, e^{-|x|^p},  \qquad x\in\R.$$
\begin{rmk}\label{rmk:DifferentNormalizations}
In the literature different normalizations for generalized Gaussian distributions are used. The papers \cite{APTldp}, \cite{APTclt}, \cite{GKR} and \cite{KimRamanan} for example consider ${\bN}_{\textup{gen}}(0, p^{1/p}, p)$, whereas \cite{BartheGuedonEtAl}, \cite{KPTEnsembles}, \cite{KPTSanov} and \cite{SchechtmanZinn} work with ${\bN}_{\textup{gen}}(0, 1, p)$. This merely results in different normalization factors when constructing probabilistic representations.
\end{rmk}
\subsection{Polar Integration}\label{subsec:PolarInt}
Since distributions given by radially symmetric densities play a central role in our results, we need a tool to work with them efficiently. This tool is provided by the polar integration formula. Let $K\subset \R^n, n \in \N,$ be a set that is star shaped with respect to the origin and has finite non-zero volume. We define the uniform distribution on $K$ and the cone probability measure on the boundary $\partial K$ as
$$
\bU_{K}(\,\cdot\,) := {\vol_n(\,\cdot\,)\over \vol_n(K)}\qquad\text{and}\qquad\bC_{K}(\,\cdot\,) := {\vol_n(\{rx:r\in[0,1],x\in\,\cdot\,\})\over\vol_n(K)},
$$
respectively.
We can now formulate the polar integration formula. 

\begin{lemma}\label{lem:PolarIngerationAllg}
	For any set $K \subset \R^n$, $n \in \N$, that is star shaped with respect to the origin, contains the origin in its interior, and has finite non-zero volume, and any non-negative measurable function $h:\R^n\to\R$ it holds that
	$$
	\int_{\R^n} h(x)\,\dint x = n \, \vol_n(K)\int_0^\infty r^{n-1}\int_{\partial K} h(ry) \, \bC_{K}(\dint y)\,\dint r.
	$$
\end{lemma}
The proof of Lemma \ref{lem:PolarIngerationAllg} is the same as that of Proposition 3.3 in \cite{PTTSurvey}, which deals with the case where $K$ is a symmetric convex body, see also \cite[Proposition 1]{NR2003}.
When working with (non-negative) singular values in later sections, it will be convenient to have a version of the polar integration formula for the non-negative orthant $\R^n_+$ of $\R^n$.
\begin{cor}\label{cor:PolarIngerationPositive}
	For any set $K \subset \R^n_+$, $n \in \N$, that is star shaped with respect to the origin, contains the origin in its interior with respect to $\R^n_+$, and has finite non-zero volume, and any non-negative measurable function $h:\R^n_+\to\R$ it holds that
	$$
	\int_{\R^n_+} h(x)\,\dint x = n \, \vol_n(K)\int_0^\infty r^{n-1}\int_{\partial K} h(ry) \, \bC_{K}(\dint y)\,\dint r.
	$$
\end{cor}
\subsection{Geometry of $\ell_p^n$-balls}\label{subsec:Geometry}
 For $0<p\leq \infty$, $n\in\N$, and $x=(x_1,\ldots,x_n)\in\R^n$ let us denote by
$$
\|x\|_p := \begin{cases}
\Big(\sum\limits_{i=1}|x_i|^p\Big)^{1/p} &: p < \infty\\
\max\{|x_1|,\ldots,|x_n|\} &: p=\infty
\end{cases}
$$
the $\ell_p^n$-norm of $x$ (although this is only a quasi-norm for $0<p<1$). We let $\B_p^n:=\{x\in\R^n:\|x\|_p\leq 1\}$ be the unit $\ell_p^n$-ball and $\SSS_p^{n-1}:=\{x\in\R^n:\|x\|_p=1\}$ be the corresponding unit $\ell_p^n$-sphere. By $\bU_{n,p}:=\bU_{\B_p^n}$ we indicate the uniform distribution on $\B_p^n$ and by $\bC_{n,p}:=\bC_{\B_p^n}$ the cone probability measure on $\SSS_p^{n-1}$. Using the general polar integration formula from Lemma \ref{lem:PolarIngerationAllg} for $K=\B_p^n$ yields the polar integration formula for $\ell_p^n$-balls, which says that for any non-negative measurable function $h:\R^n\to\R$ it holds that
\begin{equation}\label{eq:PolarIngeration}
\int_{\R^n}h(x)\,\dint x = n \, \vol_n(\B_p^n)\int_0^\infty r^{n-1}\int_{\SSS_p^{n-1}}h(ry) \, \bC_{n,p}(\dint y)\,\dint r.
\end{equation}

The following result provides a probabilistic representation of certain mixtures of $\bU_{n,p}$ and $\bC_{n,p}$, see \cite[Theorem 3]{BartheGuedonEtAl}. It serves as a motivation for the results we present in Section \ref{sec:Weighted} below.

\begin{proposition}\label{prop:Barthe}
Let $n\in\N$ and $p\in(0,\infty)$. Let $\bW$ be a Borel probability measure on $[0,\infty)$ and $W$ be a random variable with $W \sim \bW$. Further, let $X_1, \ldots, X_n$ be independent and identically distributed random variables with $X_i \sim \bN_p$, which are independent of $W$. Then the random vector
$$
X\over (\|X\|_p^p+W)^{1/p}
$$
\noindent has distribution 
$$\bP_{n,p,\bW} := \bW(\{0\})\bC_{n,p}+\Psi\bU_{n,p}$$ 
on $\B^n_p$, where $\Psi(x)=\psi(\|x\|_p)$, $x\in\B_p^n$, is a $p$-radial density with
$$
\displaystyle \psi(s) = \displaystyle {1\over \Gamma\big({n\over p}+1\big)}{1\over (1-s^p)^{{n\over p}+1}}\bigg[\int_{(0,\infty)} w^{n\over p} \, e^{-{s^p\over 1-s^p} w} \, \bW(\dint w)\bigg],\qquad 0\leq s\leq 1.
$$
\end{proposition}

\subsection{Geometry of matrix $p$-balls}\label{subsec:GeometryMatrix}
Let $\F_\beta$ be the real numbers (if $\beta=1$), the complex numbers (if $\beta=2$) or the Hamiltonian quaternions (if $\beta=4$). For $n\in\N$ and $\beta\in\{1,2,4\}$ we let $\Mat_{n}(\mathbb F_\beta)$ be the space of $(n\times n)$-matrices with entries from $\F_\beta$. For a matrix $A\in \Mat_{n}(\mathbb F_\beta)$ we let $A^*$ be the adjoint of $A$. It is well known that, together with the scalar product $\langle A, B \rangle = \Re \Tr (A B^*)$, $\mathscr M_n(\mathbb{F}_\beta)$ becomes a Euclidean vector space. By $\vol_{\beta,n}(\,\cdot\,)$ we denote the volume on $\mathscr M_n(\mathbb{F}_\beta)$ corresponding to this scalar product. We can now introduce the self-adjoint matrix space $\mathscr H_n(\mathbb{F}_\beta) := \{A\in \Mat_{n}(\mathbb F_\beta):A=A^*\}$.   For each $A\in \mathscr H_n(\mathbb{F}_\beta)$ we denote by $\lambda_1(A) \le\ldots \le\lambda_n(A)$ the (real) eigenvalues of $A$ (see \cite[Appendix E]{AGZ2010} for a formal definition in the case $\beta=4$) and define $\lambda(A):=(\lambda_1(A),\ldots,\lambda_n(A)) \in \R^n$. For $0<p\leq\infty$ the self-adjoint matrix $p$-ball in $\mathscr H_n(\mathbb{F}_\beta)$ is defined as
$$
\B_{p,\beta}^{n, \mathscr{H}} := \Big\{A\in\mathscr H_n(\mathbb{F}_\beta) : \sum_{i=1}^n|\lambda_i(A)|^p\leq 1\Big\},
$$
where we interpret the condition as $\max\{|\lambda_1(A)|,\ldots,|\lambda_n(A)|\}\leq 1$ if $p=\infty$. Similarly, we let
$$
\SSS_{p,\beta}^{n-1, \mathscr{H}} := \Big\{A\in\mathscr H_n(\mathbb{F}_\beta) : \sum_{i=1}^n|\lambda_i(A)|^p = 1\Big\}
$$
be the self-adjoint matrix $p$-sphere. 
%
%Although the volume of $\B_{p,\beta}^{n, \mathscr{H}}$ cannot be computed explicitly, it is known from \cite{KPTEnsembles} that
%$$
%\lim_{n\to\infty}n^{{1\over p}+{1\over 2}}\,\vol_{\beta,n}\big(\B_{p,\beta}^{n, \mathscr{H}}\big)^{2 \over \beta n^2} = \Big({p\sqrt{\pi}\Gamma({p\over 2})\over \sqrt{e}\,\Gamma({p+1\over 2})}\Big)^{1/ p}\Big({\pi\over\beta}\Big)^{1/ 2}e^{3/ 4},
%$$
%with the convention that ${1\over p}=0$ and that the right-hand side is $({\pi\over\beta})^{1\over 2}e^{3\over 4}$ if $p=\infty$. 
%
The uniform distribution on $\B_{p,\beta}^{n, \mathscr{H}}$ and the cone probability measure on $\SSS_{p,\beta}^{n-1, \mathscr{H}}$ are denoted by $\bU_{n,p,\beta}^{\mathscr{H}}$ and $\bC_{n,p,\beta}^{\mathscr{H}}$, respectively. %
In the self-adjoint case, one can identify the matrix $p$-balls by virtue of the eigenvalues. We now consider the non-self-adjoint case, where this will be done via the singular values. For $A\in\Mat_n(\F_\beta)$, $n\in \N$, we denote by $s_1(A) \le \ldots \le s_n(A)$ the singular values of $A$, that is, $s_1(A),\ldots,s_n(A)$ are the non-negative eigenvalues of  $\sqrt{AA^*}$ (if $\beta\in\{1,2\}$ and if $\beta=4$ we refer to \cite[Corollary E.13]{AGZ2010} for a formal definition) and define $s(A):=(s_1(A),\ldots,s_n(A)) \in \R^n_+$. Additionally, we set  $s^2(A):=(s^2_1(A),\ldots,s^2_n(A)) \in \R^n_+$ to be the vector of squared ordered singular values. We do so, as the coordinates of $s^2(A)$ are the eigenvalues of $AA^*$ and can hence be treated in a fashion analogue to the vector of eigenvalues without needing to account for the root-operation. For $0<p\leq\infty$ the non-self-adjoint matrix $p$-ball is defined as
$$
\B_{p,\beta}^{n,{\mathscr{M}}} := \Big\{A\in\Mat_n(\F_\beta):\sum_{i=1}^n|s_i(A)|^p\leq 1\Big\},
$$
once again with the convention that the condition is replaced by $\max\{|s_1(A)|,\ldots,|s_n(A)|\}\leq 1$ if $p=\infty$. We also denote by
$$
\SSS_{p,\beta}^{n-1,{\mathscr{M}}} := \Big\{A\in\Mat_n(\F_\beta):\sum_{i=1}^n|s_i(A)|^p = 1\Big\}
$$
the non-self-adjoint matrix $p$-sphere. The uniform distribution on $\B_{p,\beta}^{n,{\mathscr{M}}}$ is denoted by $\bU_{n,p,\beta}^{\mathscr{M}}$ and we let $\bC_{n,p,\beta}^{\mathscr{M}}$ be the cone probability measure on $\SSS_{p,\beta}^{n-1,{\mathscr{M}}}$. %
 %
 %-----------------------------------
 %
%
Since the singular values are non-negative, we define the non-negative parts of the $\ell_p^n$-ball and $\ell_p^n$-sphere as $\B^n_{p,+}:=\B^n_p \cap \R^n_+$ and $\SSS^{n-1}_{p,+}:=\SSS_p^{n-1} \cap \R^n_+$. Accordingly, we define the respective uniform distribution $\bU_{n,p,+}:=\bU_{\B^n_{p,+}}$ and cone probability measure $\bC_{n,p,+}:=\bC_{\SSS^{n-1}_{p,+}}$.

\begin{rmk}\label{rmk:starsharped}\text{}
\begin{itemize} 
\item[(i)] Note that both $\mathscr H_n(\mathbb{F}_\beta)$ and $\mathscr M_n(\mathbb{F}_\beta)$ are Euclidean vector spaces of dimensions $\frac{\beta n (n-1)}{2} + \beta n$ and $\beta n^2$, respectively, and $\B_{p,\beta}^{n, \mathscr{H}}$ and $\B_{p,\beta}^{n, \mathscr{M}}$ both contain their respective origin in their interior and are star shaped with respect to their origins, as $\|\lambda(\kappa\,A)\|_p = \kappa\|\,\lambda(A)\|_p \le 1$ for $\kappa \in[0,1]$ (analogue for $\B_{p,\beta}^{n,{\mathscr{M}}}$). Finally, the volumes of $\B_{p,\beta}^{n,{\mathscr{H}}}$ and $\B_{p,\beta}^{n,{\mathscr{M}}}$ are non-zero and bounded (see e.g. \cite{KPTEnsembles, KPTVolumeRatio}). Hence, they both satisfy the conditions of the general polar integration formula in Lemma \ref{lem:PolarIngerationAllg}.
\item[(ii)] When referring to $\B^n_p$ as the \enquote{Euclidean} $\ell_p^n$-ball the term is supposed to denote the commutative setting of $\B^n_p$  in contrast to matrix $p$-balls $\B_{p,\beta}^{n,{\mathscr{H}}}$ and $\B_{p,\beta}^{n,{\mathscr{M}}}$ in the non-commutative setting of matrix space, although the matrix spaces themselves being Euclidean vector spaces.
\end{itemize}
\end{rmk}

 %-----------------------------------
 
For a Borel probability measure $\bW$ on $[0,\infty)$ we can now construct the analogues of the measure $\bP_{n,p,\bW}$ on the matrix $p$-balls $\B_{p,\beta}^{n,\mathscr{H}}$ and $\B_{p,\beta}^{n,{\mathscr{M}}}$ as 
\begin{equation}
\bP_{n,p,\bW,\beta}^{\mathscr{H}} :=\bW(\{0\})\bC_{n,p,\beta}^{\mathscr{H}}+\Psi^{\mathscr{H}}\bU_{n,p,\beta}^{\mathscr{H}} \text{ on } \B_{p,\beta}^{n,\mathscr{H}}, \text{ with } \Psi^{\mathscr{H}}(A) := \psi^{\mathscr{H}}(\|\lambda(A)\|_p), A\in \B_{p,\beta}^{n,\mathscr{H}}, \label{eq:PnpwSA}
\end{equation} 
and 
\begin{equation}\bP_{n,p,\bW,\beta}^{\mathscr{M}} :=\bW(\{0\})\bC_{n,p,\beta}^{\mathscr{M}}+\Psi^{\mathscr{M}}\bU_{n,p,\beta}^{\mathscr{M}} \text{ on } \B_{p,\beta}^{n,{\mathscr{M}}}, \text{ with } \Psi^{\mathscr{M}}(A) := \psi^{\mathscr{M}}(\|s^2(A)\|_p), A\in \B_{p,\beta}^{n,\mathscr{M}},\label{eq:PnpwNSA}
\end{equation}  
where $\psi^{\mathscr{H}}(s)$ and $\psi^{\mathscr{M}}(s)$ are $p$-radial densities given by
$$
 \displaystyle {1\over \Gamma\big({n+m\over p}+1\big)}{1\over (1-s^p)^{{n+m\over p}+1}}\bigg[\int_{(0,\infty)} w^{n+m\over p} \, e^{-{s^p\over 1-s^p}w} \, \bW(\dint w)\bigg],\qquad 0\leq s\leq 1,
$$
with $m= \frac{\beta n(n-1)}{2}$ for $\psi^{\mathscr{H}}(s)$, and $m=\frac{\beta }{2}n^2-n$ for $\psi^{\mathscr{M}}(s)$.
\\
\\
We define our distribution classes on matrix $p$-balls similarly to those on Euclidean $\ell_p^n$-balls via a $p$-radial distribution. Although we do not yet have a probabilistic representation for $\bP_{n,p,\bW,\beta}^{\mathscr{H}}$ and $\bP_{n,p,\bW,\beta}^{\mathscr{M}}$ as in Proposition \ref{prop:Barthe}, we still want to analyse the eigenvalue and singular value distribution of random matrices selected on $\B_{p,\beta}^{n,{\mathscr{H}}}$ and $\B_{p,\beta}^{n,{\mathscr{M}}}$ according to these distributions. We will be able to achieve this by establishing a new connection between these distributions on matrix $p$-balls and suitably weighted ($p$-radial) distributions on Euclidean $\ell_p^n$-balls. In contrast to the results of Proposition \ref{prop:Barthe} however, we need to account for the repulsion between the eigenvalues and singular values, hence the $p$-radial densities $\psi^{\mathscr{H}}(s)$, $\psi^{\mathscr{M}}(s)$ look different than the $\psi$ in Proposition \ref{prop:Barthe}, insofar as  the $n$ in $\psi$ is replaced by $n+m$, with $m$ being the degree of homogeneity $m$ of these repulsion factors. We will denote these repulsion factors of the eigen- and sigular values by $\Delta_\beta^c$ and $\nabla_\beta^c$ (formal definitions will follow in Section \ref{sec:EigenSingularDistr}) and as we will see, the two values for $m$ in the definitions \eqref{eq:PnpwSA} and \eqref{eq:PnpwNSA} are their respective degrees of homogeneity. We will explain this in further detail in the following sections. Also, the fact that $\bP_{n,p,\bW,\beta}^{\mathscr{H}}$ and $\bP_{n,p,\bW,\beta}^{\mathscr{M}}$ are in fact probability measures will follow directly from their probabilistic representations in Theorem \ref{thm:EvDistr} and Theorem \ref{thm:SvDistr}, respectively.

\subsection{Background material from large deviations theory}\label{subsec:PrelimLDP} We will need some basic results from large deviations theory. To keep this paper self-contained, we will present them here, while referring the reader to \cite{DZ,dH,Kallenberg} for further background material on large deviations.

\begin{df}
Let $\mathbb{X}$ be a Polish space equipped with the Borel $\sigma$-field $\mathcal{B}(\mathbb{X})$ and $(\textup{\textbf{P}}_n)_{n\in\N}$ a sequence of probability measures on $\mathbb{X}$. We say that $(\textup{\textbf{P}}_n)_{n\in\N}$ satisfies a large deviation principle (LDP) if there are two functions $s:\N\to (0,\infty)$ and $\mathcal{I}:\mathbb{X}\to[0,\infty]$, such that $\mathcal{I}$ is lower semi-continuous and
\begin{center}
$
\begin{array}{lllll}
a)&\displaystyle \underset{n\to\infty}{\textup{lim sup}} \,\, \frac{1}{s(n)}\log \textup{\textbf{P}}_n(C) &\le& -\mathcal{I}(C) &\,\,\,\, \,\,\,\, \,\,\,\, \,\,\,\, \text{ for all } \,\, C \in  \mathcal{B}(\mathbb{X}) \text{ closed,}\\
b)&\displaystyle \underset{n\to\infty}{\textup{lim inf}} \,\, \frac{1}{s(n)}\log \textup{\textbf{P}}_n(O) &\ge& -\mathcal{I}(O) &\,\,\,\, \,\,\,\, \,\,\,\, \,\,\,\, \text{ for all } \,\, O\in  \mathcal{B}(\mathbb{X})  \text{ open,}\\
\end{array}
$
\end{center}
where for $B \in  \mathcal{B}(\mathbb{X}) $ we define $\mathcal{I}(B) :=\displaystyle\inf_{x\in B}\mathcal{I}(x)$. We call $s$ the speed and $\mathcal{I}$ the rate function. We say that $\mathcal{I}$ is a good rate function, if it has compact sub-level sets.
\end{df}

Frequently LDPs are defined for sequences $(X_n)_{n\in\N}$ of random variables by applying the above definition to the sequence of their distributions. We apply the definition of LDPs to sequences of random measures as well. For a Polish space $\mathbb{X}$ we denote by $\mathcal{M}(\mathbb{X})$ the space of probability measures on $\mathbb{X}$ endowed with the weak topology and recall that $\mathcal{M}(\mathbb{X})$ is itself again Polish. %
Now we can go forward with presenting the results from large deviations theory. The first concerns the large deviation behavior of two sequences of random variables with the same speed in a product space. For the result and its proof, see \cite[Proposition 2.4 \& Appendix A]{APTldp}.

\begin{proposition}\label{prop:ProdLDP}
Let $\mathbb{X},\mathbb{Y}$ be Polish spaces. Let $(X_n)_{n\in\N}$, $(Y_n)_{n\in\N}$ be sequences of random variables in $\mathbb{X}$ and $\mathbb{Y}$, respectively. Assume that $(X_n)_{n\in\N}$ and $(Y_n)_{n\in\N}$ are independent. Further assume that both $(X_n)_{n\in\N}$ and $(Y_n)_{n\in\N}$ satisfy LDPs with the same speed $s(n)$ and respective good rate functions $\mathcal{I}_X:\mathbb{X}\to[0,\infty]$ and $\mathcal{I}_Y:\mathbb{Y}\to[0,\infty]$. Consider the sequence of random variables $(Z_n)_{n\in\N}$ on $\mathbb{X}\times \mathbb{Y}$ with $Z_n =(X_n, Y_n)$. Then $(Z_n)_{n\in\N}$ satisfies an LDP with speed $s(n)$ and good rate function $\mathcal{I}_{Z}$ with $\mathcal{I}_{Z}(z) = \mathcal{I}_{X}(x) + \mathcal{I}_{Y}(y)$ for all $z=(x,y) \in \mathbb{X}\times \mathbb{Y}$.

\end{proposition}

The next result is the so called contraction principle and it allows us to transport an LDP from one sequence of random variables to another by means of a continuous map. The result can be found, e.g., in \cite[Theorem 4.2.1]{DZ}.

\begin{proposition}\label{prop:ContrPrinc}
Let $\mathbb{X}, \mathbb{Y}$ be Polish spaces and $f:\mathbb{X}\to \mathbb{Y}$ a continuous function. Also let $(X_n)_{n\in\N}$ be a sequence of random variables in $\mathbb{X}$ that satisfies an LDP with speed $s(n)$ and good rate function $\mathcal{I}_{X}$. Then the sequence of random variables $(Y_n)_{n\in\N}: = (X_n\circ f^{-1})_{n\in\N}$ satisfies an LDP with speed $s(n)$ and good rate function $\mathcal{I}_{Y}(y)=\inf\{\mathcal{I}_{X}(x) \,|\,x \in \mathbb{X}, f(x)=y\}$.
\end{proposition}

\begin{rmk} In the upcoming LDP results we want to use the contraction principle in the following situation. Let $\mathbb{X}, \mathbb{Y}$ be Polish spaces, $f:\mathbb{X}\to\mathbb{Y}$ a continuous map and $(\mu_n)_{n\in\N}$ a sequence of random measures on $\mathbb{X}$. Let $(\mu_n)_{n\in\N}$ satisfy an LDP on $\mathcal{M}(\mathbb{X})$ with speed $s:\N\to (0,\infty)$ and rate function $\mathcal{I}_\mu:\mathcal{M}(\mathbb{X})\to [0,\infty]$. We then consider the sequence of random measures $(\nu_n)_{n\in\N}$ on $\mathbb{Y}$ with $\nu_n = \mu_n \circ f^{-1}$ and want to use the contraction principle to infer an LDP for $(\nu_n)_{n\in\N}$. In this case the function that is actually \enquote{transporting} the LDP is not $f:\mathbb{X}\to \mathbb{Y}$ but $F:\mathcal{M}(\mathbb{X}) \to \mathcal{M}(\mathbb{Y})$ with $F(\mu)=\mu \circ f^{-1}$. So in general the continuity of $F$ has to be given rather than that of $f$. But the latter follows directly from the continuity of $f$ by the definition of weak convergence.
\end{rmk}

For sequences of non-identically distributed random variables, that do however exhibit a certain level of distributional convergence, the theorem of Gärtner-Ellis (see e.g. \cite[Proposition 2.9]{APTldp}, \cite[Theorem 2.3.6]{DZ}, or \cite[Theorem V.6]{dH}) provides a useful way to gain an LDP. 
\begin{proposition}\label{prop:GaertnerEllis}
Let $(X^{(n)})_{n \in \N}$ be a sequence of random variables with cumulant generating functions $\Lambda_n$ and $k \in [1, \infty)$. We assume that for all $t\in \R$ the limit $\Lambda(t):= \lim\limits_{n\to \infty} \frac{1}{n^k} \Lambda_n(n^k t)$ exists in $[-\infty, +\infty]$ and that the origin is an interior point of the effective domain $\mathcal{D}_\Lambda$. We furthermore assume that $\Lambda$ is semi-continuous and differentiable on the interior of $\mathcal{D}_\Lambda$. Then the sequence $(X^{(n)})_{n \in \N}$ satisfies an LDP with speed $n^k$ and rate function $\Lambda^*$.
\end{proposition}
\subsection{Asymptotic Approximations for Laplace-type Integrals}\label{subsec:Laplace}
Finally, we will need some tools to analyze asymptotic behavior of Laplace-type integrals to prove our large deviation results. One of them will be provided by the Laplace principle, as presented in \cite[Proposition 2.10]{APTldp}, and several useful adaptations fitting for our purposes. We begin with the former.
\begin{proposition}\label{prop:LaplacePrinc}
Let $- \infty < a < b < +\infty$ and $p:[a,b] \to \R$ be a twice continuously differentiable function with a unique point $x_0 \in (a,b)$ such that $p(x_0) = \max\limits_{x\in[a,b]} p(x)$ and $p^{\prime \prime}(x_0) <0$. Further, let $q: [a,b] \to \R$ be a positive measurable function. Then
$$\lim \limits_{n \to \infty} \frac{\int_a^b q(x) \, e^{np(x)} \, \dint x}{\sqrt{\frac{2 \pi}{n |p^{\prime \prime}(x_0)|}} \, q(x_0) \, e^{np(x_0)}} = 1.$$
\end{proposition}
\begin{rmk}\label{rmk:AdaptationLaplace}
	Proposition \ref{prop:LaplacePrinc} effectively means that 
	$$\lim\limits_{n\to\infty} \frac{1}{n} \log \int_a^b q(x) \, e^{np(x)} \, \dint x = p(x_0).$$
	However, we want to fit this result somewhat further to our needs. Assuming the set-up of Proposition \ref{prop:LaplacePrinc}, lets say $s^{(1)} := (s_n^{(1)})_{n \in \N}$ and $s^{(2)} := (s_n^{(2)})_{n \in \N}$ are sequences, where $s^{(1)}$ is non-negative  and bounded, and $s^{(2)}$ is positive (or at least positive almost everywhere), such that
	\begin{eqnarray} \label{eq:LaplaceCondition}
	\lim\limits_{n\to\infty} \frac{1}{n} \, \big| \! \log s_n^{(2)} \big| &<& +\infty.
	\end{eqnarray}
	Expanding the fraction in the Laplace principle in Proposition \ref{prop:LaplacePrinc} by $s_n^{(2)}$ and adding 
	$$\lim_{n\to \infty} \frac{s_n^{(1)}}{ s_n^{(2)} \, \sqrt{\frac{2 \pi}{n |p^{\prime \prime}(x_0)|}} \, q(x_0) \, e^{np(x_0)}},$$
	which is zero, since $s_n^{(1)}$ is bounded, yields that
	%
	%It then directly follows from the Laplace principle in Proposition \ref{prop:LaplacePrinc} that
	%
	$$\lim \limits_{n \to \infty} \frac{s_n^{(1)} +s_n^{(2)} \int_a^b q(x) \, e^{np(x)} \, \dint x}{s_n^{(2)} \, \sqrt{\frac{2 \pi}{n |p^{\prime \prime}(x_0)|}} \, q(x_0) \, e^{n p(x_0)}} = 1.$$
	Thus, we have that 
	\begin{equation} \label{eq:ExpLaplace2}
	\lim\limits_{n\to\infty} \frac{1}{n} \log \left[ s_n^{(1)} + s_n^{(2)}\int_a^b q(x) \, e^{n p(x)} \, \dint x \right]= \lim\limits_{n\to\infty} \frac{1}{n} \log s_n^{(2)} + p(x_0).
	\end{equation}
\end{rmk}
The last tool for analysing asymptotic integral behavior will be the following result by Breitung and Hohenbichler \cite{Breitung1989}, that provides us with asymptotic approximations of Laplace-type integrals even if the involved functions maximize on the boundary of the integration domain, specifically at the origin. This is the result given in \cite[Lemma 4]{Breitung1989} for $n=1, k=1$, applied to functions $q$ and $p$ instead of $h$ and $f$. The parameter $\lambda$ from \cite{Breitung1989} in our setting is replaced by the integer $n \in \N$. Since $n=k=1$, the last condition in \cite[Lemma 4]{Breitung1989} regarding the Hessian of $p$ at $0$, that is, $p^{\prime \prime}(0)$, falls away.
\begin{proposition} \label{prop:Breitung}
		Let $F \subset \R$ be a compact set containing the origin in its interior. If 
		\begin{enumerate}[label=(\alph*)]
		\item $p:F\to \R$ and $q:F \to \R$ are continuous functions with $q(0) \ne 0$, 
		\item $p(x) < p(0)$ for all $x \in F \cap \R_+ \setminus\{0\}$,
		\item there is a neighbourhood $V \subset F$ of $0$ in which $p$ is twice continuously differentiable,
		\item $p^\prime(0) < 0$,
		\end{enumerate}
		then it holds that 
		\begin{eqnarray*}
	\displaystyle \lim \limits_{n \to \infty} \frac{\int_{F \cap \R_+} q(x) \, e^{np(x)} \, \dint x}{n^{-1} \, {|p^\prime(0)|}^{-1} q(0) \, e^{n p(0)} } = 1.
	\end{eqnarray*}
\end{proposition}
We will only need the results from Proposition \ref{prop:Breitung} to handle the asymptotics of one specific Laplace-type integral over the set $[0,1]$, where the function in the exponent maximizes on the boundary at $0$. Hence we will derive another asymptotic integral expansion result tailored specifically to our purposes.
\begin{rmk} \label{rmk:ApplBreitung}
For functions $q$ and $p$ as described in Proposition \ref{prop:Breitung} and the set $F = [-1,1]$ it holds that
$$\lim\limits_{n\to\infty} \frac{1}{n} \log \int_{0}^{1} q(x) \, e^{np(x)} \, \dint x  = \lim\limits_{n\to\infty} \frac{1}{n} \log \int_{[-1,1] \cap \R_+} q(x) \, e^{np(x)} \, \dint x  = p(0).$$
By the same arguments as in Remark \ref{rmk:AdaptationLaplace} it also holds that 
\begin{equation} \label{eq:ExpLaplaceBreitung}
\lim\limits_{n\to\infty} \frac{1}{n} \log \left[ s_n^{(1)} + s_n^{(2)}\int_0^1 q(x) \, e^{np(x)} \, \dint x \right]= \lim\limits_{n\to\infty} \frac{1}{n} \log s_n^{(2)} + p(0).
\end{equation}
for sequences $s^{(1)} := (s_n^{(1)})_{n \in \N}$ and $s^{(2)} := (s_n^{(2)})_{n \in \N}$ as described there.
\end{rmk}

\section{Weighted $p$-radial distributions on $\ell_p^n$-balls}\label{sec:Weighted}

In this section, we describe a class of probability distributions on the classical $\ell_p^n$-ball $\B_p^n$, $n\in\N$, and its non-negative counterpart $\B_{p,+}^n$ in $\R^n_+$, generalizing the approach in \cite{BartheGuedonEtAl}, by allowing for an additional homogeneous weight function. To introduce our framework, we let $f:\R^n\to [0,\infty)$ be a measurable function, which we assume to be (positively) homogeneous of degree $m$ for some $m\geq 0$. By this we mean that $f(tx)=t^mf(x)$ for all $t\geq 0$. We also assume that $f$ is integrable with respect to the cone probability measure $\bC_{n,p}$ on the $\ell_p^n$-sphere $\SSS_p^{n-1}$. In this paper, we write $\mathscr F_m^+(\R^n)$ for the class of such functions (omitting its dependence on $p$ in our notation). For $p \in (0,\infty)$ and $f\in\mathscr{F}_m^+(\R^n)$ we let $C_{n,p,f}\in(0,\infty)$ be the normalization constant such that
\begin{equation} \label{eq:NormConst}
C_{n,p,f}\int_{\R^n}f(x)\,e^{-\|x\|_p^p} \,\dint x = 1,
\end{equation}
and denote by $\bU_{n,p,f}$ the probability measure on $\B_p^n$ with density 
$$
x\mapsto C_{n,p,f} \, \vol_n(\B_p^n)  \, \Gamma\Big({n+m\over p}+1\Big) \, f(x),\qquad x\in\B_p^n,
$$
with respect to $\bU_{n,p}$. Similarly, we let $\bC_{n,p,f}$ be the probability measure on $\SSS_p^{n-1}$ with density
$$
y\mapsto C_{n,p,f} \, n \, \vol_n(\B_p^n)  \, p^{-1} \, \Gamma\Big({n+m\over p}\Big) \, f(y),\qquad y\in\SSS_p^{n-1},
$$
with respect to $\bC_{n,p}$. The density property of the above functions follows from applying the polar integration formula for $\B^n_p$ in \eqref{eq:PolarIngeration} and straight-forward calculations. It also follows from the calculations in the proof of Lemma \ref{lem:DistributionConeMeasure}. As mentioned in Section \ref{subsec:GeometryMatrix}, singular values are non-negative and therefore, as we will see in Section \ref{subsec:SingularDistr}, the vector of singular values is distributed on $\B^n_{p,+}$ and $\SSS^{n-1}_{p,+}$. For $p \in (0,\infty)$ and $f\in\mathscr{F}_m^+(\R^n_+)$ we define a constant $C_{n,p,f,+}$ and distributions $\bU_{n,p,f,+}$ and $ \bC_{n,p,f,+}$ analogue to the above with respect to $\B^n_{p,+}$ and $\SSS^{n-1}_{p,+}$. We want to formulate all results in this section for both the classical $\ell_p^n$-balls and -spheres and their non-negative counterparts. However, as the proofs work in an entirely analogue fashion, for the sake of brevity we will use the index $\boxplus$ with all relevant quantities, indicating that any given result can be formulated with and without a $+$ in the index of these quantities, i.e., for both $\B^n_{p}$, $\SSS^{n-1}_{p}$ and $\B^n_{p,+}$, $\SSS^{n-1}_{p,+}$. The relevant proof will then always be carried out for $\B^n_{p}$ and $\SSS^{n-1}_{p}$, and only the changes necessary in the non-negative case pointed out, if any need to be made.\\
\\
In the next lemma we derive probabilistic representations of the distributions $\bU_{n,p,f, {\ssbp}}$ and $\bC_{n,p,f, \ssbp}$. This was proven in \cite[Lemma 4.2]{KPTEnsembles} for the classical case in $\R^n$ and the proof here works completely analogue.

\begin{lemma}\label{lem:DistributionConeMeasure}
Let $0 <  p <\infty$ and $f\in\mathscr F_m^+(\R^n_{\ssbp})$ for some $m \geq 0$. Let $X=(X_1,\ldots,X_n)$ be a random vector with joint density $C_{n,p,f,\ssbp} \, e^{-\|x\|_p^p} \, f(x)$, $x\in\R^n_{\ssbp}$. 
\begin{itemize}
\item[(i)] Then the random vector $X\over\|X\|_p$ has distribution $\bC_{n,p,f,\ssbp}$ and $X\over\|X\|_p$ and $\|X\|_p$ are independent. 
\item[(ii)] Independently of $X$, let $U$ be uniformly distributed on $[0,1]$. Then $U^{1\over n+m}{X\over\|X\|_p}$ has distribution $\bU_{n,p,f, \ssbp}$.
\end{itemize}
\end{lemma}
\begin{proof}
Consider a non-negative measurable function $h: \SSS_p^{n-1} \to \R$. We use the polar integration formula \eqref{eq:PolarIngeration} as well as the homogeneity of $f$ to deduce that
\begin{align*}
\E h\bigg(\frac{X}{\|X\|_p}\bigg) &= C_{n,p,f}\int_{\R^n}  f(x) \, e^{-\|x\|_p^p}\, h\bigg(\frac{x}{\|x\|_p}\bigg) \dint x \cr 
& =C_{n,p,f}\,n\,\vol_n(\B_p^n)
\int_0^\infty r^{n+m-1} \, e^{-r^p} \dint r \int_{\SSS_p^{n-1}} f(y)\, h(y) \, \bC_{n,p}(\dint y) \cr
& = C_{n,p,f} \,n\, \vol_n(\B_p^n) \, p^{-1} \, \Gamma\Big({n+m\over p}\Big) \int_{\SSS_p^{n-1}} f(y)\, h(y) \, \bC_{n,p}(\dint y)\\
&= \int_{\SSS_p^{n-1}}h(y) \, \bC_{n,p,f}(\dint y).
\end{align*}
This proves the claim in (i). To show (ii), let $h:\B_p^n\to\R$ be a non-negative measurable function. We notice that if $U$ is uniformly distributed on $[0,1]$, the random variable $U^{1\over n+m}$ has density $r\mapsto (n+m) \, r^{n+m-1}$, $r\in[0,1]$, with respect to the Lebesgue measure on $[0,1]$. Using the result from part (i), the homogeneity of $f$, and the polar integration formula \eqref{eq:PolarIngeration}, we find that%
\begin{align*}
&\bE h\Big(U^{1\over n+m}{X\over\|X\|_p}\Big)\\
&\qquad= C_{n,p,f} \, n \, \vol_n(\B_p^n)  \, p^{-1} \, \Gamma\Big({n+m\over p}\Big) \int_0^1 (n+m)\,r^{n+m-1}\int_{\SSS_p^{n-1}} f(y)\,h(ry) \,\bC_{n,p}(\dint y)\,\dint r\\
&\qquad=C_{n,p,f}  \,\Gamma\Big({n+m\over p}\Big)\, {n+m\over p}\, n \, \vol_n(\B_p^n) \int_0^1 r^{n-1}\int_{\SSS_p^{n-1}} f(ry)\,h(ry) \,\bC_{n,p}(\dint y)\,\dint r\\
&\qquad=C_{n,p,f} \, \Gamma\Big({n+m\over p }+1\Big) \, \int_{\B_p^n} f(x)\,h(x)\,\dint x\\
&\qquad=C_{n,p,f} \, \vol_n(\B_p^n) \, \Gamma\Big({n+m\over p }+1\Big) \, \int_{\B_p^n} f(x)\,h(x)\,\bU_{n,p}(\dint x).\\
&\qquad=\int_{\B_p^n}h(x) \,\bU_{n,p,f}(\dint x).
\end{align*}
The proof for $\B^n_p$ and $\SSS_p^{n-1}$ is thus complete. In the non-negative setting one proceeds in the same way, but applies the non-negative polar integration formula from Corollary \ref{cor:PolarIngerationPositive} for $K=\B_{p,+}^n$.
\end{proof}

\begin{rmk}\label{rmk:WeightedGenGaussians}
Note, that the distribution of a random vector $X=(X_1,\ldots,X_n)$ with joint density $C_{n,p,f} \, e^{-\|x\|_p^p} \, f(x)$, $x\in\R^n$, is just the $n$-fold product distribution ${\bN}_p^{\otimes n}$ of the generalized Gaussian distribution ${\bN}_p$, weighted by the function $f$ (and appropriately renormalized). So the distribution ${\bN}_p$ is in fact the core building block of the probabilistic representations, but is somewhat implicit in the density of the random vector. In the non-negative case $\bN_p$ is replaced by the truncated and renormalized version of $\bN_p$ in this role.%
\end{rmk}%
The next result is the main result of this section. It is a general version of \cite[Theorem 3]{BartheGuedonEtAl}, where it is assumed that the weight function $f$ is identically equal to $1$. The proof will work along the lines of that in \cite{BartheGuedonEtAl} and relies on a multiple application of the polar integration formula.
\begin{thm}\label{thm:ProbablisticRepresentation}
Let $\bW$ be a Borel probability measure on $[0,\infty)$. Let $0<p<\infty$ and $f\in\mathscr{F}_m^+(\R^n_{\ssbp})$ for some $m\geq 0$. Let $X=(X_1,\ldots,X_n)$ be a random vector with density $C_{n,p,f, \ssbp}\,e^{-\|x\|_p^p}\, f(x)$, $x\in\R^n_{\ssbp}$, and $W$ a non-negative random variable with distribution $\bW$, which is independent of $X$. Then the random vector
$$
X\over (\|X\|_p^p+W)^{1/p}
$$
has distribution $\bP_{n,p,\bW,f, \ssbp} := \bW(\{0\})\bC_{n,p,f, \ssbp}+\Psi_f\bU_{n,p,f, \ssbp},$
where $\Psi_f(x)=\psi_f(\|x\|_p)$, $x\in\B_{p,\ssbp}^n$, is a $p$-radial density with
$$
\displaystyle \psi_f(s) =  \displaystyle {1\over \Gamma\big({n+m\over p}+1\big)}{1\over (1-s^p)^{{n+m\over p}+1}}\bigg[\int_{(0,\infty)} w^{n+m\over p} \, e^{-{s^p\over 1-s^p}w} \, \bW(\dint w)\bigg],\qquad 0\leq s\leq 1.
$$
\end{thm}
\begin{proof}
Let $h:\R^n\to\R$ be an arbitrary non-negative measurable function. Then,
\begin{align}
\nonumber\bE h\Big({X\over (\|X\|_p^p+W)^{1/p}}\Big) &= \int_{[0,\infty)}\bE h\Big({X\over (\|X\|_p^p+w)^{1/p}}\Big) \,\bW(\dint w)\\
&= \int_{[0,\infty)} \bE h\Big(\Big({\|X\|_p^p\over \|X\|_p^p+w}\Big)^{1/p}{X\over\|X\|_p}\Big) \, \bW(\dint w).\label{eq:PolarIntegrationStep1}
\end{align}
For fixed $w>0$ we compute the expectation under the integral sign as follows by means of the polar integration formula \eqref{eq:PolarIngeration}:
\begin{align*}
&\bE h\Big(\Big({\|X\|_p^p\over \|X\|_p^p+w}\Big)^{1/p}{X\over\|X\|_p}\Big)\\
&\qquad=C_{n,p,f}\int_{\R^n}e^{-\|x\|_p^p}\,f(x)\,h\Big(\Big({\|x\|_p^p\over \|x\|_p^p+w}\Big)^{1/p}{x\over\|x\|_p}\Big)\,\dint x\\
&\qquad= n \, \vol_n(\B_p^n) \, C_{n,p,f}\int_0^\infty r^{n-1}\,e^{-r^p}\int_{\SSS_p^{n-1}} f(ry)\,h\Big(\Big({r^p\over r^p+w}\Big)^{1/p}y\Big)\,\bC_{n,p}(\dint y)\, \dint r \\
&\qquad= n \, \vol_n(\B_p^n) \, C_{n,p,f}\int_0^\infty r^{n+m-1}\,e^{-r^p}\int_{\SSS_p^{n-1}} f(y)\,h\Big(\Big({r^p\over r^p+w}\Big)^{1/p}y\Big) \, \bC_{n,p}(\dint y)\, \dint r,
\end{align*}
where we used in addition the assumption that $f$ is $m$-homogeneous. Applying the change of variables $r^p={s^p\over 1-s^p} \, w$, we get
\begin{align*}
&\bE h\Big(\Big({\|X\|_p^p\over \|X\|_p^p+w}\Big)^{1/p}{X\over\|X\|_p}\Big)\\
&\qquad= n \, \vol_n(\B_p^n) \, C_{n,p,f} \, w^{n+m\over p}\int_0^1 {s^{n+m-1}\over (1-s^p)^{{n+m\over p}+1}}\,e^{-{s^p\over 1-s^p}w}\int_{\SSS_p^{n-1}} f(y)\,h(sy) \, \bC_{n,p}(\dint y) \, \dint s.
\end{align*}
Also, we know from Lemma \ref{lem:DistributionConeMeasure} (i) that $X/\|X\|_p$ has distribution $\bC_{n,p,f}$, which in turn has density
$$
y\mapsto  C_{n,p,f} \, n \, \vol_n(\B_p^n)  \, p^{-1} \, \Gamma\Big({n+m\over p}\Big) \, f(y),\qquad y\in\SSS_p^{n-1},
$$
with respect to $\bC_{n,p}$. Thus,
$$\bE h\Big(\Big({\|X\|_p^p\over (\|X\|_p^p+w)^{1/p}}\Big)^{1/p}{X\over\|X\|_p}\Big) = p \, {\Gamma\Big({n+m\over p}\Big)}^{-1}  w^{n+m\over p}\int_0^1{s^{n+m-1}\over (1-s^p)^{{n+m\over p}+1}}\,e^{-{s^p\over 1-s^p}w}\,\bE h\Big(s{X\over\|X\|_p}\Big)\, \dint s.$$
As a consequence, recalling \eqref{eq:PolarIntegrationStep1}, we see that
\begin{align}
\nonumber &\bE h\Big({X\over (\|X\|_p^p+W)^{1/p}}\Big) - \bW(\{0\})\bE h\Big({X\over \|X\|_p}\Big)\\
\nonumber&\qquad=  p\, {\Gamma\Big({n+m\over p}\Big)}^{-1}  \,\int_{(0,\infty)} w^{n+m\over p}\int_0^1{s^{n+m-1}\over (1-s^p)^{{n+m\over p}+1}}\,e^{-{s^p\over 1-s^p}w}\,\bE h\Big(s{X\over\|X\|_p}\Big) \,\dint s\,\bW(\dint w)\\
&\nonumber\qquad=  \frac{n+m}{ {\Gamma\big({n+m\over p}+1\big)}} \, \int_0^1 {s^{n+m-1}\over (1-s^p)^{{n+m\over p}+1}}\Bigg[\int_{(0,\infty)} w^{n+m\over p}e^{-{s^p\over 1-s^p}w} \, \bW(\dint w) \Bigg]\bE h\Big(s{X\over\|X\|_p}\Big) \,\dint s\\
&\qquad= (n+m)\int_0^1\,s^{n+m-1}\,\psi_f(s)\,\bE h\Big(s{X\over\|X\|_p}\Big) \,\dint s. \label{eq:PolarIntegrationFinal1}
\end{align}
Finally, if $\bM$ is any probability measure on $\B_p^n$ with $p$-radial density $\Phi(x)=\varphi(\|x\|_p)$, $x\in\B_p^n$, with respect to $\bU_{n,p,f}$, the polar integration formula \eqref{eq:PolarIngeration}, together with Lemma \ref{lem:DistributionConeMeasure} (i), yield the identity
\begin{align}
\nonumber&\int_{\B_p^n} h(x) \,\bM(\dint x)\\
\nonumber&\qquad= \int_{\B_p^n} h(x) \, \Phi(x) \, \bU_{n,p,f}(\dint x)\\
\nonumber&\qquad= C_{n,p,f}\, \Gamma\Big({n+m\over p}+1\Big)\int_{\B_p^n} h(x) \, \varphi(\|x\|_p)\,f(x) \, \dint x\\
\nonumber&\qquad= (n+m) \, C_{n,p,f} \, p^{-1} \,\Gamma\Big({n+m\over p}\Big)\, n\, \vol_n(\B_p^n)\int_0^1\varphi(s)\,s^{n+m-1}\int_{\SSS_p^{n-1}} f(y)\,h(sy)\, \bC_{n,p}(\dint y)\,\dint s\\
&\qquad= (n+m)\int_0^1\varphi(s)\,s^{n+m-1}\,\bE h\Big(s{X\over\|X\|_p}\Big) \,\dint s.\label{eq:PolarIntegrationFinal2}
\end{align}
The claim follows by comparing \eqref{eq:PolarIntegrationFinal1} with \eqref{eq:PolarIntegrationFinal2}. Again, the same follows in the non-negative setting by using the non-negative polar integration formula from Corollary \ref{cor:PolarIngerationPositive} for $K=\B_{p,+}^n$.
\end{proof}
\begin{rmk}
Taking $f\equiv 1$, which is homogeneous of degree $m=0$, reduces $\bU_{n,p,f, \ssbp}$ to $\bU_{n,p, \ssbp}$ on $\B_{p, \ssbp}^n$ and $\bC_{n,p,f, \ssbp}$ to $\bC_{n,p, \ssbp}$. As a consequence, Theorem \ref{thm:ProbablisticRepresentation} turns into \cite[Theorem 3]{BartheGuedonEtAl} (see Proposition \ref{prop:Barthe}), as already pointed out above.
\end{rmk}
Let us now consider a few specific distributions for $\bW$ and observe the corresponding distributions $\bP_{n,p,\bW,f, \ssbp}$ on $\B_{p, \ssbp}^n$.

\begin{example}\rm 
Let $f\in \mathscr{F}_m^+(\R^n_{\ssbp})$ and $\bW=\delta_0$ be the Dirac measure at $0$. Then $\Psi_f\equiv 0$ and $\bW(\{0\})=1$, thus for $\bP_{n,p,\bW,f, \ssbp}$ we obtain the weighted cone probability measure $\bC_{n,p,f, \ssbp}$ on $\B^n_{p, \ssbp}$.
\end{example}

\begin{example}\label{ex:Uniform}\rm 
Let $f\in \mathscr{F}_m^+(\R^n_{\ssbp})$ and $\bW=\textbf{E}(1)$ be the exponential distribution with parameter $1$. In this case, we get 
\begin{align*}
\psi_f(s) &= \displaystyle {1\over  \Gamma({n+m\over p}+1)}{1\over (1-s^p)^{{n+m\over p}+1}}\bigg[\int_{(0,\infty)} w^{n+m\over p}e^{-{s^p\over 1-s^p}w}\,\textbf{E}(1)(\dint w)\bigg]\\
&={1\over  \Gamma({n+m\over p}+1)}{1\over (1-s^p)^{{n+m\over p}+1}}\bigg[\int_{(0,\infty)}   w^{({n+m\over p}+1)-1}e^{-{1\over 1-s^p}w} \, \dint w \bigg] = 1.
\end{align*}
Thus, $\bP_{n,p,\bW,f, \ssbp}$ is the weighted uniform distribution $\bU_{n,p,f, \ssbp}$ on $\B^n_{p, \ssbp}$.
\end{example}
\begin{example}\label{ex:Beta}\rm 
As a third example, we consider $f\in \mathscr{F}_m^+(\R^n_{\ssbp})$ and $\bW=\bG(\alpha,1)$ to be a gamma distribution with shape $\alpha>0$ and rate $1$. In this situation the random variable $X\over(\|X\|_p^p+W)^{1/p}$ generates a beta-type distribution $\bP_{n,p,\bW,f, \ssbp} = \Psi_f\bU_{n,p,f, \ssbp}$ on $\B_{p,\ssbp}^n$, whose density is a constant multiple of $x\mapsto(1-\|x\|_p^p)^{\alpha-1}$, $\|x\|_p\leq 1$. To see that, we set $\bW = \textbf{G}(\alpha,1)$ and compute $\psi_f(s)$ for $s\in[0,1]$:
\begin{align*}
\psi_f(s) &= \displaystyle {1\over\Gamma({n+m\over p}+1)}{1\over (1-s^p)^{{n+m\over p}+1}}\bigg[\int_{(0,\infty)} w^{n+m\over p}e^{-{s^p\over 1-s^p}w}\,\textbf{G}(\alpha,1)(\dint w)\bigg]\\
&= \displaystyle {1\over  \Gamma(\alpha)\Gamma({n+m\over p}+1)}{1\over (1-s^p)^{{n+m\over p}+1}}\bigg[\int_{(0,\infty)}   w^{(\alpha + {n+m\over p})-1}e^{-{1\over 1-s^p}w}  \, \dint w \bigg]\\
&=\displaystyle {\Gamma(\alpha + {n+m\over p})\over \Gamma(\alpha)\Gamma({n+m\over p}+1)}  (1-s^p)^{\alpha -1}.
\end{align*}
\end{example}

\begin{rmk}\label{rmk:NormPropRepExpGamma}
In \cite[Lemma 4.2]{KPTEnsembles} (and Lemma \ref{lem:DistributionConeMeasure} (ii)) we have seen a different probabilistic representation for the uniform distribution $\bU_{n,p,f, \ssbp}$ to that in Example \ref{ex:Beta}, namely $U^{1\over n+m}{X\over\|X\|_p}$, where $U$ is uniformly distributed on $[0,1]$ and independent of $X$. However, these two representations are equivalent. Indeed, since both are $p$-radially symmetric, it is sufficient to prove that the distributions of the $p$-norms of the random variables $U^{1\over n+m}{X\over\|X\|_p}$ and ${X\over(\|X\|_p^p+W)^{1/p}}$ with $W \sim \bW = \textbf{G}(\alpha,1)$, $\alpha =1$, are the same. For this we start by noticing that
\begin{align*}
\Pro(\|X\|_p^p\leq t) = \Pro(\|X\|_p\leq t^{1/p}) = C_{n,p,f}\int_{\{x\in\R^n:\|x\|_p\leq t^{1/p}\}}  e^{-\|x\|_p^p}f(x) \, \dint x.
\end{align*}
Using the polar integration formula \eqref{eq:PolarIngeration}, the fact that $f$ is homogeneous of degree $m$, and the substitution $s=r^p$, we deduce that
\begin{align*}
\Pro(\|X\|_p^p\leq t) &= C_{n,p,f}\, n\, \vol_n(\B_p^n)\int_0^{t^{1/p}} r^{n-1}\,e^{-r^p}\int_{\SSS_p^{n-1}} f(ry)\, \bC_{n,p}(\dint y)\, \dint r\\
&=C_{n,p,f}\,n\,\vol_n(\B_p^n)\int_0^{t^{1/p}} r^{n+m-1}\,e^{-r^p}\int_{\SSS_p^{n-1}} f(y)\,\bC_{n,p}(\dint y) \, \dint r\\
&= \frac{1}{\Gamma\Big({n+m\over p}\Big)}\,\int_0^t  s^{{n+m\over p}-1} \, e^{-s}\, \dint s.
\end{align*}
This proves that $\|X\|_p^p\sim \textbf{G}({n+m \over p}, 1)$. By %$W \sim \bW = \bE(1) = \textbf{G}(1,1)$ and
the well-known relation between the gamma and the beta distribution, this implies that 
$$
{\|X\|_p^p\over \|X\|_p^p+W} \sim \bB\Big({n+m\over p},1\Big).
$$
The proof is completed by noting that $U^{p\over n+m}$ follows precisely the same beta distribution. Note that taking the $p$-norm of $X\over(\|X\|_p^p+W)^{1/p}$ for $\bW=\textbf{G}(\alpha,1)$ for some $\alpha>0$, by the same arguments, yields 
\begin{equation*}
{\|X\|_p^p\over \|X\|_p^p+W} \sim \bB\Big({n+m\over p},\alpha\Big).
\end{equation*}
By analogue arguments the same holds for $X \sim \bU_{n,p,f, +}$.
\end{rmk}

Choosing $\bW$ to be a gamma distribution in the distribution $\bP_{n,p,\bW,f, \ssbp}$ leaves us simply with $\bP_{n,p,\bW,f, \ssbp} = \Psi_f\bU_{n,p,f, \ssbp}$, as $\bW(\{0\})=\bG(a,b)(\{0\})=0$ for all $a,b>0$. So, in this case all probability mass is distributed within the interior of $\B^n_{p, \ssbp}$. But we are also interested in cases where a certain amount of probability mass remains at the boundary. For this we consider the mixture $\bP_{n,p,\bW,f, \ssbp}= \vartheta \bC_{n,p,f, \ssbp}+(1-\vartheta)\Psi_f\bU_{n,p,f, \ssbp}$ for $\vartheta\in[0,1]$, which is simply a convex combination of weighted cone probability measure and weighted uniform distribution. This will be the main class of distributions we will consider in Section \ref{sec:ApplicationEucl} and Section \ref{sec:ApplicationMat} below. In this context, the following two propositions will turn out to be useful. The first one shows that for a specific choice of $\bW$ the random vector from Theorem \ref{thm:ProbablisticRepresentation} generates the required distribution. The second deals with the $p$-norm of that random vector. 

\begin{proposition} \label{prop:ProbRepMixedDistributionConeUnif}\rm 
Let $\vartheta \in [0,1]$, $\alpha \in (0,\infty)$, and consider the probability measure $\bW = \vartheta\delta_0 + (1-\vartheta)\,\textbf{G}(\alpha,1)$. Other than that, we assume the setup of Theorem \ref{thm:ProbablisticRepresentation}. Then the random vector $X\over(\|X\|_p^p+W)^{1/p}$ generates the distribution $\bP_{n,p,\bW,f, \ssbp} = \vartheta \,\bC_{n,p,f, \ssbp}+(1-\vartheta)\Psi_f\bU_{n,p,f, \ssbp}$.
\end{proposition}
\begin{proof}
Let $h:\R^n\to\R$ be a non-negative measurable function. Then, following the arguments in the proof of Theorem \ref{thm:ProbablisticRepresentation}, and using the results from Lemma \ref{lem:DistributionConeMeasure} (i) and Example \ref{ex:Beta}, we get
\begin{align*}
\nonumber\bE h\Big({X\over (\|X\|_p^p+W)^{1/p}}\Big) &= \int_{[0,\infty)} \bE h\Big({X\over (\|X\|_p^p+w)^{1/p}}\Big) \bW(\dint w)\\
&=  \bE h\Big({X\over \|X\|_p}\Big) \bW(\{0\}) + \int_{(0,\infty)} \bE h\Big({X\over (\|X\|_p^p+w)^{1/p}}\Big) \bW(\dint w)\\
&= \vartheta\int_{\SSS^{n-1}_p}h(x) \, \bC_{n,p,f}(\dint x)+  (1-\vartheta)\int_{(0,\infty)} \bE h\Big({X\over (\|X\|_p^p+w)^{1/p}}\Big) \,\textbf{G}(\alpha,1)(\dint w)\\
&= \vartheta\int_{\SSS^{n-1}_p}h(x) \, \bC_{n,p,f}(\dint x)+  (1-\vartheta)\int_{\B^n_p} h(x) \, \Psi_f\bU_{n,p,f}(\dint x).
\end{align*}
This completes the proof.
\end{proof}

\begin{proposition} \label{prop:DistributionNormB}\rm 
We assume the same setup as in Theorem \ref{thm:ProbablisticRepresentation} for the specific choice $\bW = \vartheta\delta_0 + (1-\vartheta)\,\textbf{G}(\alpha,1)$, where $\vartheta \in [0,1]$ and $\alpha \in (0,\infty)$. Then the random variable $ B := {\|X\|_p^p\over \|X\|_p^p+W}$ has distribution $\vartheta\delta_1 + (1-\vartheta)\bB\big({n+m \over p}, \alpha\big)$.
\end{proposition}
\begin{proof} Let $A \subset \R$ be a Borel set. Then, by the same arguments as in Remark \ref{rmk:NormPropRepExpGamma}, we get
\begin{center}
$
\begin{array}{lcl}
\displaystyle \Pro(B \in A) &=& \displaystyle\Pro\Big( {\|X\|_p^p\over \|X\|_p^p+W} \in A\Big)= \displaystyle \int\limits_{[0,\infty)} \Pro\Big( {\|X\|_p^p\over \|X\|_p^p+w} \in A\Big) \bW(\dint w)\\
&=& \displaystyle \Pro(1 \in A)  \, \vartheta  \delta_{0}(\{0\}) + \int\limits_{(0,\infty)} \Pro\Big( {\|X\|_p^p\over \|X\|_p^p+w} \in A\Big) (1-\vartheta)\,\textbf{G}(\alpha,1)(\dint w)\\
&=& \displaystyle \vartheta \delta_1(A) + (1-\vartheta)\,\bB\Big({n+m \over p}, \alpha\Big)(A).
\end{array}
$
\end{center}
The proof is thus complete.
\end{proof}

\section{Eigen- and singular value distributions on $\B_{p,\beta}^{n,{\mathscr{H}}}$ and $\B_{p,\beta}^{n,{\mathscr{M}}}$}\label{sec:EigenSingularDistr}

\subsection{Eigenvalue distribution for self-adjoint random matrices in matrix $p$-balls}\label{subsec:EigenvalueDistr}
After having studied the Euclidean case, we now turn to the eigenvalue distributions for self-adjoint random matrices in matrix $p$-balls. The following theorem shows how the distribution $\bP_{n,p,\bW,\beta}^{\mathscr{H}}$ in matrix $p$-balls is connected to the weighted $p$-radial distribution $\bP_{n,p,\bW,f}$ in Euclidean $\ell_p^n$-balls studied above, and generalizes the probabilistic representation in \cite[Corollary 4.3]{KPTEnsembles}, using a similar method of proof to do so, based on polar integration and the Weyl integration formula.\\
\\
Before we proceed with our main result, let us present the aforementioned tool we will need during its proof: the Weyl integration formula for $\mathscr H_n(\mathbb{F}_\beta)$, see \cite[Proposition 4.1.1]{AGZ2010} and also \cite[Proposition 4.1.14]{AGZ2010}. It states that for any non-negative measurable function $f:\mathscr H_n(\mathbb{F}_\beta)\to [0,\infty)$, such that $f(A)$ only depends on the eigenvalues of $A$, we have that
\begin{equation}\label{eq:WeylIntegration}
\int_{\mathscr H_n(\mathbb{F}_\beta)} f(A) \, \vol_{\beta,n}(\dint A) = c_{n,\beta}^{\mathscr{H}}\int_{\R^n} f(\lambda)\prod_{1\leq i<j\leq n}|\lambda_i-\lambda_j|^\beta \, \dint \lambda,
\end{equation}
where for every $\lambda=(\lambda_1,\ldots,\lambda_n)\in\R^n$ we write $f(\lambda)=f(A)$ for any matrix $A\in \mathscr H_n(\mathbb{F}_\beta)$ with (unordered) eigenvalues $\lambda_1, \ldots,\lambda_n$, and the constant $c_{n,\beta}^{\mathscr H}$ is given by
$$
c_{n,\beta}^{\mathscr{H}} := {1\over n!}\bigg({2\pi^{\beta/2}\over\Gamma({\beta\over 2})}\bigg)^{-n}\,{\prod\limits_{k=1}^n{2(2\pi)^{\beta k/2}\over 2^{\beta/2}\Gamma({\beta k\over 2})}}.
$$
To distinguish between the distributions of random eigenvalues in the standard increasing order and in unordered form, we will use the following version of the Weyl integration formula
\begin{equation}\label{eq:WeylIntegrationOrdered}
\int_{\mathscr H_n(\mathbb{F}_\beta)} f(A) \, \vol_{\beta,n}(\dint A) = n! \, c_{n,\beta}^{\mathscr{H}}\int_{\R^n} f(\lambda)\prod_{1\leq i<j\leq n}|\lambda_i-\lambda_j|^\beta \,{\bf 1}_{\{x\in\R^n:x_1\le\ldots \le x_n\}}(\lambda) \,  \dint\lambda.
\end{equation}
We do so to carry out most of the proof of the main theorem of this section in the more canonical increasingly ordered setting, so we only need to apply an appropriate permutation argument at the very end.
Our next result is derived by application of Weyl's integration formula in connection with the polar integration formula. In the case that $\bW(\{0\})=0$ it could also be deduced from a classical formula in \cite{R1984}, which is essentially based on the same ingredients, see also \cite[Lemma 4.3.1]{IsotropicConvexBodies}. We present a detailed argument for completeness. The following functions and normalization terms are needed for said result: For $x\in\R^n$, set 
$$ \Delta_\beta(x):= \prod_{1\leq i<j\leq n}|x_i-x_j|^\beta,$$
which is the repulsion factor of the eigenvalues of a random matrix given by the Weyl integration formula \eqref{eq:WeylIntegration}. Additionally, in the spirit of \eqref{eq:NormConst}, define a constrant $C_{n,p,\Delta_\beta}$ such that
\begin{equation*} \label{eq:NormConstEV1}
C_{n,p,\Delta_\beta}\int_{\R^n} \Delta_\beta(x)\,e^{-\|x\|_p^p} \,\dint x = 1.
\end{equation*}
Further, we define the function  $\Delta_\beta^c(x) := C_{\Delta_\beta} \,\Delta_\beta(x)$ with a more elaborate normalization factor
$$C_{\Delta_\beta} := {c_{n,\beta}^{\mathscr{H}} \over \vol_{\beta,n}\big(\B_{p,\beta}^{n, \mathscr{H}} \big) \,C_{n,p,\Delta_\beta} \,\Gamma\Big({n+m\over p} +1\Big)},$$
where $m=\frac{1}{2}\beta n(n-1)$ is the degree of homogeneity of $\Delta_\beta(x)$. Lastly, we define annother normalization constant $C_{n,p,\Delta_\beta^c}$ in the spirit of \eqref{eq:NormConst} satisfying
\begin{equation*} \label{eq:NormConstEV2}
C_{n,p,\Delta_\beta^c}\int_{\R^n} \Delta_\beta^c(x)\,e^{-\|x\|_p^p} \,\dint x = 1.
\end{equation*}
\begin{thm}\label{thm:EvDistr}
Let $0<p < \infty$, $\beta\in\{1,2,4\}$ and $\bW$ be a Borel probability measure on $[0,\infty)$. Let $W$ be a real random variable with distribution $\bW$ and, independently of $W$, $X$ be a random vector with density $C_{n,p,\Delta_\beta^c} \,e^{-\|x\|_p^p} \, \Delta_\beta^c(x)$, $x\in\R^n$, with respect to the Lebesgue measure. %
%, where $\Delta_\beta^c(x) := C_{\Delta_\beta} \,\Delta_\beta(x),$
%with 
%$$ \Delta_\beta(x):= \prod_{1\leq i<j\leq n}|x_i-x_j|^\beta,\qquad x=(x_1,\ldots,x_n)\in\R^n,$$
%and normalization factor
%$$C_{\Delta_\beta} := {c_{n,\beta}^{\mathscr{H}} \over \vol_{\beta,n}\big(\B_{p,\beta}^{n, \mathscr{H}} \big) \,C_{n,p,\Delta_\beta} \,\Gamma\Big({n+m\over p} +1\Big)}, $$
%
%where $m=\frac{\beta n(n-1)}{2}$ is the degree of homogeneity of $\Delta_\beta^c(x)$ and $C_{n,p,\Delta_\beta^c}, C_{n,p,\Delta_\beta}$ as in \eqref{eq:NormConst}. 
Let $Z$ be a random matrix with distribution $\bP_{n,p,\bW,\beta}^{\mathscr{H}} :=\bW(\{0\})\bC_{n,p,\beta}^{\mathscr{H}}+\Psi^{\mathscr{H}}\bU_{n,p,\beta}^{\mathscr{H}}$ on $\B_{p,\beta}^{n, \mathscr{H}}$, where $\Psi^{\mathscr{H}}(A) := \Psi_{\Delta^c_\beta}(\lambda(A)) = \psi_{\Delta_\beta^c}(\|\lambda(A)\|_p)$ for $A\in \B_{p,\beta}^{n,\mathscr{H}}$, and $\psi_{\Delta_\beta^c}$ is defined as in Theorem \ref{thm:ProbablisticRepresentation} for $f=\Delta_\beta^c$. Independently, let $\sigma$ be a uniform random permutation in the symmetric group on $n$ elements. Then 
$$
\lambda_\sigma(Z):=\big(\lambda_{\sigma(1)}(Z),\ldots,\lambda_{\sigma(n)}(Z)\big) \quad\text{and}\quad{X\over(\|X\|_p^p+W)^{1/p}}
$$
are identically distributed with distribution $\bP_{n,p,\bW,\Delta_\beta^c} :=\bW(\{0\})\bC_{n,p,\Delta_\beta^c}+\Psi_{\Delta_\beta^c}\bU_{n,p,\Delta_\beta^c}$ on $\B_p^n$.
\end{thm}

\begin{rmk}\text{}
	\begin{itemize}
		\item[(i)]{If $\bW$ is the Dirac measure at 0 (that is, $\bP_{n,p,\bW,\beta}^{\mathscr{H}} =\bC_{n,p,\beta}^{\mathscr{H}}$) or the exponential distribution with parameter $1$ (that is, $\bP_{n,p,\bW,\beta}^{\mathscr{H}} =\bU_{n,p,\beta}^{\mathscr{H}}$)  the result was previously obtained in \cite{KPTEnsembles}.}
		\item[(ii)]{We can see that the current definition $\Psi^{\mathscr{H}}(A) := \psi_{\Delta_\beta^c}(\|\lambda(A)\|_p)$ coincides with that of $\Psi^{\mathscr{H}}(A)$ from \eqref{eq:PnpwSA}, as the degree of homogeneity $m=\frac{1}{2}\beta n(n-1)$ is the same.}
	\end{itemize}

\end{rmk}

\begin{proof}[Proof of Theorem \ref{thm:EvDistr}]
Let $h:\R^n\to\R$ be a non-negative measurable function and $\tilde{h}: \mathscr H_n(\mathbb{F}_\beta) \to \R$ given by $\tilde{h}(A):=h(\lambda(A))$. %
% for $A \in \mathscr H_n(\mathbb{F}_\beta)$. 
We now want to compute $\bE \tilde{h}(Z)$:
\begin{align}
\nonumber \bE \tilde{h}(Z) &= \int_{\mathscr H_n(\mathbb{F}_\beta)} \tilde{h}(A)\,\big(\bW(\{0\})\bC_{n,p, \beta}^{\mathscr{H}} + \Psi^{\mathscr{H}} \bU_{n,p, \beta}^{\mathscr{H}} \big) (\dint A)\\
 &=\bW(\{0\})\int_{\SSS_{p,\beta}^{n-1, \mathscr{H}} } h(\lambda(A))\,\bC_{n,p, \beta}^{\mathscr{H}}(\dint A) + \int_{\B_{p,\beta}^{n, \mathscr{H}}} h(\lambda(A))\, \Psi^{\mathscr{H}} \bU_{n,p, \beta}^{\mathscr{H}}(\dint A). \label{eq:ProofMatPropRep}
\end{align}
Consider the radial extension $h\big({ \lambda(A) / \|\lambda(A)\|_p}\big)$, $A \in \mathscr H_n(\mathbb{F}_\beta)$, of $h(\lambda(A))$ from $\SSS_{p,\beta}^{n-1, \mathscr{H}}$ onto $\mathscr H_n(\mathbb{F}_\beta)$. By Remark \ref{rmk:starsharped} we can apply the polar integration formula from Lemma \ref{lem:PolarIngerationAllg} to $\B_{p,\beta}^{n, \mathscr{H}}$ to get
\begin{eqnarray*}\label{eq:polarext}
&&\nonumber \int_{\B_{p,\beta}^{n, \mathscr{H}}} h\Big({ \lambda(A) \over \|\lambda(A)\|_p}\Big) \bU_{n,p, \beta}^{\mathscr{H}}(\dint A)\\
&=& \Big(\frac{\beta n (n-1)}{2} + \beta n\Big) \int_0^1 r^{\frac{\beta n (n-1)}{2} + \beta n-1} \int_{\SSS_{p,\beta}^{n-1, \mathscr{H}}} h(\lambda(A)) \,\bC_{n,p, \beta}^{\mathscr{H}}(\dint A) \dint r\\
&=& \int_{\SSS_{p,\beta}^{n-1, \mathscr{H}}} h(\lambda(A)) \, \bC_{n,p, \beta}^{\mathscr{H}}(\dint A). 
\end{eqnarray*}
With \eqref{eq:polarext} it follows that \eqref{eq:ProofMatPropRep} can be rewritten as:
\begin{align*}
\nonumber \bE \tilde{h}(Z) =&\bW(\{0\})\, \vol_{\beta,n}\big(\B_{p,\beta}^{n, \mathscr{H}} \big)^{-1} \int_{\B_{p,\beta}^{n, \mathscr{H}}} h\Big({\lambda(A)\over \|\lambda(A)\|_p}\Big) \vol_{\beta,n}(\dint A)\\
& + \vol_{\beta,n}\big(\B_{p,\beta}^{n, \mathscr{H}} \big)^{-1} \int_{\B_{p,\beta}^{n, \mathscr{H}}} h(\lambda(A))\Psi^{\mathscr{H}}(A) \vol_{\beta,n}(\dint A). 
\end{align*}
To both of those terms on the right-hand side we can now apply the \enquote{ordered} Weyl integration formula \eqref{eq:WeylIntegrationOrdered} with respect to the functions $f_1(A) = h\big({\lambda(A) / \|\lambda(A)\|_p}\big)$ and $f_2(A) = h(\lambda(A))\,\Psi^{\mathscr{H}}(\lambda(A))$. We use the fact that $\Delta_\beta \in \mathscr{F}_m^+(\R^n)$ with $m =\frac{1}{2}\beta n(n-1)$ and $\Psi^{\mathscr{H}}= \Psi_{\Delta_\beta}= \Psi_{\Delta_\beta^c}$, to see that 
\begin{align*}
\bE \tilde{h}(Z) &=  { n! \, c_{n,\beta}^{\mathscr{H}} \over \vol_{\beta,n}\big(\B_{p,\beta}^{n, \mathscr{H}} \big)} \quad \bW(\{0\})\int_{\B^n_p} h\Big({ \lambda \over \|\lambda\|_p}\Big)\,\Delta_\beta(\lambda)\,{\bf 1}_{\{x\in\R^n:x_1 \le \ldots \le x_n\}}(\lambda)\, \dint\lambda\\
& \qquad+  {n! \, c_{n,\beta}^{\mathscr{H}} \over \vol_{\beta,n}\big(\B_{p,\beta}^{n, \mathscr{H}} \big)} \quad\int_{\B_p^n}h(\lambda)\,\Psi_{\Delta_\beta^c}(\lambda)\,\Delta_\beta(\lambda)\,{\bf 1}_{\{x\in\R^n:x_1 \le \ldots \le x_n\}}(\lambda)\, \dint\lambda\\
&= { n! \, c_{n,\beta}^{\mathscr{H}} \over \vol_{\beta,n}\big(\B_{p,\beta}^{n, \mathscr{H}} \big)} \quad \bW(\{0\})\int_{\B^n_p} h\Big({ \lambda \over \|\lambda\|_p}\Big)\,\Delta_\beta\Big({\lambda \over \|\lambda\|_p}\Big)  \|\lambda\|_p^m \,{\bf 1}_{\{x\in\R^n:x_1 \le \ldots \le x_n\}}(\lambda)\, \dint\lambda\\
&\qquad +  { n! \, c_{n,\beta}^{\mathscr{H}} \over \vol_{\beta,n}\big(\B_{p,\beta}^{n, \mathscr{H}} \big)} \quad\int_{\B_p^n}h(\lambda)\,\Psi_{\Delta_\beta^c}(\lambda)\,\Delta_\beta(\lambda)\,{\bf 1}_{\{x\in\R^n:x_1 \le \ldots \le x_n\}}(\lambda)\, \dint\lambda. 
\end{align*}
Applying now the polar integration formula from Lemma \ref{lem:PolarIngerationAllg}, we conclude that the last expression is equal to 
\begin{align*}
&{n! \, c_{n,\beta}^{\mathscr{H}} \,n\,  \vol_{n}(\B^n_p) \over \vol_{\beta,n}\big(\B_{p,\beta}^{n, \mathscr{H}} \big)} \quad \bW(\{0\}) \int\limits_0^1 r^{n+m-1} \dint r \int_{\SSS_p^{n-1}} h(\lambda)\,\Delta_\beta(\lambda)\,{\bf 1}_{\{x\in\R^n:x_1 \le \ldots \le x_n\}}(\lambda)\, \bC_{n,p}(\dint \lambda) \\
&\qquad +  {n! \, c_{n,\beta}^{\mathscr{H}} \over \vol_{\beta,n}\big(\B_{p,\beta}^{n, \mathscr{H}} \big)} \quad\int_{\B_p^n}h(\lambda)\,\Psi_{\Delta_\beta^c}(\lambda)\,\Delta_\beta(\lambda)\,{\bf 1}_{\{x\in\R^n:x_1 \le \ldots \le x_n\}}(\lambda)\, \dint \lambda \\
&= { n! \, c_{n,\beta}^{\mathscr{H}} \,n\,  \vol_{n}(\B^n_p) \over \vol_{\beta,n}\big(\B_{p,\beta}^{n, \mathscr{H}} \big) (n+m)} \quad \bW(\{0\}) \int_{\SSS_p^{n-1}} h(\lambda)\,\Delta_\beta(\lambda)\,{\bf 1}_{\{x\in\R^n:x_1 \le \ldots \le x_n\}}(\lambda)\, \bC_{n,p}(\dint \lambda) \\
&\qquad + \quad { n! \, c_{n,\beta}^{\mathscr{H}} \, \vol_{n}(\B^n_p) \over \vol_{\beta,n}\big(\B_{p,\beta}^{n, \mathscr{H}} \big)} \quad\int_{\B_p^n}h(\lambda)\,\Psi_{\Delta_\beta^c}(\lambda)\,\Delta_\beta(\lambda)\,{\bf 1}_{\{x\in\R^n:x_1 \le \ldots \le x_n\}}(\lambda)\, \bU_{n,p}(\dint \lambda).
\end{align*}
Next, we use the definition of $\bU_{n,p, f}$ and $\bC_{n,p, f}$ for $f=\Delta_\beta$ and the definition of $\Delta_\beta^c = C_{\Delta_\beta} \, \Delta_\beta$. This gives
\begin{align*}
 \bE \tilde{h}(Z) &= \displaystyle { n! \, c_{n,\beta}^{\mathscr{H}} \over \vol_{\beta,n}\big(\B_{p,\beta}^{n, \mathscr{H}} \big) \,C_{n,p,\Delta_\beta} \,\Gamma\Big({n+m\over p} +1\Big)}  \bW(\{0\})\int_{\SSS_p^{n-1}} h(\lambda)\,{\bf 1}_{\{x\in\R^n:x_1 \le \ldots \le x_n\}}(\lambda)\, \bC_{n,p,\Delta_\beta}(\dint\lambda)\\
&\qquad +  { n! \, c_{n,\beta}^{\mathscr{H}} \over \vol_{\beta,n}\big(\B_{p,\beta}^{n, \mathscr{H}} \big) \,C_{n,p,\Delta_\beta} \,\Gamma\Big({n+m\over p} +1\Big)}  \int_{\B_p^n}h(\lambda)\,\Psi_{\Delta_\beta^c}(\lambda)\,{\bf 1}_{\{x\in\R^n:x_1 \le \ldots \le x_n\}}(\lambda)\, \bU_{n,p,\Delta_\beta}(\dint\lambda)\\
&= n!\, \bW(\{0\})\int_{\SSS_p^{n-1}} h(\lambda)\,{\bf 1}_{\{x\in\R^n:x_1 \le \ldots \le x_n\}}(\lambda)\, \bC_{n,p,\Delta_\beta^c}(\dint\lambda)\\
& \qquad+ n!\,  \int_{\B_p^n}h(\lambda)\,\Psi_{\Delta_\beta^c}(\lambda)\,{\bf 1}_{\{x\in\R^n:x_1 \le \ldots \le x_n\}}(\lambda)\, \bU_{n,p,\Delta_\beta^c}(\dint\lambda). 
\end{align*}
As a consequence, when applying a uniform random permutation $\sigma \in \mathfrak{S}(n)$, we get by Theorem \ref{thm:ProbablisticRepresentation}
\begin{align*}
\bE (\tilde{h}\circ\sigma)(Z) &= \bW(\{0\})\int_{\SSS_p^{n-1}} h( \lambda  )\, \bC_{n,p,\Delta_\beta^c}(\dint\lambda)  + \int_{\B_p^n} h(  \lambda )\,\Psi_{\Delta_\beta^c}(\lambda)\,\bU_{n,p,\Delta_\beta^c}(\dint\lambda) \\
&= \bE h\Big({X\over(\|X\|_p^p+W)^{1/p}}\Big).
\end{align*}
This proves the claim.
\end{proof}
\begin{rmk}\label{rmk:NormDistrSA}\rm 
For $\Delta^c_\beta$ the degree of homogeneity is $m=\frac{\beta n(n-1)}{2}$. Thus, if $\bW=\bG(\alpha,1)$ for $\alpha>0$, by the results outlined in Remark \ref{rmk:NormPropRepExpGamma}, we have that
$$
{\|X\|_p^p\over \|X\|_p^p+W} \sim \bB\Big( \, \frac{\beta n^2}{2p} - \frac{\beta n}{2p} + \frac{n}{p} \, , \, \alpha \, \Big),
$$
and for $\bW = \vartheta\delta_0 + (1-\vartheta)\,\textbf{G}(\alpha,1)$, where $\vartheta \in [0,1]$ and $\alpha \in (0,\infty)$, we have by the arguments from Proposition \ref{prop:DistributionNormB} that
$$
{\|X\|_p^p\over \|X\|_p^p+W} \sim \vartheta\delta_1 + (1-\vartheta)\bB\Big( \, \frac{\beta n^2}{2p} - \frac{\beta n}{2p} + \frac{n}{p} \, , \, \alpha \, \Big).
$$
\end{rmk}

\subsection{Singular value distribution for non-self-adjoint random matrices in matrix $p$-balls}\label{subsec:SingularDistr}

Let us now consider the non-self-adjoint case, where the singular values take over the role of the eigenvalues. The following result is proven by almost literally repeating the proof of Theorem \ref{thm:EvDistr} (or, at least in the case that $\bW(\{0\})=0$, by applying a formula from \cite{R1984}, which corresponds to \cite[Lemma 4.3.1]{IsotropicConvexBodies} as we explained before Theorem \ref{thm:EvDistr}). However, this time the argument is based on the Weyl-type integration formula from \cite[Proposition 4.1.3]{AGZ2010}, which replaces \eqref{eq:WeylIntegration}. Proposition \cite[Proposition 4.1.3]{AGZ2010} primarily changes the repulsion factor from $\Delta_\beta$ to an appropriate $\nabla_\beta$ and the normalization constant from $c_{n,\beta}^{\mathscr{H}} $ to $c_{n,\beta}^{\mathscr{M}}$ as follows:
It says that for any non-negative measurable function $f:\mathscr M_n(\mathbb{F}_\beta)\to[0,\infty)$, such that $f(A)$ only depends on the singular values of $A$, we have that
\begin{equation}\label{eq:WeylIntegration2}
\int_{\mathscr M_n(\mathbb{F}_\beta)} f(A) \, \vol_{\beta,n}(\dint A) = c_{n,\beta}^{\mathscr{M}}\int_{\R_+^n} f(s)\prod_{1\leq i<j\leq n}|s^2_i-s^2_j|^\beta \, \prod_{i=1}^n{s_i}^{\beta-1}\, \dint s,
\end{equation}
writing $f(s)=f(A)$ for any matrix $A\in \mathscr M_n(\mathbb{F}_\beta)$ with (unordered) singular values $(s_1, \ldots,s_n) \in \R_+^n$, and where
$$
c_{n,\beta}^{\mathscr{M}} := {1\over n!}\frac{1}{2^{\frac{\beta}{2} n(n-1)}}\bigg({2\pi^{\beta/2}\over\Gamma({\beta\over 2})}\bigg)^{-n}\,{\prod\limits_{k=1}^n \bigg({2(2\pi)^{\beta k/2}\over 2^{\beta/2}\Gamma({\beta k\over 2})}\bigg)^2}.
$$
Again, we derive from this an \enquote{ordered version} of the Weyl integration formula to shift the necessity for permutations to the end of the proof, for which we additionally apply a useful change of variable: 
\begin{eqnarray}\label{eq:WeylIntegration2Ordered}
\nonumber &&\int_{\mathscr M_n(\mathbb{F}_\beta)} f(A) \, \vol_{\beta,n}(\dint A) \\
&&\qquad \qquad = n! \, c_{n,\beta}^{\mathscr{M}} \, 2^{-n} \int_{\R^n_+} f(s)\prod_{1\leq i<j\leq n}|s_i^2-s_j^2|^\beta \, \prod_{i=1}^n{s_i}^{\frac{\beta}{2}-1} \, {\bf 1}_{\{x\in\R^n_+:x_1\le\ldots \le x_n\}}(s^2) \,  \dint s^2.
\end{eqnarray}
As discussed in Section \ref{subsec:GeometryMatrix}, the vector $s(A):=(s_1(A), \ldots, s_n(A))$ of singular values of a matrix $A \in \mathscr M_n(\mathbb{F}_\beta)$ lives in the non-negative orthant $\B^n_{p,+}$ of the $\ell_p^n$-ball $\B^n_p$. Furthermore, the matrix $p$-ball $\B_{p,\beta}^{n,{\mathscr{M}}}$ will be represented in Euclidean space via $\B^n_{p/2, +}$, not $\B^n_{p, +}$, due to the structure of the Weyl integration formula for singular values in \eqref{eq:WeylIntegration2Ordered}. Since it uses the squares of the singular values in its repulsion factor, we adapt our representation appropriately, such that the same arguments as for the eigenvalues are applicable. Thus, we reformulate the defining condition of $\B_{p,\beta}^{n,{\mathscr{M}}}$ from $\sum_{i=1}^n|s_i(A)|^p \le 1$ to  $\sum_{i=1}^n|s_i^2(A)|^{p/2} \le 1$, and apply the same arguments as before to the vector $s^2(A):=(s_1^2(A), \ldots, s_n^2(A))$, which then in turn lies in $\B^n_{p/2,+}$. %
As in the self-adjoint setting, we need to define some functions and normalization terms to formulate the next result. For $x \in \R^n_+$ we set 
$$
\nabla_\beta(x) := \prod_{1\leq i<j\leq n}|x_i-x_j|^\beta\prod_{i=1}^nx_i^{{\beta\over 2}-1},
$$
which again is the repulsion factor of singular values from the Weyl integration formula \eqref{eq:WeylIntegration2}, and define $C_{n,p,\nabla_\beta, +}$ to be the normalization constant such that
\begin{equation*} \label{eq:NormConstSV1}
C_{n,p,\nabla_\beta}\int_{\R^n_+} \nabla_\beta(x)\,e^{-\|x\|_p^p} \,\dint x = 1.
\end{equation*}
Based on this definition, we further set $\nabla^c_\beta(x) := C_{\nabla_\beta}\,\nabla_\beta(x)$ for $x \in \R^n_+$ with 
$$C_{\nabla_\beta} := {c_{n,\beta}^{\mathscr{M}} \over \vol_{\beta,n}\big(\B_{p,\beta}^{n, \mathscr{M}} \big) \,C_{n,p/2,\nabla_\beta} \,\Gamma\Big({n+m\over p/2} +1\Big) 2^n },$$ 
where $m=\frac{\beta }{2}n^2-n$ is the degree of homogeneity of $\nabla_\beta^c(x)$. A final normalization constant $C_{n,p,\nabla_\beta^c, +}$ is defined by
\begin{equation*} \label{eq:NormConstSV2}
C_{n,p,\nabla_\beta^c}\int_{\R^n_+} \nabla_\beta^c(x)\,e^{-\|x\|_p^p} \,\dint x = 1.
\end{equation*}
\begin{thm}\label{thm:SvDistr}
Let $0<p<\infty$, $\beta\in\{1,2,4\}$ and $\bW$ be a Borel probability measure on $[0,\infty)$. Let $W$ be a real random variable with density $\bW$ and, independently of $W$, $X$ be a random vector with distribution given by the density $C_{n,{p / 2},\nabla^c_\beta, +} \, e^{-\|x\|_{p / 2}^{p / 2}} \, \nabla^c_\beta(x)$, $x\in\R^n_{+}$,  with respect to the Lebesgue measure. %
%, where $\nabla^c_\beta(x) := C_{\nabla_\beta}\,\nabla_\beta(x),$
%with
%$$
%\nabla_\beta(x) := \prod_{1\leq i<j\leq n}|x_i-x_j|^\beta\prod_{i=1}^nx_i^{{\beta\over 2}-1},\qquad x=(x_1,\ldots,x_n)\in\R^n_{+},
%$$
%and normalization factor
%
%$$C_{\nabla_\beta} := {c_{n,\beta}^{\mathscr{M}} \over \vol_{\beta,n}\big(\B_{p,\beta}^{n, \mathscr{M}} \big) \,C_{n,p/2,\nabla_\beta} \,\Gamma\Big({n+m\over p/2} +1\Big) },$$ 
%
%where $m=\frac{\beta }{2}n^2-n$ is the degree of homogeneity of $\nabla_\beta^c(x)$ and $C_{n,p,\nabla_\beta^c, +}, C_{n,p,\nabla_\beta, +}$ as in \eqref{eq:NormConst}, but for $\R^n_+$. 
Let $\sigma$ be a uniform random permutation in the symmetric group on $n$ elements and $Z$ be a random matrix with distribution $\bP_{n,p,\bW,\beta}^{\mathscr{M}} :=\bW(\{0\})\bC_{n,p,\beta}^{\mathscr{M}}+\Psi^{\mathscr{M}}\bU_{n,p,\beta}^{\mathscr{M}}$ on $\B_{p,\beta}^{n,{\mathscr{M}}}$, where $\Psi^{\mathscr{M}}(A) := \Psi_{\nabla^c_\beta}(s(A)) = \psi_{\nabla^c_\beta}(\|s(A)\|_{p})$ for $A\in \B_{p,\beta}^{n,\mathscr{M}}$, and $\psi_{\nabla^c_\beta}$ is defined as in Theorem \ref{thm:ProbablisticRepresentation} for $f=\nabla^c_\beta$. Then
$$
s^2_\sigma(Z):=\big(s_{\sigma(1)}^2(Z),\ldots,s_{\sigma(n)}^2(Z)\big)\qquad\text{and}\qquad{X\over(\|X\|_{p/2}^{p/2}+W)^{2/p}}
$$
are identically distributed with distribution $\bP_{n,{p/2},\bW,\nabla^c_\beta, +} = \bW(\{0\})\bC_{n, {p/ 2},\nabla^c_\beta, +}+\Psi_{\nabla^c_\beta}\bU_{n, {p/ 2},\nabla^c_\beta, +}$ on $\B_{p/2, +}^n$. 
\end{thm}
The proof of this goes along the very same lines as that of Theorem \ref{thm:EvDistr}, just using representation results from Theorem \ref{thm:ProbablisticRepresentation} in the non-negative setting and the Weyl integration formula from \eqref{eq:WeylIntegration2Ordered} instead of \eqref{eq:WeylIntegrationOrdered} in conjunction with an appropriate change of variables regarding the square variable, resulting in the repulsion factor $\nabla^c_\beta$.%\\
\begin{rmk}\label{rmk:NormDistrNSA}\rm 
For $\nabla^c_\beta$ the degree of homogeneity is $m=\frac{\beta }{2}n^2-n$. Thus, if $\bW=\bG(\alpha,1)$ for $\alpha>0$,  analogue arguments as in Remark \ref{rmk:NormPropRepExpGamma} for a random vector $X$ distributed on $\R^n_+$ as in Theorem \ref{thm:SvDistr} yield that
$$
{\|X\|_{p/2}^{p/2}\over \|X\|_{p/2}^{p/2}+W} \sim \bB\Big({\beta\over p} n^2, \alpha \Big),
$$
and for $\bW = \vartheta\delta_0 + (1-\vartheta)\,\textbf{G}(\alpha,1)$, where $\vartheta \in [0,1]$ and $\alpha \in (0,\infty)$, we have by the arguments from Proposition \ref{prop:DistributionNormB} that
$$
{\|X\|_{p/2}^{p/2}\over \|X\|_{p/2}^{p/2}+W}  \sim \vartheta\delta_1 + (1-\vartheta)\bB\Big({\beta\over p} n^2, \alpha \Big).
$$
\end{rmk}

\section{Application to large deviations:  Euclidean $\ell_p^n$-balls}\label{sec:ApplicationEucl}

\subsection{LDPs for the empirical measure of random vectors in $\B^n_p$} \label{subsec:LDPGeometric}

In \cite{KimRamanan} an LDP was derived for the empirical measure of the (suitably scaled) coordinates of a random vector that is distributed according to the cone probability measure on $\B^n_p$. In this section, we prove a similar large deviation principle with the random vectors chosen according to one of the more general distributions $\bP_{n,p,\bW}$. We restrict ourselves to the following situation: for each $n\in\N$ we consider $\bW_n:=\vartheta_n\delta_0 +(1-\vartheta_n)\textbf{G}(\alpha_n, 1)$ with $\vartheta_n\in[0,1]$ and $\alpha_n\geq 0$. This way, we are specific enough to compute a concrete rate function, yet broad enough to still encapsulate many interesting distributions for the corresponding $\bP_{n,p,\bW_n}$. As we will see, the large deviation behavior of the empirical measure will be dependent both on the limits $\lim_{n\to\infty} \vartheta_n =: \vartheta \in [0,1]$ and $\lim_{n\to\infty} \alpha_n/n =: \alpha \in [0,\infty)$ of the parameter sequences and their speed of convergence, and thus will be universal to all distributions who have the same parameter limits and parameter convergence speeds. We shall appropriately write $\Psi_{f,n}$ for the $p$-radial density associated with $\bW_n$ as defined in Theorem \ref{thm:ProbablisticRepresentation}. (However, the weighting function will not be needed in this section, i.e., can be set to $f \equiv 1$). %
For a probability measure $\mu \in \mathcal{M}(\R)$ we will denote by $$m_p(\mu):=\int_{\R}|x|^p\,\mu(\dint x) \in [0,\infty]$$ its $p$-th absolute moment if $p\in (0,\infty)$. We also define the relative entropy as
$$
H(\nu\|\mu) := 
\begin{cases}
\displaystyle \int_\R \log\frac{\nu(\dint x)}{\mu(\dint x)} \, \nu(\dint x) &:\nu\ll\mu\\
+\infty &: otherwise,
\end{cases}
$$
for $\mu, \nu \in \mathcal{M}(\R)$, where $\frac{\nu(\dint x)}{\mu(\dint x)}$ denotes the Radon-Nikod\'ym derivative of $\nu$ with respect to $\mu$. 
Finally, for a random vector $Z:=(Z_1, \ldots, Z_n) $ in $\R^n$ the empirical measure of its coordinates is defined as $\nu_n:={1\over n}\sum_{i=1}^n\delta_{Z_i}$. In the following result, the random vector $Z$ will have distribution $\bC_{n,p}$ on $\B^n_p$, thus we will consider the empirical measure of the coordinates scaled by the factor $n^{1 / p}$, i.e., $\mu_n:={1\over n}\sum_{i=1}^n\delta_{n^{1/p}Z_i}$. %
%
% The scaling is due to the fact that the length of a random vector with distribution $\bC_{n,p} $ is of order $n^{-{1/ p}}$, so the scaling is necessary to have non-trivial results (see \cite[Proposition 2.2]{KimRamanan}).
 %
 The scaling is necessary to receive non-trivial results and can be derived by the following reasoning. Since the defining condition of $\SSS^{n-1}_p$ restricts the $n$-fold sum of $p$-th powers of the coordinates of a random vector to be equal to one it follows that the typical coordinate of that vector must be of order $n^{-1/p}$, which the rescaling counteracts (see \cite[Proposition 2.2]{KimRamanan}).
 This will be the case for all other distributions on $\ell_p^n$-balls as well, as they all have $p$-radial components that are less or equal to that of the cone probability measure. We will often just call $\mu_n$ the empirical measure of a random vector $Z$. As mentioned in the introduction, Rachev and Rüschendorf \cite{RachevRueschendorf} showed that the (one dimensional) marginal distributions of $\bC_{n,p}$ asymptotically are generalized Gaussian distributions $\bN_p$  with expectation $0$, rate $1$ and shape $p$, thus the expectation of the $\mu_n$ is $\bN_p$. In \cite[Proposition 3.6]{KimRamanan} Kim and Ramanan derived the following Sanov-type LDP for the empirical measure of a random vector in $\B^n_p$ with distribution $\bC_{n,p}$.
\begin{proposition}\label{prop:LDPKimRamanan}
Let $0<p < \infty$ and let $(Z^{(n)})_{n\in\N}$ be a sequence of random vectors $Z^{(n)}=\big(Z^{(n)}_1, \ldots, Z_n^{(n)})$ in $\B^n_p$ with distribution $\bC_{n,p}$. Then the sequence of random probability measures $(\mu_n)_{n\in\N}$ with  $\mu_n:={1\over n}\sum\limits_{i=1}^n\delta_{n^{1/p}Z^{(n)}_i}$ satisfies a large deviation principle on $\mathcal{M}(\R)$ with speed $n$ and good rate function
$$
\mathcal{I}_{\textup{cone}}(\mu) = \begin{cases}
\displaystyle H(\mu\|\bN_p) + (1 -m_p(\mu)) &: m_p(\mu)\le1\\
+\infty &: otherwise,
\end{cases}
$$
where $\bN_p$ is  the generalized Gaussian measure with expectation $0$, rate $1$ and shape $p$.
\end{proposition}
\begin{rmk}\text{}
\begin{itemize}
\item[(i)] We remark that the original version of this result in \cite[Proposition 3.6]{KimRamanan} was only formulated for $p\in [1,\infty]$, but can be expanded to $p \in (0, \infty]$, as all the probabilistic representations used in the proof also hold for $p\in (0,1)$, and neither the convexity of $\B^n_p$ nor the norm-property of $\| \cdot \|_p$ was used in the proof. We exclude the case $p=\infty$ in this paper though, hence we only present results for $p\in(0,\infty)$.
\item[(ii)] As we already mentioned in Remark \ref{rmk:DifferentNormalizations}, the scale of the generalized Gaussian in \cite{KimRamanan} is $p^{1/p}$ instead of $1$. The rate function in Proposition \ref{prop:LDPKimRamanan} had to be adjusted to compensate for the different parametrization.
\item[(iii)] The above Sanov-type LDP for Euclidean $\ell_p^n$-balls of Kim and Ramanan \cite{KimRamanan} has been recently generalized to a Sanov-type LDP for Orlicz-balls by Frühwirth and Prochno in \cite{FProchnoSanovOrlicz}. Despite being proven differently, due to the lack of Schechtman-Zinn type probabilistic representations, their results still exhibit a similarity to those in \cite{KimRamanan} with the rate function of the LDP being given by a relative entropy term and a generalization of the moment penalty.
\end{itemize}
\end{rmk}
We now extend Proposition \ref{prop:LDPKimRamanan} to random vectors with distribution $\bP_{n,p,\bW_n}$ on $\B^n_p$. It will turn out, that the rate function will again be the relative entropy, this time perturbed by some more elaborate $p$-th moment penalty. %
\begin{thm}\label{thm:LDPEmpMeasure}
Let $0<p<\infty$ and let $(\vartheta_n)_{n\in\N}$ be a sequence in $[0,1]$ with $\lim\limits_{n\to\infty} \vartheta_n = \vartheta \in [0,1]$ and denote by $k(\vartheta) \ge 1$ the smallest number such that $\lim\limits_{n\to\infty} n^{-k(\vartheta)} \, |\log(1-\vartheta_n)| < +\infty$. Also let $(\alpha_n)_{n\in\N}$ be a positive, real sequence such that $\lim\limits_{n\to\infty}\alpha_n \, n^{-1}=\alpha \in [0,\infty)$. For each $n\in\N$ let $\bW_n=\vartheta_n\delta_0 +(1-\vartheta_n)\textbf{G}(\alpha_n, 1)$, and let $(Z^{(n)})_{n\in\N}$ be a sequence of random vectors $Z^{(n)}=\big(Z^{(n)}_1, \ldots, Z^{(n)}_n\big)$ in $\B^n_p$ chosen according to the distribution $\bP_{n,p,\bW_n}$. Then the sequence of random probability measures $(\mu_n)_{n\in\N}$ with  $\mu_n:={1\over n}\sum\limits_{i=1}^n\delta_{n^{1/p}Z^{(n)}_i}$ satisfies a large deviation principle on $\mathcal{M}(\R)$ with speed $n$ and good rate function
\begin{align*}
\mathcal{I}_{\textup{emp}}(\mu) = \begin{cases}
\displaystyle \mathcal{I}_{\textup{cone}}(\mu) - c_{(1-\vartheta)}&: \parbox{5cm}{$m_p(\mu)\le1, k(\vartheta) \ge 1, \alpha =0$}\\
\parbox{9cm}{$ \displaystyle  \mathcal{I}_{\rm cone}(\mu) +{1\over p}\log\Big({1\over p}\Big) - \Big({1\over p} + \alpha\Big) \log \Big({1\over p} + \alpha\Big) \vphantom{\int\limits_0^1}$\\$- \alpha \log \Big(\frac{1 -m_p(\mu)}{\alpha}\Big) - c_{(1-\vartheta)} \vphantom{\int\limits_0}$} &: \parbox{5cm}{$m_p(\mu)<1, k(\vartheta)=1,  \alpha>0$}\\
+\infty &: otherwise,
\end{cases}
\end{align*}
where $\mathcal{I}_{\textup{cone}}$ is the same as in Proposition \ref{prop:LDPKimRamanan} and 
$$ 
c_{(1-\vartheta)} := \begin{cases} 
\lim\limits_{n\to\infty} n^{-1} \, \log(1-\vartheta_n) \vphantom{\int\limits_{0}}&: k(\vartheta) = 1\\
0 &: k(\vartheta) > 1.
\end{cases}
$$
\end{thm}
\begin{rmk} \text{}
\begin{enumerate}
\item[(i)] The term $c_{(1-\vartheta)}$ serves as a correction term that is only positive, if $\vartheta_n$ tends to $1$ in such a way that both $n^{-1}\log(1-\vartheta_n)$ and $(\alpha_n n^{-1})_{n\in\N}$ share the same speed of convergence. For $\vartheta \in [0,1)$ we always have $k(\vartheta)=1$ and  $c_{(1-\vartheta)} = \lim\limits_{n\to\infty} n^{-1} \, \log(1-\vartheta_n) =0$ and the rate function simplifies accordingly. For $k(\vartheta)>1$ (which implies that $\vartheta_n$ tends to  $\vartheta = 1$ faster than $\alpha_n n^{-1}$ tends to $\alpha$), the term $c_{(1-\vartheta)}$ also vanishes and the sequence $(\mu_n)_{n\in\N}$ from Theorem \ref{thm:LDPEmpMeasure} based on $\bP_{n,p,\bW_n}$ shares its rate function with that from Proposition \ref{prop:LDPKimRamanan} for the cone measure $\bC_{n,p}$. Any convergence speeds slower than $k(\vartheta)=1$ would only yield trivial results, as the resulting LDP for the $p$-radial component of our probabilistic representation (see Lemma \ref{lem:LDPBetaExtended-Revised}) would have a speed slower than the LDP of the directional component (see Proposition \ref{prop:LDPKimRamanan}). However, overall we see that for many parameter sequences $(\vartheta_n)_{n \in \N}$, i.e., for many distributions $\bP_{n,p,\bW_n}$, the rate functions of the corresponding LDPs are universal. 
\item[(ii)]  We need to consider $k(\vartheta)$ such that $\lim_{n\to\infty} n^{-k(\vartheta)} \, |\log(1-\vartheta_n)| < \infty$ in order to analyze the interplay between the convex combination of measures in $\bW_n= \vartheta_n\delta_0 + (1-\vartheta_n)\textbf{G}(\alpha_n,1)$ and the parameter sequence $(\alpha_n)_{n \in \N}$ of the involved gamma distributions. %As we will see in the proof of Lemma \ref{lem:LDPBetaExtended-Revised}, the behavior of the parameter sequence $(\alpha_n)_{n \in \N}$ sets the speed of the LDP of the norm-component of the probabilistic representation of the $Z^{(n)} \sim \bP_{n,p,\bW_n}$ in \ref{thm:LDPEmpMeasure}. 
	The value of $k(\vartheta)$ and the limiting behavior of  $n^{-k(\vartheta)} \, |\log(1-\vartheta_n)|$ determine if the convex combination in $\bW_n$ \enquote{drowns out} the involved gamma distributions $\textbf{G}(\alpha_n,1)$ faster than their parameter sequence $(\alpha_n)_{n \in \N}$ can grow and have an influence on the large deviation behavior.
\end{enumerate}
\end{rmk}
The strategy of the proof of Theorem \ref{thm:LDPEmpMeasure} will be the following: for a given random vector in $\B^n_p$ with distribution $\bP_{n,p,\bW_n}$ we apply the probabilistic representation from Proposition \ref{prop:Barthe} for the specific $\bW_n$. We split that representation into two components, one representing the direction and the other the $\ell_p^n$-norm of the random vector and derive LDPs for these components separately. However, the LDP for the directional component (which has distribution $\bC_{n,p}$) has been obtained in \cite{KimRamanan} (see Proposition \ref{prop:LDPKimRamanan}). So, only the LDP for the norm component has to be established. Applying the contraction principle will then conclude the proof. %\\
%
%\newpage
%
\begin{lemma}\label{lem:LDPBetaExtended-Revised}
	Let $0<p<\infty$ and let $(\vartheta_n)_{n\in\N}$ be a sequence in $[0,1]$ with $\lim\limits_{n\to\infty} \vartheta_n = \vartheta \in [0,1]$ and denote by $k(\vartheta) \ge 1$ the smallest number such that $\lim\limits_{n\to\infty} n^{-k(\vartheta)} \, |\log(1-\vartheta_n)| < +\infty$. Also let $(\alpha_n)_{n\in\N}$ be a positive, real sequence such that $\lim\limits_{n\to\infty}\alpha_n \, n^{-1}=\alpha \in [0,\infty)$. For each $n\in\N$ let $X^{(n)}=(X_1^{(n)},\ldots,X_n^{(n)})$ be a random vector with independent coordinates such that $X_i \sim {\bN}_p$. Independently of $(X^{(n)})_{n\in\N}$, let $(W^{(n)})_{n\in\N}$ be a sequence of random variables with $W^{(n)} \sim \bW_n= \vartheta_n\delta_0 + (1-\vartheta_n)\textbf{G}(\alpha_n,1)$. Then the sequence of random variables $(B^{(n)})_{n\in\N}$ with $B^{(n)}:={\|X^{(n)}\|_p^p\over\|X^{(n)}\|_p^p+W^{(n)}}$ satisfies a large deviation principle on $[0, 1]$ with speed $n$ and good rate function %
	$$
	\mathcal{I}_{\rm beta}(x) = \begin{cases}
		0 &: k(\vartheta) > 1,  x =1 \vphantom{\int\limits_{0}} \\
		%//
		-\frac{1}{p} \log(x) - c_{(1- \vartheta)}&: k(\vartheta) = 1,  \alpha =0, x \in (0,1] \vphantom{\int\limits_{0}} \\
		%//
	-\frac{1}{p} \log(xp) - \alpha \log\Big(\frac{1-x}{\alpha}\Big) - \Big(\frac{1}{p} + \alpha\Big) \log\Big(\frac{1}{p} + \alpha \Big) - c_{(1- \vartheta)} &:k(\vartheta) = 1,  \alpha >0,  x \in (0,1) \vphantom{\int\limits_{0}}\\
	%\\
	+\infty&: otherwise,
	\end{cases}
	$$
	where 
	$$ 
	c_{(1-\vartheta)} := \begin{cases} 
	\lim\limits_{n\to\infty} n^{-1} \, \log(1-\vartheta_n) \vphantom{\int\limits_{0}}&: k(\vartheta) = 1\\
	0 &: k(\vartheta) > 1.
	\end{cases}
	$$
\end{lemma}
\begin{proof} We have seen in Proposition \ref{prop:DistributionNormB} that $B^{(n)}\sim \vartheta_n\delta_1 + (1-\vartheta_n)\textbf{B}\big({n\over p}, \alpha_n\big)$ (for $f \equiv 1$ with $m=0$). We intend to apply the theorem of Gärtner-Ellis (Proposition \ref{prop:GaertnerEllis}) to show the above LDP, somewhat following along the proof of Lemma 4.1 in \cite{APTldp}, and thus consider the following limit for $t\in \R$:
	\begin{eqnarray*}
	\Lambda(t)&:=& \lim\limits_{n \to  \infty} \frac{1}{n} \log \E\left[ e^{n t B^{(n)}} \right]\\
	\\
	&=& \lim\limits_{n \to  \infty} \frac{1}{n} \log \left[ \int_0^1 e^{n t x}  \left(\vartheta_n\delta_1 + (1-\vartheta_n)\textbf{B}\Big({n\over p}, \alpha_n\Big)\right) (\dint x) \right]\\
	\\
	&=& \lim\limits_{n \to  \infty} \frac{1}{n} \log \left[ \vartheta_n e^{n t}  + (1-\vartheta_n)\int_0^1 e^{n t x} \, \textbf{B}\Big({n\over p}, \alpha_n\Big)(\dint x) \right]\\
	\\
	&=& t + \lim\limits_{n \to  \infty} \frac{1}{n} \log \left[ \vartheta_n  + (1-\vartheta_n)\int_0^1 e^{n t (x-1)} \, \textbf{B}\Big({n\over p}, \alpha_n\Big)(\dint x) \right],
	\end{eqnarray*}
	which yields
	\begin{eqnarray*}
	\Lambda(t)&=& t + \lim\limits_{n \to  \infty} \frac{1}{n} \log \left[ \vartheta_n  + (1-\vartheta_n) \frac{1}{B(\frac{n}{p}, \alpha_n)}\int_0^1 e^{n t (x-1)} \, x^{\frac{n}{p}-1} (1-x)^{\alpha_n-1} \, \dint x \right].
	\end{eqnarray*}
The change of variables $y = 1-x$ then gives us
\begin{eqnarray} \label{eq:LemmaIntTerm1}
	\nonumber \Lambda(t)	&=& t + \lim\limits_{n \to  \infty} \frac{1}{n} \log \left[ \vartheta_n  + (1-\vartheta_n) \frac{1}{B(\frac{n}{p}, \alpha_n)} \int_0^1 e^{-n t y} \, (1-y)^{\frac{n}{p}-1} y^{\alpha_n-1} \, \dint y \right]\\
	\nonumber\\
	&=&t + \lim\limits_{n \to  \infty} \frac{1}{n} \log \left[ \vartheta_n  + (1-\vartheta_n) \frac{1}{B(\frac{n}{p}, \alpha_n)} \int_0^1 e^{n(- ty \, + \frac{{n/p\, -1}}{n} \log(1-y) + \frac{\alpha_n-1}{n} \log(y))}  \, \dint y \right].
\end{eqnarray}
At this point, we need to distinguish the cases $k(\vartheta)=1$ and $k(\vartheta)>1$ and the cases $\alpha=0$ and $\alpha >0$. The method of %the 
proof will be mostly the same for those with $k(\vartheta)=1$, which is to give upper and lower bounds for the integrand in the above expression, such that only the initial coefficient in the exponent remains  dependent on $n$, whereby we can apply one of the asymptotic integral expansion results presented in Section \ref{sec:Preliminaries}. After some explicit calculations, we will then let our upper and lower estimates  approach our initial integrand, and thereby give the sought-after limit explicitly. For $k(\vartheta)>1$ the proof follows from rather straightforward calculations. \\

We begin with $k(\vartheta)=1$ and $\alpha >0$. For any $\epsilon >0$ there exists $n_0 \in\N$ such that for $n \ge n_0$ and $y \in (0,1)$ we have that
\begin{eqnarray} \label{eq:BoundAbove}
\displaystyle e^{n(- ty \, + (\frac{1}{p} - \frac{1}{n}) \log(1-y) + \frac{\alpha_n-1}{n} \log(y)} \le e^{n(- ty \, + (\frac{1}{p}- \epsilon) \log(1-y) + (\alpha - \epsilon) \log(y))}
\end{eqnarray}
and 
\begin{eqnarray} \label{eq:BoundBelow}
\displaystyle e^{n(- ty \, + (\frac{1}{p} - \frac{1}{n}) \log(1-y) + \frac{\alpha_n-1}{n} \log(y)}  \ge e^{n(- ty \, + (\frac{1}{p}+ \epsilon) \log(1-y) + (\alpha + \epsilon)  \log(y))}.
\end{eqnarray}
Thus, the term in \eqref{eq:LemmaIntTerm1} for $\alpha >0$ is bounded from above by
\begin{eqnarray} \label{eq:IntBoundAbove}
\qquad \Lambda_{-\epsilon} (t) := t + \lim\limits_{n \to  \infty} \frac{1}{n} \log \left[ \vartheta_n  + (1-\vartheta_n) \frac{1}{B(\frac{n}{p}, \alpha_n)} \int_0^1 e^{n(- ty \, + (\frac{1}{p}- \epsilon) \log(1-y) + (\alpha - \epsilon) \log(y))}  \, \dint y \right],
\end{eqnarray}
and from below by
\begin{eqnarray} \label{eq:IntBoundBelow}
\qquad \Lambda_{+\epsilon} (t) := t + \lim\limits_{n \to  \infty} \frac{1}{n} \log \left[ \vartheta_n  + (1-\vartheta_n) \frac{1}{B(\frac{n}{p}, \alpha_n)} \int_0^1 e^{n(- ty \, + (\frac{1}{p}+ \epsilon) \log(1-y) + (\alpha + \epsilon) \log(y))}  \, \dint y \right]. 
\end{eqnarray}
We want to apply the adapted Laplace principle from Remark \ref{rmk:AdaptationLaplace} to the terms in limits of the above expressions, and thus denote 
$$\varrho_{-\epsilon, t}(y):=-ty \, + \Big(\frac{1}{p}- \epsilon\Big) \log(1-y) + (\alpha - \epsilon) \log(y),$$
and 
$$\varrho_{+\epsilon, t}(y) := -ty \, +  \Big(\frac{1}{p}+ \epsilon\Big) \log(1-y) + (\alpha + \epsilon) \log(y).$$
We already have that $(\vartheta_n)_{n \in \N}$ is bounded and non-negative. Also, $(1-\vartheta_n) {B(\frac{n}{p}, \alpha_n)}^{-1}$ is positive for all $n\in \N$ bigger than some $N \in \N$, since $k(\vartheta)=1$ implies that $\vartheta_n \ne 1$ for $n \in\N$ bigger than some $N\in\N$. Furthermore, both $\varrho_{-\epsilon, t}$ and $\varrho_{+\epsilon, t}$ are twice continuously differentiable on $(0,1)$. It remains to show that \eqref{eq:LaplaceCondition} holds for the sequence $(s^{(2)})_{n \in \N}$ with $s^{(2)}_n := (1-\vartheta_n) {B(\frac{n}{p}, \alpha_n)}^{-1}$, and that the maximum conditions of the Laplace principle are met by $\varrho_{-\epsilon, t}$ and $\varrho_{+\epsilon, t}$. We begin with the former. It follows from $\alpha >0$ that $\alpha_n \to +\infty$, thus Stirling's formula tells us that, for increasing $n$, $B(\frac{n}{p}, \alpha_n)$ behaves like
$$ \sqrt{2 \pi} \frac{(\frac{n}{p})^{\frac{n}{p} - 1/2} \alpha_n^{\alpha_n-1/2}}{(\frac{n}{p} + \alpha_n)^{\frac{n}{p} + \alpha_n - 1/2}}.$$ 
Hence, we have 
\begin{eqnarray} \label{eq:Stirling}
\nonumber \lim\limits_{n \to \infty} \frac{1}{n} \log \frac{1}{B(\frac{n}{p}, \alpha_n)} &=& - \lim\limits_{n \to \infty} \Bigg[ \frac{\log \sqrt{2\pi}}{n} + \frac{\frac{n}{p} - \frac{1}{2}}{n}\Big(\log n + \log \frac{n/p}{n}\Big)  + \frac{\alpha_n - \frac{1}{2}}{n}\Big(\log n + \log \frac{\alpha_n}{n}\Big) \\
\nonumber\\
\nonumber&&  \qquad \qquad  - \frac{\frac{n}{p} + \alpha_n - \frac{1}{2}}{n}\Big(\log n + \log \frac{\frac{n}{p} + \alpha_n}{n}\Big)\Bigg]\\
\nonumber\\
&=& - \frac{1}{p} \log \frac{1}{p} - \alpha \log \alpha + \Big(\frac{1}{p} + \alpha\Big) \log\Big(\frac{1}{p} + \alpha \Big),
\end{eqnarray}
so with $k(\vartheta)=1$ and the above it holds for the sequence $(s^{(2)})_{n \in \N}$ with $s^{(2)}:= (1-\vartheta_n) {B(\frac{n}{p}, \alpha_n)}^{-1}$ that
\begin{eqnarray} \label{eq:LogBeta} 
\lim\limits_{n \to \infty} \frac{1}{n} \log s_n^{(2)} = c_{(1- \vartheta)} - \frac{1}{p} \log \frac{1}{p} - \alpha \log \alpha + \Big(\frac{1}{p} + \alpha\Big) \log\Big(\frac{1}{p} + \alpha \Big) < +\infty,
\end{eqnarray}
Regarding the maximum conditions of the Laplace principle, direct calculation yields that for $\epsilon < \min\{ \alpha, \frac{1}{p}\}$ and $t \in \R \setminus \{0\}$ we have
\begin{eqnarray} \label{eq:SupLaplaceFct}
\nonumber \sup\limits_{y \in (0,1)} \varrho_{-\epsilon, t}(y) &=& \sup\limits_{y \in (0,1)} \Big[-ty \, + \Big( \Big(\frac{1}{p} - \epsilon \Big) \log(1-y) + (\alpha - \epsilon) \log(y)\Big)\Big]\\
\nonumber \\
\nonumber &=& \frac{1}{2} \Bigg[ -t - \Big( \alpha + \frac{1}{p} - 2\epsilon \Big) - \sqrt{(\alpha + \frac{1}{p} - 2\epsilon + t)^2 - 4(\alpha - \epsilon)t} \Bigg]\\
\nonumber \\
\nonumber && + \, \Big(\frac{1}{p}-\epsilon \Big) \log \frac{t - \Big( \alpha + \frac{1}{p} - 2\epsilon\Big) - \sqrt{(\alpha + \frac{1}{p} - 2\epsilon + t)^2 - 4(\alpha - \epsilon)t}}{2t}\\
\nonumber \\
&& + \, (\alpha - \epsilon) \log \frac{t + \Big( \alpha + \frac{1}{p} - 2\epsilon\Big) + \sqrt{(\alpha + \frac{1}{p} - 2\epsilon + t)^2 - 4(\alpha - \epsilon)t}}{2t},
\end{eqnarray}
and for $t=0$ it holds that 
\begin{eqnarray}
\nonumber \sup\limits_{y \in (0,1)} \varrho_{-\epsilon, 0}(y) &=&   \Big(\frac{1}{p}-\epsilon \Big) \log \frac{\frac{1}{p}}{\alpha + \frac{1}{p} - 2\epsilon } +  (\alpha - \epsilon) \log \frac{\alpha - \epsilon }{\alpha + \frac{1}{p} - 2\epsilon}.
\end{eqnarray}
The analogue of the above holds for the maximum of $\varrho_{+\epsilon,t}$, only with $-\epsilon$ replaced by $+\epsilon$ (In this latter calculation the condition $\epsilon < \min\{ \alpha, \frac{1}{p}\}$ is not required). By the above, it follows that the suprema of  $\varrho_{-\epsilon, t}$ and  $\varrho_{+\epsilon, t}$ are not attained on the boundary of the interval $[0,1]$, hence the Laplace principle can be applied to both. But before doing so, by setting 
\begin{eqnarray*} \label{eq:DefPsi1}
 \Psi_{-\epsilon}(y):= - \Big(\frac{1}{p}-\epsilon \Big) \log(1-y) - (\alpha - \epsilon) \log(y),
\end{eqnarray*}
and 
\begin{eqnarray*} \label{eq:DefPsi2}
 \Psi_{+\epsilon}(y):= - \Big(\frac{1}{p}+\epsilon \Big)  \log(1-y) - (\alpha + \epsilon) \log(y),
\end{eqnarray*}
we see that
\begin{eqnarray} \label{eq:IdentLegendre1}
\sup\limits_{y \in (0,1)} \varrho_{-\epsilon, t}(y) = \sup\limits_{y \in (0,1)} \Big[(-t)y - \Psi_{-\epsilon}(y)\Big] = \Psi_{-\epsilon}^*(-t)
\end{eqnarray}
and
\begin{eqnarray} \label{eq:IdentLegendre2}
\sup\limits_{y \in (0,1)} \varrho_{+\epsilon, t}(y) = \sup\limits_{y \in (0,1)} \Big[(-t)y - \Psi_{+\epsilon}(y)\Big] = \Psi_{+\epsilon}^*(-t),
\end{eqnarray}
i.e., the suprema of $\varrho_{-\epsilon, t}$ and $\varrho_{+\epsilon, t}$ can be written as Legendre-Fenchel transforms of $\Psi_{-\epsilon}$ and $\Psi_{+\epsilon}$ at $(-t)$, respectively. 
Now, using the adapted Laplace principle from \eqref{eq:ExpLaplace2}, and the identities from \eqref{eq:LogBeta}, \eqref{eq:IdentLegendre1}, and \eqref{eq:IdentLegendre2}, we can reformulate the respective upper and lower bounds $\Lambda_{-\epsilon}(t), \Lambda_{+\epsilon}(t)$ from \eqref{eq:IntBoundAbove} and \eqref{eq:IntBoundBelow} as
\begin{eqnarray} \label{eq:IntBoundAboveLaplace}
\Lambda_{-\epsilon} (t) = t + c_{(1- \vartheta)} - \frac{1}{p} \log \frac{1}{p} - \alpha \log \alpha + \Big(\frac{1}{p} + \alpha\Big) \log\Big(\frac{1}{p} + \alpha \Big) + \Psi_{-\epsilon}^*(-t)
\end{eqnarray}
and
\begin{eqnarray} \label{eq:IntBoundBelowLaplace}
\Lambda_{+\epsilon} (t) = t + c_{(1- \vartheta)} - \frac{1}{p} \log \frac{1}{p} - \alpha \log \alpha + \Big(\frac{1}{p} + \alpha\Big) \log\Big(\frac{1}{p} + \alpha \Big) + \Psi_{+\epsilon}^*(-t).
\end{eqnarray}
As the above holds for every sufficiently small $\epsilon >0$, considering the limit of $\Lambda_{-\epsilon}(t)$ and $\Lambda_{+\epsilon}(t)$ as $\epsilon$ tends to $0$ yields that 
$$\Lambda (t) = t + c_{(1- \vartheta)} - \frac{1}{p} \log \frac{1}{p} - \alpha \log \alpha + \Big(\frac{1}{p} + \alpha\Big) \log\Big(\frac{1}{p} + \alpha \Big) + \Psi^*(-t),$$
where $\Psi^*$ is the Legendre-Fenchel transform of $\Psi$ with $\Psi(y) := -\frac{1}{p} \log(1-y) - \alpha \log(y)$, which is the limit of both $\Psi_{-\epsilon}$ and $\Psi_{+\epsilon}$ as $\epsilon$ tends to $0$.
Since we can see that $\Lambda$ is finite in an open neighbourhood of the origin and is semi-continuous and differentiable, it now follows via the theorem of Gärtner-Ellis (Proposition \ref{prop:GaertnerEllis}) that the sequence $(B^{(n)})_{n\in\N}$ satisfies an LDP with speed $n$ and rate function $\Lambda^*$.  Setting 
$$c_{\vartheta, p, \alpha} := c_{(1- \vartheta)} - \frac{1}{p} \log \frac{1}{p} - \alpha \log \alpha + \Big(\frac{1}{p} + \alpha\Big) \log\Big(\frac{1}{p} + \alpha \Big),$$
we get that for $x \in (0,1)$
\begin{eqnarray}
\nonumber \Lambda^*(x) &=& \sup\limits_{t \in \R} \Big[ tx - \Lambda(t)\Big]\\
\nonumber \\
\nonumber&=& \sup\limits_{t \in \R} \Big[ tx - t - \Psi^*(-t)\Big] - c_{\vartheta, p, \alpha}\\
\nonumber \\
\nonumber&=& \sup\limits_{t \in \R} \Big[ t(x-1) -  \Psi^*(-t)\Big] - c_{\vartheta, p, \alpha}.
\end{eqnarray}
Again, using the change of variables $z = 1-x$ as in \eqref{eq:LemmaIntTerm1}, we get
\begin{eqnarray}
\nonumber \Lambda^*(x) &=& \sup\limits_{t \in \R} \Big[ (-t)z + \Psi^*(-t)\Big] - c_{\vartheta, p, \alpha} = \sup\limits_{\tilde t \in \R} \Big[\tilde t z - \Psi^*(\tilde t)\Big] - c_{\vartheta, p, \alpha} = (\Psi^*)^*(z) - c_{\vartheta, p, \alpha}.
\end{eqnarray}
As the Legendre-Fenchel transform is an involution on $(0,1)$, we have that%
\begin{eqnarray}
\nonumber \Lambda^*(x) &=& (\Psi^{*})^*(z) - c_{\vartheta, p, \alpha} = \Psi(z) - c_{\vartheta, p, \alpha}.
\end{eqnarray}
Plugging in the definition of $\Psi$ and rolling back the previous change of variables, we have that
\begin{eqnarray}
\nonumber \Lambda^*(x) &=& -\frac{1}{p} \log(1-z) - \alpha \log(z) + \frac{1}{p} \log \frac{1}{p} + \alpha \log \alpha - \Big(\frac{1}{p} + \alpha\Big) \log\Big(\frac{1}{p} + \alpha \Big) - c_{(1- \vartheta)}\\
\nonumber\\
\nonumber&=& -\frac{1}{p} \log(x) - \alpha \log(1-x) + \frac{1}{p} \log \frac{1}{p} + \alpha \log \alpha - \Big(\frac{1}{p} + \alpha\Big) \log\Big(\frac{1}{p} + \alpha \Big) - c_{(1- \vartheta)}\\
\nonumber\\
&=& -\frac{1}{p} \log(xp) - \alpha \log\Big(\frac{1-x}{\alpha}\Big) - \Big(\frac{1}{p} + \alpha\Big) \log\Big(\frac{1}{p} + \alpha \Big) - c_{(1- \vartheta)},
\end{eqnarray}
yielding the first case of our rate function. For $x \in \{0,1\}$ direct computation yields that $\Lambda^*(x)=+\infty$ in these cases.\\
\\
For $k(\vartheta)=1$ and $\alpha =0$ we need to slightly adapt some of the steps in the proof of the previous case. We again provide upper and lower bounds for the integrand, where the lower bound will be handled completely analogue to the previous case via the adapted Laplace principle \eqref{eq:ExpLaplace2}, but the upper bound needs to be approached via the asymptotic integral results from \eqref{eq:ExpLaplaceBreitung}. Let $\alpha =0$, then there exists $n_0 \in\N$ such that for $n \ge n_0$ and $y \in (0,1)$ we have that
\begin{eqnarray} \label{eq:BoundAboveBelow2}
\qquad  e^{n(- ty \, + (\frac{1}{p}+\epsilon) \log(1-y) + \epsilon \log(y))} \le e^{n(- ty \, + (\frac{1}{p}-\frac{1}{n}) \log(1-y) + \frac{\alpha_n-1}{n} \log(y)} \le \, e^{n(- ty \, + (\frac{1}{p}-\epsilon) \log(1-y))}.
\end{eqnarray}
We choose a different upper bound here than in the previous case in \eqref{eq:BoundAbove}, since for $\alpha=0$ and $t \ge (-1)$ the function $- ty \, + (\frac{1}{p}-\epsilon) \log(1-y) + (\alpha - \epsilon) \log(y)$ is strictly decreasing and attains its maximum over $[0,1]$ on the boundary of the interval at $0$. Since this is not the case for the lower bound in \eqref{eq:BoundBelow}, we can use its analogue for $\alpha=0$ here as well. %
With these bounds we get the following respective upper and lower bounds for the term in \eqref{eq:LemmaIntTerm1}:
\begin{eqnarray} \label{eq:IntBoundAbove2}
\qquad \Lambda_{- \epsilon} (t) := t + \lim\limits_{n \to  \infty} \frac{1}{n} \log \left[ \vartheta_n  + (1-\vartheta_n) \frac{1}{B(\frac{n}{p}, \alpha_n)} \int_0^1 \, e^{n(- ty \, + (\frac{1}{p}-\epsilon)\log(1-y))}  \, \dint y \right]
\end{eqnarray}
and
\begin{eqnarray} \label{eq:IntBoundBelow2}
\qquad \Lambda_{+\epsilon} (t) := t + \lim\limits_{n \to  \infty} \frac{1}{n} \log \left[ \vartheta_n  + (1-\vartheta_n) \frac{1}{B(\frac{n}{p}, \alpha_n)} \int_0^1 e^{n(- ty \, + (\frac{1}{p}+\epsilon)\log(1-y) + \epsilon \log(y))}  \, \dint y \right]. 
\end{eqnarray}
We again need to consider the behavior of $(s^{(2)})_{n \in \N}$ with $s^{(2)}_n := (1-\vartheta_n) {B(\frac{n}{p}, \alpha_n)}^{-1}$ and check the conditions of the relevant asymptotic integral expansions for the functions in the exponents of the respective integrands, denoted as 
$$ \tilde\varrho_{-\epsilon, t}(y):=-ty \, + \Big(\frac{1}{p}-\epsilon\Big) \log(1-y) \quad \text{ and } \quad \tilde\varrho_{+\epsilon, t}(y) := -ty \, + \Big(\frac{1}{p}+\epsilon\Big)\log(1-y) + \epsilon \log(y).$$
If, on the one hand, both $\alpha_n \to +\infty$ and $\alpha =0$ hold simultaneously, applying Stirling's formula as in \eqref{eq:Stirling} and interpreting the expression $0 \log(0)$ as $0$ yields that 
$$ \lim\limits_{n \to \infty} \frac{1}{n} \log \frac{1}{B(\frac{n}{p}, \alpha_n)} = 0.
$$
If, on the other hand, $\alpha_n$ is bounded, $B(\frac{n}{p}, \alpha_n)$ behaves like $\Gamma(\alpha_n) \, \big(\frac{n}{p}\big)^{-\alpha_n}$, which implies
$$
 \lim\limits_{n \to \infty} \frac{1}{n} \log \frac{1}{B(\frac{n}{p}, \alpha_n)} =  - \lim\limits_{n \to \infty} \frac{\log(\Gamma(\alpha_n))}{n} -  \lim\limits_{n \to \infty} \frac{\alpha_n}{n} \log\Big(\frac{n}{p}\Big) = 0.
$$
The positivity of $s^{(2)}_n$ follows again by $k(\vartheta)=1$. The function $\tilde p_{+\epsilon, t}$ satisfies the conditions of the Laplace principle (Proposition \ref{prop:LaplacePrinc}) by the same arguments as in the previous case. %\\
%\\
%
%
We again set
$$
	\tilde \Psi_{-\epsilon}(y):= - \Big(\frac{1}{p}-\epsilon \Big) \log(1-y), \quad \textup{ and }  \quad
	\tilde \Psi_{+\epsilon}(y):= - \Big(\frac{1}{p}+\epsilon \Big)  \log(1-y) -  \epsilon \log(y),
$$
such that
\begin{eqnarray*} %\label{eq:IdentLegendreTilde}
\sup\limits_{y \in (0,1)} \tilde\varrho_{-\epsilon, t}(y) = \tilde\Psi_{-\epsilon}^*(-t) \quad \textup{ and } \quad
\sup\limits_{y \in (0,1)} \tilde\varrho_{+\epsilon, t}(y)  = \tilde\Psi_{+\epsilon}^*(-t),
\end{eqnarray*}
as in \eqref{eq:IdentLegendre1} and \eqref{eq:IdentLegendre2}. 
Now, applying the adapted Laplace principle from \eqref{eq:ExpLaplace2} to the limit in $\Lambda_{+\epsilon}(t)$ in \eqref{eq:IntBoundBelow2}, we get
\begin{eqnarray*} \label{eq:IntBoundBelowLaplace2}
\Lambda_{+\epsilon} (t) = t + c_{(1- \vartheta)} + \tilde \Psi_{+\epsilon}^*(-t).
\end{eqnarray*}
%\vspace{3cm}
%
%
%Thus, we can again apply the adapted Laplace principle from \eqref{eq:ExpLaplace2} to the limit in $\Lambda_{+\epsilon}(t)$ in \eqref{eq:IntBoundBelow2}, yielding
%
%\begin{eqnarray*} \label{eq:IntBoundBelowLaplace2}
%\Lambda_{+\epsilon} (t) = t + c_{(1- \vartheta)} + \tilde p_{+\epsilon}^*(t),\quad  \text{ where } \quad p_{+\epsilon}(t) := \sup\limits_{y \in (0,1)} \Big[ -ty \, + \Big(\frac{1}{p} \log(1-y)  + \epsilon \log(y)\Big)\Big].
%\end{eqnarray*}
%
%
Again, we consider the limit of $\Lambda_{+\epsilon} $ as $\epsilon$ tends to zero, giving the lower bound for $\Lambda(t)$
\begin{eqnarray} \label{eq:IntBoundBelowLaplace3}
\Lambda_{+0} (t) = t + c_{(1- \vartheta)} + \tilde \Psi_{+0}^*(-t) = t + c_{(1- \vartheta)}  + \sup\limits_{y \in (0,1)} \Big[ -ty \, + \frac{1}{p} \log(1-y)\Big].
\end{eqnarray}
As to the upper bound, for $t \ge (-1)$ the function $\tilde \varrho_{-\epsilon, t}$ satisfies conditions $(a)-(d)$ from Proposition \ref{prop:Breitung}, thus, by \eqref{eq:ExpLaplaceBreitung} from Remark \ref{rmk:ApplBreitung} we get the upper bound for $\Lambda(t)$:
\begin{eqnarray*} \label{eq:IntBoundAboveLaplace3}
\Lambda_{-\epsilon} (t) = t + c_{(1- \vartheta)} + \tilde \varrho_{-\epsilon, t}(0) = t + c_{(1- \vartheta)} + \sup\limits_{y \in[0,1]} \Big[ -ty \, + \Big(\frac{1}{p}-\epsilon\Big) \log(1-y)\Big] = t + c_{(1- \vartheta)}.
\end{eqnarray*}
For $t < (-1)$ the function $\tilde \Psi_{-\epsilon, t}$ is again strictly concave and attains its supremum on $(0,1)$, so applying the adapted Laplace principle  \eqref{eq:ExpLaplace2} yields
\begin{eqnarray*} \label{eq:IntBoundAboveLaplace4}
	\Lambda_{-\epsilon} (t) = t + c_{(1- \vartheta)} + \tilde \Psi_{-\epsilon}^*(t). %,\quad  \text{ where } \quad p_{-\epsilon}^*(t) := \sup\limits_{y \in (0,1)} \Big[ -ty \, + \Big(\frac{1}{p}-\epsilon\Big) \log(1-y) \Big].
\end{eqnarray*}
Combining the two and again considering the limit for $\epsilon$ tending to zero, we see that overall it holds for $t \in \R$ 
\begin{eqnarray} \label{eq:IntBoundAboveLaplace5}
\Lambda_{-0} (t) = t + c_{(1- \vartheta)} + \sup\limits_{y \in (0,1)} \Big[ -ty \, + \frac{1}{p} \log(1-y) \Big],
\end{eqnarray}
which together with \eqref{eq:IntBoundBelowLaplace3} yields that 
$$\Lambda (t) = t + c_{(1- \vartheta)} + \sup\limits_{y \in (0,1)} \Big[ -ty \, + \frac{1}{p} \log(1-y) \Big]= t + c_{(1- \vartheta)} + \tilde \Psi^*(-t),$$
with $\tilde \Psi^*$ being the Legendre-Fenchel transform of $\tilde\Psi$ with $\tilde \Psi(y) := -\frac{1}{p} \log(1-y)$. By the theorem of Gärtner-Ellis (Proposition \ref{prop:GaertnerEllis}) and the same involution and change of variables arguments as in the previous case, we get that for $\alpha =0$ the sequence $(B^{(n)})_{n\in\N}$ thus satisfies an LDP with speed $n$ and rate function
\begin{eqnarray}
\nonumber \Lambda^*(x) &=& -\frac{1}{p} \log(x) - c_{(1- \vartheta)}.
\end{eqnarray}
Lastly, let $k(\vartheta)>1$. This implies on the one hand that $\vartheta =1$ and on the other hand that $(1-\vartheta_n)$ tends to zero faster than the integral expression in \eqref{eq:LemmaIntTerm1} tends to infinity, i.e., the product of both tends to zero. Hence,  the overall expression in \eqref{eq:LemmaIntTerm1} simplifies to
\begin{eqnarray}\label{eq:k=0Int}
\nonumber \Lambda(t)	&=&t + \lim\limits_{n \to  \infty} \frac{1}{n} \log \left[ \vartheta_n  + (1-\vartheta_n) \frac{1}{B(\frac{n}{p}, \alpha_n)} \int_0^1 e^{n(- ty \, + \frac{{n/p-1}}{n} \log(1-y) + \frac{\alpha_n-1}{n} \log(y))}  \, \dint y \right]\\
\nonumber \\
\nonumber&=& t + \lim\limits_{n \to  \infty} \frac{1}{n} \log \left[ \vartheta_n\right] \, \, \,  = \, \, \, t, 
\end{eqnarray}
which implies via the theorem of Gärtner-Ellis (Proposition \ref{prop:GaertnerEllis}) that the sequence $(B^{(n)})_{n\in\N}$ satisfies an LDP on $[0,1]$ with speed $n$ and rate function
\begin{eqnarray}
\nonumber 	\mathcal{I}_{\rm beta}(x) = \Lambda^*(x) &=& \sup\limits_{t \in \R} \Big[ tx \, -t \Big] =  \sup\limits_{t \in \R} \Big[ t(x-1) \Big] =  \begin{cases} 
0&: x=1\\
+\infty &: otherwise.
\end{cases} 
\end{eqnarray}
This finishes the proof. 
\end{proof} 
\begin{proof}[Proof of Theorem \ref{thm:LDPEmpMeasure}] By Proposition \ref{prop:Barthe}, Lemma \ref{lem:DistributionConeMeasure} (i), and Proposition \ref{prop:ProbRepMixedDistributionConeUnif} and \ref{prop:DistributionNormB} (all for $f\equiv 1$), we have that	
$$
\mu_n={1\over n}\sum\limits_{i=1}^n\delta_{n^{1/p}Z^{(n)}_i} \overset{d}{=} {1\over n}\sum_{i=1}^n\delta_{n^{1/p}{X_i^{(n)}\over(\|X^{(n)}\|_p^p+W^{(n)})^{1/p}}} = {1\over n}\sum_{i=1}^n\delta_{n^{1/p}{B^{(n)}}^{1/p}{X_i^{(n)}\over\|X^{(n)}\|_p}},
$$
where $X^{(n)}$ is a random vector with density $C_{n,p,1} \, e^{-\|x\|_p^p}$, $x\in\R^n$, (i.e., with distribution ${\bN}_p^{\otimes n}$), $W^{(n)}$ a random variable on $[0,\infty)$ with distribution $\bW_n=\vartheta_n\delta_0+(1-\vartheta_n)\textbf{G}(\alpha_n, 1)$, and $B^{(n)}={\|X^{(n)}\|_p^p\over\|X^{(n)}\|_p^p+W^{(n)}}$ as in Lemma \ref{lem:LDPBetaExtended-Revised}. Let us define a sequence of random probability measures $(\xi_n)_{n\in\N}$ by
 $$
\xi_n := {1\over n}\sum_{i=1}^n\delta_{n^{1/p}{X_i^{(n)}\over\|X^{(n)}\|_p}}.
$$
Then, since $X^{(n)}\over\|X^{(n)}\|_p$ is independent from $\|X^{(n)}\|_p$ (see \cite[Theorem 3.2]{PTTSurvey}), Proposition \ref{prop:ProdLDP}, Proposition \ref{prop:LDPKimRamanan} and Lemma \ref{lem:LDPBetaExtended-Revised}  imply that the sequence of random elements $(\xi_n,B^{(n)})_{n\in\N}$ satisfies a large deviation principle on $\mathcal{M}(\R)\times[0,1]$ with speed $n$ and good rate function
$$
\mathcal{I}_1(\xi,z) = \mathcal{I}_{\textup{cone}}(\xi) + \mathcal{I}_{\rm beta}(z),\qquad (\xi,z)\in\mathcal{M}(\R)\times[0,1].
$$
In the case $z=0$, we can see by Lemma \ref{lem:LDPBetaExtended-Revised}, that $\mathcal{I}_{\rm beta}(0)=+\infty$ and thereby  $\mathcal{I}_1(\xi,0)=\mathcal{I}_{\textup{cone}}(\xi) + \infty = + \infty$ for all $\xi \in \mathcal{M}(\R)$. Thus, we confine ourselves to $z\in (0,1]$. Next, we introduce the continuous map $F_p:\mathcal{M}(\R)\times(0,1]\to\mathcal{M}(\R):(\xi,z)\mapsto \xi(z^{-1/p}\,\cdot\,)$ and notice that for each $n\in\N$ and for any Borel set $A\in\mathcal{B}(\R)$, $F_p(\xi_n,B^{(n)})(A)= \mu_n(A)$. By the contraction principle in Proposition \ref{prop:ContrPrinc}, the sequence of random probability measures $(\mu_n)_{n\in\N}$ thus satisfies a large deviation principle with speed $n$ and good rate function $\mathcal{I}_2: \mathcal{M}(\R) \times (0,1] \to [0, \infty]$ given by
$$
\mathcal{I}_2(\mu) = \inf_{\xi(z^{-1/p} \, \cdot \,)=\mu(\cdot)}\Big[\mathcal{I}_{\textup{cone}}(\xi) + \mathcal{I}_{\rm beta}(z)\Big],\qquad\mu\in\mathcal{M}(\R), z \in (0, 1].
$$
It remains to show that $\mathcal{I}_2$ in fact coincides with the rate function $\mathcal{I}_{\textup{emp}}$ stated in the theorem. The rate functions $\mathcal{I}_{\textup{cone}}$ and $\mathcal{I}_{\rm beta}$ each depend on their respective parameters $m_p(\mu) \in [0, \infty]$, $k(\vartheta) \ge 1$ and $\alpha \in [0,\infty)$, so we need to check for which parameter configurations they remain finite.% \\
\medskip
\paragraph{\em Case 1.} Let $\mu\in\mathcal{M}(\R)$ be such that $m_p(\mu) > 1$. Then, by $\xi(z^{-1/p}\, \cdot \,)=\mu(\,\cdot \,)$, we know that $m_p(\xi) = z^{-1} \,m_p(\mu)$, so $m_p(\xi) >1$. Therefore  $\mathcal{I}_{\textup{cone}}(\xi) =+\infty$ and $\mathcal{I}_{\textup{emp}}(\mu)= \mathcal{I}_2(\mu) = +\infty$.%
\medskip
\paragraph{\em Case 2.} Let $\mu\in\mathcal{M}(\R)$ be such that $m_p(\mu) \leq 1$ and $\textbf{W}_n$ be such that $\alpha =0$. By $\xi(z^{-1/p} \,\cdot\,)=\mu(\,\cdot\,)$, we again know that $m_p(\xi) = z^{-1} \,m_p(\mu)$. Now we have to distinguish between the cases $m_p(\xi) >1$ and $m_p(\xi) \leq1$. In the first case, $\mathcal{I}_{\textup{cone}}(\xi) =+\infty$ and therefore $\mathcal{I}_{\textup{emp}}(\mu)= \mathcal{I}_2(\mu) = +\infty$. If $m_p(\xi) \leq1$, then $z$ is restricted to the non-empty interval $[m_p(\mu), 1]$. Hence, $z\in [m_p(\mu),1] \cap (0,1]$. If $k(\vartheta)>1$, we now by Lemma \ref{lem:LDPBetaExtended-Revised} that  $\mathcal{I}_{\rm beta}(z)$ is only finite for $z=1$, in which case it follows that $\xi = \mu$ and $\mathcal{I}_2(\mu) = \mathcal{I}_{\textup{cone}}(\mu)= \mathcal{I}_{\textup{cone}}(\mu) - c_{(1-\vartheta)}$. If $k(\vartheta)=1$, by Proposition \ref{prop:LDPKimRamanan} and Lemma \ref{lem:LDPBetaExtended-Revised}, we get 
\begin{align*}
	\mathcal{I}_2(\mu) &= \inf_{\xi(z^{-1/p} \, \cdot \,)=\mu(\cdot)}\Big[H(\xi\|\bN_p) + (1 -m_p(\xi)) -{1\over p}\log (z) - c_{(1-\vartheta)}\Big]\\
	&= \inf_{\xi(z^{-1/p} \, \cdot \,)=\mu(\cdot)}\Big[ \int_\R \log\frac{\xi(\dint x)}{\bN_p(\dint x)} \,\xi( \dint x) + (1 -z^{-1}m_p(\mu)) -{1\over p}\log (z)\Big] - c_{(1-\vartheta)}.
\end{align*}
The change of variables $y=z^{1/p}x$ then gives us $\xi(\dint x) = \xi(\dint z^{-1/p}y)=\mu( \dint y)$, and
$$ \bN_p(\dint x)= \bN_p(\dint z^{-1/p}y) = (2z^{1/p}\Gamma(1+1/p))^{-1} \, e^{-z^{-1}|y|^p} \dint y=: \bN_{p, z}( \dint y).$$
Thus, 
\begin{align*}
	\mathcal{I}_2(\mu)  &= \inf_{z\in [m_p(\mu),1] \cap (0,1]}\Big[ \int_\R \log\frac{ \mu(\dint y)}{\bN_{p, z}(\dint y)}\, \mu(\dint y) + (1 -z^{-1}m_p(\mu)) -{1\over p}\log (z) \Big] - c_{(1-\vartheta)},
\end{align*}
which is only dependent on $z\in [m_p(\mu),1] \cap (0,1]$. We further compute
\begin{align*}
	\int_\R \log\frac{ \mu( \dint y)}{\bN_{p, z}( \dint y)} \,\mu(\dint y)&=  \int_\R \log\Big(\frac{ \mu(\dint y)}{\bN_p(\dint y)} \frac{\bN_p(\dint y)}{\bN_{p, z}(\dint y)}\Big) \, \mu(\dint y)\\
	&=  \int_\R \log \frac{ \mu(\dint y)}{\bN_p(\dint y)} \, \mu(\dint y) +  \int_\R \log \frac{\bN_p(\dint y)}{\bN_{p, z}(\dint y)} \, \mu(\dint y)\\
	&= H(\mu\|\bN_p) +  \int_\R \log \frac{\bN_p(\dint y)}{\bN_{p, z}(\dint y)} \, \mu(\dint y).
\end{align*}
Since
$$\displaystyle \frac{\bN_p(\dint y)}{ \bN_{p, z}(\dint y)} = \displaystyle e^{(z^{-1}-1)|y|^p}z^{1/p},$$
we conclude that
\begin{align} \label{eq:RN-Derivative}
	\nonumber \int_\R \log\frac{ \mu(\dint y)}{\bN_{p, z}(\dint y)} \mu(\dint y) &= H(\mu\|\bN_p) +  \int_\R \log \big(e^{(z^{-1}-1)|y|^p} z^{1/p}\big) \, \mu(\dint y)\\
	\nonumber &= H(\mu\|\bN_p) +  (z^{-1}-1) \int_\R |y|^p \,  \mu(\dint y) +  \frac{1}{p}\log(z) \int_\R \,  \mu(\dint y) \\
	&= H(\mu\|\bN_p) +  (z^{-1}-1) m_p(\mu) +  \frac{1}{p}\log(z).
\end{align}
Hence, the rate function is of the form
\begin{equation*}
	\mathcal{I}_2(\mu)  = H(\mu\|\bN_p) +   (1 -m_p(\mu)) - c_{(1-\vartheta)} =\mathcal{I}_{\rm cone}(\mu) - c_{(1-\vartheta)}.
\end{equation*}
%
%\medskip
%
\paragraph{\em Case 3.} Let $\mu\in\mathcal{M}(\R)$ be such that $m_p(\mu) \leq 1$ and $\textbf{W}_n$ be such that $\alpha >0$. By the same arguments as above, we assume that $m_p(\xi)\leq1$ and $z\in[m_p(\mu), 1]\cap (0,1)$, where we exclude $z=1$ due to Lemma \ref{lem:LDPBetaExtended-Revised}. Then, by Proposition \ref{prop:LDPKimRamanan} and Lemma \ref{lem:LDPBetaExtended-Revised}, we get 
\begin{align*}
\mathcal{I}_2(\mu) &= \inf_{\xi(z^{-1/p} \, \cdot \,)=\mu(\cdot)}\Big[H(\xi\|\bN_p) + (1 -m_p(\xi))  -{1\over p}\log {pz} - \alpha \log {{1 -z}\over \alpha} \\
& \qquad \qquad \qquad \qquad - \Big({1\over p} + \alpha\Big) \log \Big({1\over p} + \alpha\Big) - c_{(1-\vartheta)}  \Big]\\
&= \inf_{\xi(z^{-1/p} \, \cdot \,)=\mu(\cdot)}\Big[ \int_\R \log\frac{\xi(\dint x)}{\bN_p(\dint x)}\xi(\dint x) + (1 -z^{-1}m_p(\mu)) -{1\over p}\log (pz) - \alpha \log {{1 -z}\over \alpha} \Big]  \\
& \qquad \qquad \qquad \qquad   - \Big({1\over p} + \alpha\Big) \log \Big({1\over p} + \alpha\Big) - c_{(1-\vartheta)}.
\end{align*}
The change of variables $y=z^{1/p}x$ as in Case 2 then gives   
 \begin{align*}
\mathcal{I}_2(\mu)  &= \inf_{z\in [m_p(\mu),1] \cap (0,1)}\Big[ \int_\R \log\frac{ \mu(\dint y)}{\bN_{p, z}(\dint y)} \mu(\dint y) + (1 -z^{-1}m_p(\mu)) -{1\over p}\log(pz) - \alpha \log {{1 -z}\over \alpha} \Big]\\
&\qquad \qquad \qquad \qquad    - \Big({1\over p} + \alpha\Big) \log \Big({1\over p} + \alpha\Big) - c_{(1-\vartheta)}.
\end{align*}
Using now the argument from \eqref{eq:RN-Derivative} it follows that
\begin{align*}
\mathcal{I}_2(\mu)  &= \inf_{z\in [m_p(\mu),1] \cap (0,1)}\Big[H(\mu\|\bN_p) +  (z^{-1}-1) m_p(\mu) +  \frac{1}{p}\log(z) + (1 -z^{-1}m_p(\mu)) \\
& \qquad \qquad \qquad \qquad   -{1\over p}\log (pz) - \alpha \log {{1 -z}\over \alpha} \Big] - \Big({1\over p} + \alpha\Big) \log \Big({1\over p} + \alpha\Big) - c_{(1-\vartheta)}\\
&= H(\mu\|\bN_p) +  (1 - m_p(\mu))  -{1\over p}\log(p) + \alpha \log (\alpha) - \Big({1\over p} + \alpha\Big) \log \Big({1\over p} + \alpha\Big) - c_{(1-\vartheta)}\\
&\quad + \inf_{z\in [m_p(\mu),1] \cap (0,1)} \Big[ - \alpha \log (1 -z) \Big]\\
&= \mathcal{I}_{\rm cone}(\mu) +{1\over p}\log\Big({1\over p}\Big) - \Big({1\over p} + \alpha\Big) \log \Big({1\over p} + \alpha\Big) - \alpha \log \Big(\frac{1 -m_p(\mu)}{\alpha}\Big) - c_{(1-\vartheta)},
\end{align*}
where the last equality only holds for $m_p(\mu)<1$, since for $m_p(\mu)=1$, we have $z\in [m_p(\mu),1] \cap (0,1) = \emptyset$, and thus $\mathcal{I}_2(\mu)= +\infty$. Hence, $m_p(\mu)=1$ can only be permitted if $\alpha=0$.\\
\\
Thus, we have shown that $\mathcal{I}_2$ in fact coincides with $\mathcal{I}_{\textup{emp}}$ as given in Theorem \ref{thm:LDPEmpMeasure}.
\end{proof}

\begin{example} \label{exp:LambdaZeroEucl} \rm
If $\vartheta=0$, we get the large deviation behavior of the empirical measure of a random vector distributed according to some beta-type distribution $\Psi_{f,n}\bU_{n,p,f}$ as discussed in Example \ref{ex:Beta} for $f\equiv 1$. Note that this could be any distribution $\vartheta_n\bC_{n,p} + (1-\vartheta_n)\Psi_{f,n}\bU_{n,p,f}$ with $\vartheta_n \to \vartheta =0$. Since $\vartheta$ only influences the rate function via $c_{(1-\vartheta)}$ and $c_{(1-\vartheta)}=0$ for $\vartheta\in [0,1)$, any distribution $\vartheta_n\bC_{n,p} + (1-\vartheta_n)\Psi_{f,n}\bU_{n,p,f}$ with $\vartheta_n \to \vartheta \in [0,1)$ exhibits the same large deviation behavior, i.e., shares the same universal rate function for the sequence $(\mu_n)_{n\in\N}$ of corresponding empirical measures 
\begin{align*}
\mathcal{I}_{\textup{emp}}(\mu) = \begin{cases}
\parbox{11cm}{$\displaystyle \mathcal{I}_{\textup{cone}}(\mu)$}&:\parbox{2cm}{$ m_p(\mu) \le 1$\\$\alpha=0$}\\
\parbox{11cm}{$\displaystyle \mathcal{I}_{\textup{cone}}(\mu) + {1 \over p}\log{1 \over p} - \Big({1\over p} + \alpha\Big) \log \Big({1\over p} + \alpha\Big) -\alpha \log \Big({{1-m_p(\mu)}\over \alpha}\Big)$}&:\parbox{2cm}{$ m_p(\mu) < 1$\\$\alpha >0$}\\
+\vphantom{\int\limits^{1}} \infty &:otherwise.
\end{cases}
\end{align*} 
\end{example}
%
%\newpage
%
\section{Application to large deviations: matrix $p$-balls}\label{sec:ApplicationMat}
In this section, we want to use the tools acquired in the previous sections to analyse the large deviation behaviors of random matrices in $\B_{p,\beta}^{n,\mathscr{H}}$ and $\B_{p,\beta}^{n,{\mathscr{M}}}$ distributed according to $\bP_{n,p,\bW,\beta}^{\mathscr{H}}$ and $\bP_{n,p,\bW,\beta}^{\mathscr{M}}$, respectively. We will use the probabilistic representations from Theorem \ref{thm:EvDistr} and Theorem \ref{thm:SvDistr} about the eigenvalue and singular value distributions together with further LDP-results for their $p$-norm component in the spirit of Lemma \ref{lem:LDPBetaExtended-Revised} to derive large deviation principles for the self-adjoint and non-self-adjoint matrix $p$-balls.%
\subsection{LDPs for the empirical spectral measure of random matrices in $\B^{n, \mathscr{H}}_{p,\beta}$} \label{subsec:LDPMatrixSA}
In the case of the matrix $p$-balls, our goal is to derive an LDP for the so called empirical spectral measure of a random matrix in $\B^{n, \mathscr{H}}_{p,\beta}$. For a self-adjoint random matrix $Z \in \mathscr{H}_n(\mathbb{F}_\beta)$ with eigenvalues $\lambda_1(Z)\le \ldots \le \lambda_n(Z)$ the empirical spectral measure is defined as the random measure $\nu_n:={1\over n}\sum_{i=1}^n\delta_{\lambda_i(Z)}$, i.e., the empirical measure with respect to the eigenvalues. We will again consider the suitably scaled version $\mu_n:={1\over n}\sum_{i=1}^n\delta_{n^{1/p}\lambda_i(Z)}$ and refer to it as the empirical spectral measure of the random matrix $Z$.

In \cite{KPTSanov} a large deviation principle for the empirical spectral measure of random matrices chosen according to either $\bU_{n,p,\beta}^{\mathscr{H}}$ or $\bC_{n,p,\beta}^{\mathscr{H}}$ was proven. In this section, we generalize this result by proving a large deviation principle for random matrices  chosen according to one of the more general distributions $\bP_{n,p,\bW,\beta}^{\mathscr{H}} :=\bW(\{0\})\bC_{n,p,\beta}^{\mathscr{H}}+\Psi^{\mathscr{H}}\bU_{n,p,\beta}^{\mathscr{H}} \text{ on } \B_{p,\beta}^{n,\mathscr{H}}$ introduced in Section \ref{sec:EigenSingularDistr}. As in the previous section, we consider for $\bW$ distributions $\bW_n:=\vartheta_n\delta_0 +(1-\vartheta_n)\textbf{G}(\alpha_n, 1)$ with weight sequence $(\vartheta_n)_{n\in\N}$ in $[0,1]$ and parameter sequence $(\alpha_n)_{n\in\N}$ in $[0,\infty)$, and thus write $\bP_{n,p,\bW_n,\beta}^{\mathscr{H}}$ and $\Psi^{\mathscr{H}}_n$ (and $\Psi_{f,n}$ in the Euclidean representation). 
\begin{thm}\label{thm:LDPEmpSpecMeasureSA}
Let $0<p<\infty$, $\beta\in\{1,2,4\}$, and let $(\vartheta_n)_{n\in\N}$ be a sequence in $[0,1]$ with $\lim\limits_{n\to\infty} \vartheta_n = \vartheta \in [0,1]$ and denote by $k(\vartheta) \ge 2$ the smallest number such that $\lim\limits_{n\to\infty} n^{-k(\vartheta)} \, |\log(1-\vartheta_n)| < +\infty$. Further, let $(\alpha_n)_{n\in\N}$ be a positive, real sequence such that $\lim\limits_{n\to\infty} \alpha_n n^{-2} =\alpha \in [0,\infty)$. For each $n\in\N$ let $\bW_n=\vartheta_n\delta_0 +(1-\vartheta_n)\textbf{G}(\alpha_n, 1)$, and let $Z^{(n)}$ be a random matrix in $\B_{p,\beta}^{n, \mathscr{H}}$ chosen according to the distribution $\bP_{n,p,\bW_n,\beta}^{\mathscr{H}}$. Then the sequence of random probability measures $\mu_n:={1\over n}\sum\limits_{i=1}^n\delta_{n^{1/p}\lambda_i(Z^{(n)})}$ satisfies a large deviation principle on $\mathcal{M}(\R)$ with speed $n^2$ and good rate function
\begin{align*}
\mathcal{I}_{\textup{emp}}^{\mathscr{H}} (\mu) = \begin{cases}
\displaystyle \mathcal{I}_{\textup{cone}}^{\mathscr{H}} (\mu) - c_{(1-\vartheta)}^{\mathscr{H}} &: \parbox{5cm}{$m_p(\mu)\le1, k(\vartheta) \ge 2, \alpha =0$}\\
\parbox{9cm}{$ \displaystyle  \mathcal{I}_{\rm cone}^{\mathscr{H}} (\mu) +{\beta\over 2p}\log\Big({\beta\over 2p}\Big) - \Big({\beta\over 2p} + \alpha\Big) \log \Big({\beta\over 2p} + \alpha\Big) \vphantom{\int\limits_0^1}$\\$- \alpha \log \Big(\frac{1 -m_p(\mu)}{\alpha}\Big) - c_{(1-\vartheta)}^{\mathscr{H}}  \vphantom{\int\limits_0}$} &: \parbox{5cm}{$m_p(\mu)<1, k(\vartheta)=2,  \alpha>0$}\\
+\infty &: otherwise,
\end{cases}
\end{align*}
where 
$$
\mathcal{I}_{\textup{cone}}^{\mathscr{H}}(\mu) = \begin{cases}
\displaystyle -{\beta\over 2}\int_{\R}\int_\R\log|x-y|\,\mu(\dint x)\mu(\dint y)+{\beta\over 2p}\log\Big({\sqrt{\pi}\,p\,\Gamma({p\over 2})\over 2^p\,\sqrt{e}\,\Gamma({p+1\over 2})}\Big) &: m_p(\mu) \leq 1\\
+\infty &: otherwise,
\end{cases}
$$
and 
$$ 
c_{(1-\vartheta)}^{\mathscr{H}} := \begin{cases} 
\lim\limits_{n\to\infty} n^{-2} \, \log(1-\vartheta_n) \vphantom{\int\limits_{0}}&: k(\vartheta) = 2\\
0 &: k(\vartheta) > 2.
\end{cases}
$$
\end{thm}
The proof of this result is rather similar to that of Theorem \ref{thm:LDPEmpMeasure}, with the main difference that we will need to use the probabilistic representation from Theorem \ref{thm:EvDistr}, which is weighted by the repulsion factor $\Delta_\beta$ of the eigenvalues ($\Delta_\beta^c$ with normalizing constants). We again split that probabilistic representation into two components, one directional component with distribution $\bC_{n,p,\beta}^{\mathscr{H}}$ on the matrix $p$-ball and the other reflecting the $p$-radial component. The main difference will be that the degree of homogeneity $m$ of the weight function $f$ is non-zero if $f =\Delta^c_\beta$, but $m={\beta n(n-1)\over2}$. Therefore, as outlined in Remark \ref{rmk:NormDistrSA}, the first parameter of the beta distribution involved in the distribution of the $p$-radial component (compare with Lemma \ref{lem:LDPBetaExtended-Revised}) will have a different limit behavior, affecting both the speed (via the order of convergence) and the rate function (via the limit).\\
We now present two results outlining the large deviation behavior of the aforementioned two components of the probabilistic representation of a random matrix with distribution $\bP_{n,p,\bW_n,\beta}^{\mathscr{H}}$. One will do so for the empirical spectral measure of random matrices with distribution $\bC_{n,p,\beta}^{\mathscr{H}}$ on $\B^{n, \mathscr{H}}_{p,\beta}$ and the other for the $p$-norm of the probabilistic representation.  We start with the latter. 
\begin{lemma}\label{lem:LDPBetaExtendedMat1}
Let $0<p<\infty$, $\beta\in\{1,2,4\}$, and let $(\vartheta_n)_{n\in\N}$ be a sequence in $[0,1]$ with $\lim\limits_{n\to\infty} \vartheta_n = \vartheta \in [0,1]$ and denote by $k(\vartheta) \ge 2$ the smallest number such that $\lim\limits_{n\to\infty} n^{-k(\vartheta)} \, |\log(1-\vartheta_n)| < +\infty$. Also let $(\alpha_n)_{n\in\N}$ be a positive, real sequences such that $\lim\limits_{n\to\infty}\alpha_n n^{-2}=\alpha \in [0,\infty)$.  For each $n\in\N$ let $X^{(n)}=(X_1^{(n)},\ldots,X_n^{(n)})$ be a random vector with density $C_{n,p,\Delta^c_\beta} \,e^{-\|x\|_p^p}\, \Delta^c_\beta(x)$, $x\in\R^n$, with $\Delta^c_\beta$ defined as in Theorem \ref{thm:EvDistr}. Independently of the sequence $(X^{(n)})_{n\in\N}$, let $(W^{(n)})_{n\in\N}$ be a sequence of random variables with $W^{(n)} \sim \bW_n= \vartheta_n\delta_0 + (1-\vartheta_n)\textbf{G}(\alpha_n,1)$. Then the sequence of random variables $(B^{(n)})_{n\in\N}$ with $B^{(n)}:={\|X^{(n)}\|_p^p\over\|X^{(n)}\|_p^p+W^{(n)}}$ satisfies a large deviation principle on $[0,\infty)$ with speed $n^2$ and good rate function
	$$
\mathcal{I}_{\rm beta}^{\mathscr{H}}(x) = \begin{cases}
0 &: k(\vartheta) > 2,  x =1 \vphantom{\int\limits_{0}} \\
%//
-{\beta\over 2p} \log(x) - c_{(1- \vartheta)}^{\mathscr{H}}(x)&: \parbox{3cm}{$k(\vartheta) = 2,  \alpha =0,$\\ $x \in (0,1] \vphantom{\int\limits_{0}}$} \\
%//
\parbox{10cm}{$-{\beta\over 2p} \log(\frac{2xp}{\beta}) - \alpha \log\Big(\frac{1-x}{\alpha}\Big) - \Big({\beta\over 2p} + \alpha\Big) \log\Big({\beta\over 2p} + \alpha \Big) - c_{(1- \vartheta)}^{\mathscr{H}}(x)$} &:\parbox{3cm}{$k(\vartheta) = 2,  \alpha >0,$\\$x \in (0,1) \vphantom{\int\limits_{0}}$}\\
%\\
+\infty&: otherwise,
\end{cases}
$$
where 
$$ 
c_{(1-\vartheta)}^{\mathscr{H}} := \begin{cases} 
\lim\limits_{n\to\infty} n^{-2} \, \log(1-\vartheta_n) \vphantom{\int\limits_{0}}&: k(\vartheta) = 2\\
0 &: k(\vartheta) > 2.
\end{cases}
$$
\end{lemma}
This is proven in the same way as Lemma \ref{lem:LDPBetaExtended-Revised} with only a few differences. Since we are dealing with matrix $p$-balls here, the weight function is $\Delta^c_\beta$, which is homogeneous of degree  $m={\beta n(n-1)\over2}$. We use the probabilistic representation of the $\ell_p^n$-norm of the eigenvalue-vector via the distributional convex combination given in Remark \ref{rmk:NormDistrSA}. We have seen in the proof of Lemma \ref{lem:LDPBetaExtended-Revised} that the LDP of the latter is heavily dependent on the limits and orders of convergence of the involved parameter sequences. It holds for the first parameter of the involved beta distribution from Remark \ref{rmk:NormDistrSA} that $\lim_{n\to\infty} \frac{n+m}{p} n^{-2} = \frac{\beta}{2p}$. This explains the appearance of $n^2$ instead of $n$ for the speed and the factor $\beta \over 2p$ instead of $1 \over p$ in the rate function. \\
The second lemma is a large deviation principle for the sequence of empirical spectral measures of a random matrix in $\B_{p,\beta}^{n, \mathscr{H}}$ with distribution $\bC_{n,p,\beta}^{\mathscr{H}}$ from \cite[Theorem 1.1]{KPTSanov}.
\begin{lemma}\label{lem:LDPfromSanovPaper}
Let $0<p<\infty$, $\beta\in\{1,2,4\}$ and $n\in\N$. Further, let $Z^{(n)}$ be a random matrix in $\B_{p,\beta}^{n, \mathscr{H}}$ with distribution $\bC_{n,p,\beta}^{\mathscr{H}}$ and eigenvalues $\lambda_i(Z^{(n)})$, $i\in\{1,\ldots,n\}$. Then the sequence of random probability measures
$
\mu_n := {1\over n}\sum_{i=1}^n\delta_{n^{1/p}\lambda_i(Z^{(n)})}
$
satisfies a large deviation principle on $\mathcal{M}(\R)$ with speed $n^2$ and good rate function
$$
\mathcal{I}_{\textup{cone}}^{\mathscr{H}}(\mu) = \begin{cases}
\displaystyle -{\beta\over 2}\int_{\R}\int_\R\log|x-y|\,\mu(\dint x)\mu(\dint y)+{\beta\over 2p}\log\Big({\sqrt{\pi}\,p\,\Gamma({p\over 2})\over 2^p\,\sqrt{e}\,\Gamma({p+1\over 2})}\Big) &: m_p(\mu) \leq 1\\
+\infty &: otherwise.
\end{cases}
$$
\end{lemma}
 
 \begin{proof}[Proof of Theorem \ref{thm:LDPEmpSpecMeasureSA}] Since this proof is again quite similar to that of Theorem \ref{thm:LDPEmpMeasure}, we reduce it to the essential differences. We use the probabilistic representations from Theorem \ref{thm:EvDistr}, Lemma \ref{lem:DistributionConeMeasure} (i), Proposition \ref{prop:ProbRepMixedDistributionConeUnif}, and Proposition \ref{prop:DistributionNormB} to get
$$
\mu_n:={1\over n}\sum\limits_{i=1}^n\delta_{n^{1/p}\lambda_i(Z^{(n)})} \overset{d}{=} {1\over n}\sum_{i=1}^n\delta_{n^{1/p}{X_i^{(n)}\over(\|X^{(n)}\|_p^p+W^{(n)})^{1/p}}} = {1\over n}\sum_{i=1}^n\delta_{n^{1/p}{B^{(n)}}^{1/p}{X_i^{(n)}\over\|X^{(n)}\|_p}},
$$
where $X^{(n)}$ is a random vector with density $C_{n,p,\Delta^c_\beta}\,e^{-\|x\|_p^p}\,\Delta^c_\beta(x)$, $x\in\R^n$, $W^{(n)}$ a random variable on $[0,\infty)$ with distribution $\bW_n=\vartheta_n\delta_0+(1-\vartheta_n)\textbf{G}(\alpha_n, 1)$, and $B^{(n)}={\|X^{(n)}\|_p^p\over\|X^{(n)}\|_p^p+W^{(n)}}$. Note, that while Theorem \ref{thm:EvDistr} makes a distributional statement for the randomly permutaed eigenvalue vector $\lambda_\sigma(Z)$, the above statement holds for the empirical measure of the ordered eigenvalue vector $\lambda(Z)$ as well, since we are considering the Dirac measures of its coordinates within a sum, in which the order of the summands is irrelevant. Using Lemma \ref{lem:LDPBetaExtendedMat1} and Lemma \ref{lem:LDPfromSanovPaper}, by the same arguments as in the proof of Theorem \ref{thm:LDPEmpMeasure}, we get that $(\mu_n)_{n\in\N}$ satisfies an LDP with speed $n^2$ and good rate function $\mathcal{I}_2^{\mathscr{H}}: \mathcal{M}(\R) \times (0,1] \to [0,\infty)$ given by
$$
\mathcal{I}_2^{\mathscr{H}}(\mu) = \inf_{\xi(z^{-1/p} \, \cdot \,)=\mu(\cdot)}\Big[\mathcal{I}_{\textup{cone}}^{\mathscr{H}}(\xi) + \mathcal{I}_{\rm beta}^{\mathscr{H}}(z)\Big]. %,\qquad\mu\in\mathcal{M}(\R), z \in (0,1].
$$
It remains to show that $\mathcal{I}_2^{\mathscr{H}}$ is just the rate function $\mathcal{I}_{\textup{emp}}^{\mathscr{H}}$ stated in the theorem. However, this is done again by a case-by-case analysis of parameter configurations $m_p(\mu) \in [0, \infty]$, $k(\vartheta) \ge 2$ and $\alpha \in [0,\infty)$, such that the rate functions $\mathcal{I}_{\textup{cone}}^{\mathscr{H}}$ and $\mathcal{I}_{\rm beta}^{\mathscr{H}}$ remain finite. Since the computations are almost the same as for the Euclidean case, we omit the details.
\end{proof}
\begin{example}\label{exp:LambdaZeroMat} \rm
Similarly as in Example \ref{exp:LambdaZeroEucl}, if we consider the case $\vartheta =0$, we get the large deviation behavior of the empirical spectral measure of a random matrix $Z^{(n)}$ distributed according to some beta-type distribution $\Psi_{n}^{\mathscr{H}}\bU_{n,p,\beta}^{\mathscr{H}}$ analogous to Example \ref{ex:Beta}. Again, the same behavior is exhibited by a multitude of distributions $\vartheta_n\bC_{n,p,\beta}^{\mathscr{H}}+ (1-\vartheta_n)\Psi_{n}^{\mathscr{H}}\bU_{n,p,\beta}^{\mathscr{H}}$ with $\vartheta_n \to \vartheta \in [0,1)$ with the associated rate function for the sequence $(\mu_n)_{n\in\N}$ of empirical spectral measures being
\begin{align*}
\mathcal{I}_{\textup{emp}}^{\mathscr{H}}(\mu) = \begin{cases}
\parbox{11cm}{$\displaystyle  \vphantom{\int\limits_0} \mathcal{I}_{\textup{cone}}^{\mathscr{H}}(\mu)$}&:\parbox{2cm}{$ m_p(\mu)\leq 1,$\\$\alpha=0$}\\
\parbox{11cm}{$\displaystyle \mathcal{I}_{\textup{cone}}^{\mathscr{H}}(\mu) + {\beta \over 2p}\log{\beta \over 2p} - \Big({\beta\over 2p} + \alpha\Big) \log \Big({\beta\over 2p} + \alpha\Big) -\alpha \log \Big({{1-m_p(\mu)}\over \alpha}\Big)$}&:\parbox{2cm}{$m_p(\mu)< 1,$\\$\alpha\in(0,\infty)$}\\
+\vphantom{\int\limits^{1}} \infty &: otherwise.
\end{cases}
\end{align*} 
\end{example}

%-------------------------------------------------------------------------------------------------------------------------------------------------

\subsection{LDPs for the empirical spectral measure of random matrices in $\B^{n, \mathscr{M}}_{p,\beta}$} \label{subsec:LDPMatrixNSA}

If the matrix is not self-adjoint, we define the empirical spectral measure of $Z\in \Mat_n(\mathbb{F}_\beta)$ with respect to the squared singular values $s^2_1(A) \le \ldots \le  s^2_n(A)$ as $\mu_n:={1\over n}\sum_{i=1}^n\delta_{n^{2/p}s^2_i(Z)}$. Note that just as before, the coordinates of the vector $(s^2_1(Z), \ldots, s^2_n(Z)) \in \R^n_+$ are suitably scaled. In the non-self-adjoint case, we mean the  rescaled empirical spectral measure with respect to the squared singular values when we talk of the empirical spectral measure. As  in the previous section, a large deviation principle for the empirical spectral measure of a sequence of random matrices with distribution $\bU_{n,p,\beta}^{\mathscr{M}}$ or $\bC_{n,p,\beta}^{\mathscr{M}}$ on $\B_{p,\beta}^{n,{\mathscr{M}}}$ was proved in \cite{KPTSanov}. Especially, it was observed that the rate function in both cases is the same up to a constant. Slightly adapting the proof of Theorem \ref{thm:LDPEmpSpecMeasureSA}, we can show that this phenomenon occurs in a more general context. The proof is now based on Theorem \ref{thm:SvDistr} and the norm distribution outlined in Remark \ref{rmk:NormDistrNSA} instead of Theorem \ref{thm:EvDistr} and Remark \ref{rmk:NormDistrSA}, but this time also on \cite[Theorem 1.5]{KPTSanov} instead of \cite[Theorem 1.1]{KPTSanov}, the latter of which we stated as Lemma \ref{lem:LDPfromSanovPaper} above.
\begin{thm}\label{thm:LDPEmpSpecMeasureNSA}
Let $0<p<\infty$, $\beta\in\{1,2,4\}$, and let $(\vartheta_n)_{n\in\N}$ be a sequence in $[0,1]$ with $\lim\limits_{n\to\infty} \vartheta_n = \vartheta \in [0,1]$ and denote by $k(\vartheta) \ge 2$ the smallest number such that $\lim\limits_{n\to\infty} n^{-k(\vartheta)} \, |\log(1-\vartheta_n)| < +\infty$. Also, let $(\alpha_n)_{n\in\N}$ be a positive, real sequences such that $\lim\limits_{n\to\infty} \alpha_n n^{-2} =\alpha \in [0,\infty)$. For each $n\in\N$ let $\bW_n=\vartheta_n\delta_0 +(1-\vartheta_n)\textbf{G}(\alpha_n, 1)$, and let $Z^{(n)}$ be a random matrix in $\B_{p,\beta}^{n, \mathscr{M}}$ chosen according to the distribution $\bP_{n,p,\bW_n,\beta}^{\mathscr{M}}$. Then the sequence of random probability measures $\mu_n:={1\over n}\sum\limits_{i=1}^n\delta_{n^{2/p}s^2_i(Z^{(n)})}$ satisfies a large deviation principle on $\mathcal{M}(\R_+)$ with speed $n^2$ and good rate function
\begin{align*}
\mathcal{I}_{\textup{emp}}^{\mathscr{M}} (\mu) = \begin{cases}
\displaystyle \mathcal{I}_{\textup{cone}}^{\mathscr{M}} (\mu) - c_{(1-\vartheta)}^{\mathscr{M}} &: \parbox{5cm}{$m_p(\mu)\le1, k(\vartheta) \ge 2, \alpha =0$}\\
\parbox{9cm}{$ \displaystyle  \mathcal{I}_{\rm cone}^{\mathscr{M}} (\mu) +{\beta\over p}\log\Big({\beta\over p}\Big) - \Big({\beta\over p} + \alpha\Big) \log \Big({\beta\over p} + \alpha\Big) \vphantom{\int\limits_0^1}$\\$- \alpha \log \Big(\frac{1 -m_p(\mu)}{\alpha}\Big) - c_{(1-\vartheta)}^{\mathscr{M}}  \vphantom{\int\limits_0}$} &: \parbox{5cm}{$m_p(\mu)<1, k(\vartheta)=2,  \alpha>0$}\\
+\infty &: otherwise,
\end{cases}
\end{align*}
where 
$$
\mathcal{I}_{\textup{cone}}^{\mathscr{M}}(\mu) = \begin{cases}
\displaystyle -{\beta\over 2}\int_{\R}\int_\R\log|x-y|\,\mu(\dint x)\mu(\dint y)+{\beta\over p}\log\Big({\sqrt{\pi}\,p\,\Gamma({p\over 2})\over 2^p\,\sqrt{e}\,\Gamma({p+1\over 2})}\Big) &: m_{p/2}(\mu) \leq 1\\
+\infty &: otherwise,
\end{cases}
$$
and 
$$ 
c_{(1-\vartheta)}^{\mathscr{M}} := \begin{cases} 
\lim\limits_{n\to\infty} n^{-2} \, \log(1-\vartheta_n) \vphantom{\int\limits_{0}}&: k(\vartheta) = 2\\
0 &: k(\vartheta) > 2.
\end{cases}
$$
\end{thm}
The proof of Theorem \ref{thm:LDPEmpSpecMeasureNSA} is completely analogous to the one of Theorem \ref{thm:LDPEmpSpecMeasureSA}, thus we will only point out the changes in in the auxiliary results that need to be made. 
\begin{lemma}\label{lem:LDPBetaExtendedMat2}
Let $0<p<\infty$, $\beta\in\{1,2,4\}$, and let $(\vartheta_n)_{n\in\N}$ be a sequence in $[0,1]$ with $\lim\limits_{n\to\infty} \vartheta_n = \vartheta \in [0,1]$ and denote by $k(\vartheta) \ge 2$ the smallest number such that $\lim\limits_{n\to\infty} n^{-k(\vartheta)} \, |\log(1-\vartheta_n)| < +\infty$. Also, let $(\alpha_n)_{n\in\N}$ be a positive, real sequences such that $\lim\limits_{n\to\infty} \alpha_nn^{-2} =\alpha \in [0,\infty)$. For each $n\in\N$ let $X^{(n)}=(X_1^{(n)},\ldots,X_n^{(n)})$ be a random vector with density $C_{n,p/2,\nabla^c_\beta}\,e^{-\|x\|_{p/2}^{p/2}}\,\nabla^c_\beta(x)$, $x\in\R^n_+$, with $\nabla_\beta^c$ defined as in Theorem \ref{thm:SvDistr}. Independently of the sequence $(X^{(n)})_{n\in\N}$, let $(W^{(n)})_{n\in\N}$ be a sequence of random variables with $W^{(n)} \sim \bW_n= \vartheta_n\delta_0 + (1-\vartheta_n)\textbf{G}(\alpha_n,1)$. Then the sequence of random variables $(B^{(n)})_{n\in\N}$ with $B^{(n)}:={\|X^{(n)}\|_{p/2}^{p/2}\over\|X^{(n)}\|_{p/2}^{p/2}+W^{(n)}}$ satisfies a large deviation principle on $[0,\infty)$ with speed $n^2$ and good rate function
	$$
\mathcal{I}_{\rm beta}^{\mathscr{M}}(x) = \begin{cases}
0 &: k(\vartheta) > 2,  x =1 \vphantom{\int\limits_{0}} \\
%//
-{\beta\over p} \log(x) - c_{(1- \vartheta)}^{\mathscr{M}}(x)&: \parbox{3cm}{$k(\vartheta) = 2,  \alpha =0,$\\ $x \in (0,1] \vphantom{\int\limits_{0}}$} \\
%//
\parbox{10cm}{$-{\beta\over p} \log(\frac{xp}{\beta}) - \alpha \log\Big(\frac{1-x}{\alpha}\Big) - \Big({\beta\over p} + \alpha\Big) \log\Big({\beta\over p} + \alpha \Big) - c_{(1- \vartheta)}^{\mathscr{M}}(x)$} &:\parbox{3cm}{$k(\vartheta) = 2,  \alpha >0,$\\$x \in (0,1) \vphantom{\int\limits_{0}}$}\\
%\\
+\infty&: otherwise,
\end{cases}
$$
where 
$$ 
c_{(1-\vartheta)}^{\mathscr{M}} := \begin{cases} 
\lim\limits_{n\to\infty} n^{-2} \, \log(1-\vartheta_n) \vphantom{\int\limits_{0}}&: k(\vartheta) = 2\\
0 &: k(\vartheta) > 2.
\end{cases}
$$
\end{lemma}
This first lemma establishes an LDP for the beta distributed $(p/2)$-norm of the random vector  ${X^{(n)}}/\big{(}{\|X^{(n)}\|_{p/2}^{p/2}+W^{(n)}}\big{)}^{2/p}$. This is proven in the same way as Lemma \ref{lem:LDPBetaExtended-Revised}. In the non-self-adjoint case nothing changes in comparison to the self-adjoint case, besides the value for $p$ (which becomes $p/2$) and the density of the random vector $X^{(n)}$ underlying that representation. For the singular value distribution in non-self-adjoint matrix $p$-balls a different weight function $\nabla^c_\beta$ is needed with a different degree of homogeneity $m=(\beta/2)n^2-n$. This $m$ only plays a role in the first parameter of the beta distribution involved in the distribution of the $p$-norm component (see Remark \ref{rmk:NormDistrNSA}). It affects the large deviation behavior of the random variable $B^{(n)}$ only insofar as the limit of the first parameter changes from $\beta/(2p)$ to $\lim_{n\to\infty} n^{-2} (n+m)/(p/2) =\beta/p$.

The second lemma is the analogue of Lemma \ref{lem:LDPfromSanovPaper} and gives a large deviation principle for the empirical spectral measure of a non-self-adjoint random matrix in $\B_{p,\beta}^{n, \mathscr{M}}$ with distribution $\bC^{\mathscr{M}}_{n,p,\beta}$. T\label{key}his result can be found in \cite[Theorem 1.5]{KPTSanov}.

\begin{lemma}\label{lem:LDPSanovNSA}
For $n\in\N$ let $Z^{(n)}$ be a random matrix in $\B_{p,\beta}^{n, \mathscr{M}}$ with distribution $\bC_{n,p,\beta}^{\mathscr{M}}$ and singular values $s_i(Z^{(n)})$, $i\in\{1,\ldots,n\}$. Then the sequence of random probability measures
$
\mu_n := {1\over n}\sum_{i=1}^n\delta_{n^{2/p}s^2_i(Z^{(n)})}
$
satisfies a large deviation principle on $\mathcal{M}(\R_+)$ with speed $n^2$ and good rate function
$$
\mathcal{I}^{\mathscr{M}}_{\textup{cone}}(\mu) = \begin{cases}
\displaystyle -{\beta\over 2}\int_{\R_+}\int_{\R_+}\log|x-y|\,\mu(\dint x)\mu(\dint y)+{\beta\over p}\log\Big({\sqrt{\pi}\,p\,\Gamma({p\over 2})\over 2^p\,\sqrt{e}\,\Gamma({p+1\over 2})}\Big) &: m_{p/2}(\mu) \leq 1\\
+\infty &: otherwise.
\end{cases}
$$
\end{lemma}

From here on out the proof will be the completely analogous to that of Theorem \ref{thm:LDPEmpSpecMeasureSA}, with the main difference being that one uses the rate functions from Lemma \ref{lem:LDPBetaExtendedMat2} and Lemma \ref{lem:LDPSanovNSA} instead of Lemma \ref{lem:LDPBetaExtended-Revised} and Lemma \ref{lem:LDPfromSanovPaper} and the probabilistic representation from Theorem \ref{thm:SvDistr} instead of Theorem \ref{thm:EvDistr}. 

\bibliographystyle{plain}
%\bibliography{ReferencesRadialDistributions}

\end{document}